\documentclass[11pt]{preprint}
\usepackage{difftrees} 
\usepackage[full]{textcomp}
\usepackage[osf]{newtxtext} 
\usepackage[cal=boondoxo]{mathalfa}
\usepackage{colortbl}

\usepackage{tikz-cd} 
\usetikzlibrary{cd} 
\usepackage[all,cmtip]{xy}
\usepackage{comment}

\usepackage{amssymb}
\usepackage{mathtools}
\usepackage{hyperref}
\usepackage{breakurl}
\usepackage{mhenvs}
\usepackage{mhequ} 
\usepackage{mhsymb}
\usepackage{booktabs}
\usepackage{tikz}
\usepackage{tcolorbox}
\usepackage{mathrsfs}
\usepackage[utf8]{inputenc}
\usepackage{longtable}
\usepackage{wrapfig}
\usepackage{rotating} 
\usepackage{subcaption}
\usepackage{mathrsfs}
\usepackage{epsfig}
\usepackage{microtype}
\usepackage{comment}
\usepackage{wasysym}
\usepackage{centernot}
\usepackage{refcheck}

\usepackage{enumitem}
\usepackage{bm}
\usepackage{stackrel}
\usepackage{graphicx}
\usepackage{axodraw}
\usepackage{xspace}
\usepackage{subcaption} 
\usepackage{epsfig} 
\usepackage{axodraw} 
\usepackage{xspace} 
\usepackage[toc,page]{appendix} 
\usepackage[all,cmtip]{xy} 
\usepackage{relsize} 
\usepackage{shuffle} 
\makeatletter
\newcommand{\globalcolor}[1]{%
  \color{#1}\global\let\default@color\current@color
}
\makeatother

\usetikzlibrary{calc}
\usetikzlibrary{decorations}
\usetikzlibrary{positioning}
\usetikzlibrary{shapes}
\usetikzlibrary{external}

\definecolor{blush}{rgb}{0.87, 0.36, 0.51}
	\definecolor{brightcerulean}{rgb}{0.11, 0.67, 0.84}
	\definecolor{greenryb}{rgb}{0.4, 0.69, 0.2}

\newif\ifdark
\darkfalse

\ifdark
\definecolor{darkred}{rgb}{0.9,0.2,0.2}
\definecolor{darkblue}{rgb}{0.7,0.3,1}
\definecolor{darkgreen}{rgb}{0.1,0.9,0.1}
\definecolor{franck}{rgb}{0,0.8,1}
\definecolor{pagebackground}{rgb}{.15,.21,.18}
\definecolor{pageforeground}{rgb}{.84,.84,.85}
\pagecolor{pagebackground}
\AtBeginDocument{\globalcolor{pageforeground}}
\definecolor{symbols}{rgb}{0,0.7,1}
\colorlet{connection}{red!80!black}
\colorlet{boxcolor}{blue!50}

\else

\definecolor{darkred}{rgb}{0.7,0.1,0.1}
\definecolor{darkblue}{rgb}{0.4,0.1,0.8}
\definecolor{darkgreen}{rgb}{0.1,0.7,0.1}
\definecolor{franck}{rgb}{0,0,1}
\definecolor{pagebackground}{rgb}{1,1,1}
\definecolor{pageforeground}{rgb}{0,0,0}
\colorlet{symbols}{blue!90!black}
\colorlet{connection}{red!30!black}
\colorlet{boxcolor}{blue!50!black}

\fi

\def\slash{\leavevmode\unskip\kern0.18em/\penalty\exhyphenpenalty\kern0.18em}
\def\dash{\leavevmode\unskip\kern0.18em--\penalty\exhyphenpenalty\kern0.18em}

\DeclareMathAlphabet{\mathbbm}{U}{bbm}{m}{n}

\DeclareFontFamily{U}{BOONDOX-calo}{\skewchar\font=45 }
\DeclareFontShape{U}{BOONDOX-calo}{m}{n}{
  <-> s*[1.05] BOONDOX-r-calo}{}
\DeclareFontShape{U}{BOONDOX-calo}{b}{n}{
  <-> s*[1.05] BOONDOX-b-calo}{}
\DeclareMathAlphabet{\mcb}{U}{BOONDOX-calo}{m}{n}
\SetMathAlphabet{\mcb}{bold}{U}{BOONDOX-calo}{b}{n}

\setlist{noitemsep,topsep=4pt,leftmargin=1.5em}

\DeclareMathAlphabet{\mathbbm}{U}{bbm}{m}{n}

\DeclareMathAlphabet{\mcb}{U}{BOONDOX-calo}{m}{n}
\SetMathAlphabet{\mcb}{bold}{U}{BOONDOX-calo}{b}{n}
\DeclareFontFamily{U}{mathx}{\hyphenchar\font45}
\DeclareFontShape{U}{mathx}{m}{n}{
      <5> <6> <7> <8> <9> <10>
      <10.95> <12> <14.4> <17.28> <20.74> <24.88>
      mathx10
      }{}
\DeclareSymbolFont{mathx}{U}{mathx}{m}{n}
\DeclareMathSymbol{\bigtimes}{1}{mathx}{"91}

\providecommand{\figures}{false}
{ \ifthenelse{\equal{\figures}{false}} {#1}{\[ {\rm Figure \ missing !} \]} }{}
\def\id{\mathrm{id}}

\def\CH{\mathcal{H}}

\def\CG{\mathcal{G}}

\def\CA{\mathcal{A}}

\def\CT{\mathcal{T}}

\tikzstyle{tinydots}=[dash pattern=on \pgflinewidth off \pgflinewidth]
\tikzstyle{superdense}=[dash pattern=on 4pt off 1pt]

\newcommand{\mcT}{\mathcal{T}}

\newcommand{\beq}{\begin{equation}}
\newcommand{\eeq}{\end{equation}}

\usepackage{empheq}

\newcommand{\mfe}{\mathfrak{e}}

\newcommand{\mfL}{\mathfrak{L}}

\newcommand{\mft}{\mathfrak{t}}

\newcommand{\mfp}{\mathfrak{p}}

\newcommand{\mff}{\mathfrak{f}}

\def\Labp{\mathfrak{p}}

\def\Labhom{\mathfrak{t}}

\def\Lab{\mathfrak{L}}

\def\${|\!|\!|}

\newcounter{theorem}

\newtheorem{ex}[theorem]{Example}

\newenvironment{DIFnomarkup}{}{} 

\theorembodyfont{\rmfamily}

\newcommand{\rrightarrow}{{\to\hskip -4.9mm\raise 1pt\hbox{$\to$}}}

\newfont{\indic}{bbmss12}

\def\Nabla_#1{\nabla_{\!#1}}

\makeatletter
\pgfdeclareshape{crosscircle}
{
  \inheritsavedanchors[from=circle] 
  \inheritanchorborder[from=circle]
  \inheritanchor[from=circle]{north}
  \inheritanchor[from=circle]{north west}
  \inheritanchor[from=circle]{north east}
  \inheritanchor[from=circle]{center}
  \inheritanchor[from=circle]{west}
  \inheritanchor[from=circle]{east}
  \inheritanchor[from=circle]{mid}
  \inheritanchor[from=circle]{mid west}
  \inheritanchor[from=circle]{mid east}
  \inheritanchor[from=circle]{base}
  \inheritanchor[from=circle]{base west}
  \inheritanchor[from=circle]{base east}
  \inheritanchor[from=circle]{south}
  \inheritanchor[from=circle]{south west}
  \inheritanchor[from=circle]{south east}
  \inheritbackgroundpath[from=circle]
  \foregroundpath{
    \centerpoint%
    \pgf@xc=\pgf@x%
    \pgf@yc=\pgf@y%
    \pgfutil@tempdima=\radius%
    \pgfmathsetlength{\pgf@xb}{\pgfkeysvalueof{/pgf/outer xsep}}%
    \pgfmathsetlength{\pgf@yb}{\pgfkeysvalueof{/pgf/outer ysep}}%
    \ifdim\pgf@xb<\pgf@yb%
      \advance\pgfutil@tempdima by-\pgf@yb%
    \else%
      \advance\pgfutil@tempdima by-\pgf@xb%
    \fi%
    \pgfpathmoveto{\pgfpointadd{\pgfqpoint{\pgf@xc}{\pgf@yc}}{\pgfqpoint{-0.707107\pgfutil@tempdima}{-0.707107\pgfutil@tempdima}}}
    \pgfpathlineto{\pgfpointadd{\pgfqpoint{\pgf@xc}{\pgf@yc}}{\pgfqpoint{0.707107\pgfutil@tempdima}{0.707107\pgfutil@tempdima}}}
    \pgfpathmoveto{\pgfpointadd{\pgfqpoint{\pgf@xc}{\pgf@yc}}{\pgfqpoint{-0.707107\pgfutil@tempdima}{0.707107\pgfutil@tempdima}}}
    \pgfpathlineto{\pgfpointadd{\pgfqpoint{\pgf@xc}{\pgf@yc}}{\pgfqpoint{0.707107\pgfutil@tempdima}{-0.707107\pgfutil@tempdima}}}
  }
}
\makeatother

\def\symbol#1{\textcolor{symbols}{#1}}

\def\decorate#1#2{
        \ifnum#2>0
    		\foreach \count in {1,...,#2}{
	       	let
				\p1 = (sourcenode.center),
                \p2 = (sourcenode.east),
				\n1 = {\x2-\x1},
				\n2 = {1mm},
				\n3 = {(1.3+0.6*(\count-1))*\n1},
				\n4 = {0.7*\n1}
			in 
        		node[rectangle,fill=symbols,rotate=30,inner sep=0pt,minimum width=0.2*\n2,minimum height=\n2] at ($(sourcenode.center) + (\n3,\n4)$) {}
				}
		\fi
        \ifnum#1>0
    		\foreach \count in {1,...,#1}{
	       	let
				\p1 = (sourcenode.center),
                \p2 = (sourcenode.east),
				\n1 = {\x2-\x1},
				\n2 = {1mm},
				\n3 = {(1.3+0.6*(\count-1))*\n1},
				\n4 = {0.7*\n1}
			in 
        		node[rectangle,fill=symbols,rotate=-30,inner sep=0pt,minimum width=0.2*\n2,minimum height=\n2] at ($(sourcenode.center) + (-\n3,\n4)$) {}
				}
		\fi
}

\tikzset{
    dectriangle/.style 2 args={
        triangle,
        alias=sourcenode,
        append after command={\decorate{#1}{#2}}
    },
    dectriangle/.default={0}{0},
}

\tikzset{
	cross/.style={path picture={ 
  		\draw[symbols]
			(path picture bounding box.south east) -- (path picture bounding box.north west) (path picture bounding box.south west) -- (path picture bounding box.north east);
		}},
root/.style={circle,fill=green!50!black,inner sep=0pt, minimum size=1.2mm},
        dot/.style={circle,fill=pageforeground,inner sep=0pt, minimum size=1mm},
        dotred/.style={circle,fill=pageforeground!50!pagebackground,inner sep=0pt, minimum size=2mm},
        var/.style={circle,fill=pageforeground!10!pagebackground,draw=pageforeground,inner sep=0pt, minimum size=3mm},
        kernel/.style={semithick,shorten >=2pt,shorten <=2pt},
        kernels/.style={snake=zigzag,shorten >=2pt,shorten <=2pt,segment amplitude=1pt,segment length=4pt,line before snake=2pt,line after snake=5pt,},
        rho/.style={densely dashed,semithick,shorten >=2pt,shorten <=2pt},
           testfcn/.style={dotted,semithick,shorten >=2pt,shorten <=2pt},
        renorm/.style={shape=circle,fill=pagebackground,inner sep=1pt},
        labl/.style={shape=rectangle,fill=pagebackground,inner sep=1pt},
        xic/.style={very thin,circle,draw=symbols,fill=symbols,inner sep=0pt,minimum size=1.2mm},
        g/.style={very thin,rectangle,draw=symbols,fill=symbols!10!pagebackground,inner sep=0pt,minimum width=2.5mm,minimum height=1.2mm},
        xi/.style={very thin,circle,draw=symbols,fill=symbols!10!pagebackground,inner sep=0pt,minimum size=1.2mm},
	xies/.style={very thin,rectangle,fill=green!50!black!25,draw=symbols,inner sep=0pt,minimum size=1.1mm},
	xiesf/.style={very thin,rectangle,fill=green!50!black,draw=symbols,inner sep=0pt,minimum size=1.1mm},
        xix/.style={very thin,crosscircle,fill=symbols!10!pagebackground,draw=symbols,inner sep=0pt,minimum size=1.2mm},
        X/.style={very thin,cross,rectangle,fill=pagebackground,draw=symbols,inner sep=0pt,minimum size=1.2mm},
	xib/.style={thin,circle,fill=symbols!10!pagebackground,draw=symbols,inner sep=0pt,minimum size=1.6mm},
	xie/.style={thin,circle,fill=green!50!black,draw=symbols,inner sep=0pt,minimum size=1.6mm},
	xid/.style={thin,circle,fill=symbols,draw=symbols,inner sep=0pt,minimum size=1.6mm},
	xibx/.style={thin,crosscircle,fill=symbols!10!pagebackground,draw=symbols,inner sep=0pt,minimum size=1.6mm},
	kernels2/.style={very thick,draw=connection,segment length=12pt},
	keps/.style={thin,draw=symbols,->},
	kepspr/.style={thick,draw=connection,->},
	krho/.style={thin,draw=symbols,superdense,->},
	krhopr/.style={thick,draw=connection,superdense},
	triangle/.style = { regular polygon, regular polygon sides=3},
	not/.style={thin,circle,draw=connection,fill=connection,inner sep=0pt,minimum size=0.5mm},
	diff/.style = {very thin,draw=symbols,triangle,fill=red!50!black,inner sep=0pt,minimum size=1.6mm},
	diff1/.style = {very thin,dectriangle={1}{0},fill=red!50!black,draw=symbols,inner sep=0pt,minimum size=1.6mm},
	diff2/.style = {very thin,dectriangle={1}{1},fill=red!50!black,draw=symbols,inner sep=0pt,minimum size=1.6mm},
		diffmini/.style = {very thin,rectangle,fill=black,draw=black,inner sep=0pt,minimum size=0.75mm},
	 kernelsmod/.style={very thick,draw=connection,segment length=12pt},
	 rec/.style = {very thin,rectangle,fill=black,draw=black,inner sep=0pt,minimum size=2mm},
	cerc/.style={very thin,circle,draw=black,fill=symbols,inner sep=0pt,minimum size=2mm},
	stars/.style={very thin,star,star points=6,star point ratio=0.5, draw=black,fill=red,inner sep=0pt,minimum size=0.7mm},
	>=stealth,
        }
        \tikzset{
root/.style={circle,fill=black!50,inner sep=0pt, minimum size=3mm},
        circ/.style={circle,fill=white,draw=black,very thin,inner sep=.5pt, minimum size=1.2mm},
        round1/.style={fill=white,outer sep = 0,inner sep=2pt,rounded corners=1mm,draw,text=black,thin,minimum size=1.2mm},
          circ1/.style={circle,fill=red!10,draw=red,very thin,inner sep=.5pt, minimum size=1.2mm},
        rect/.style={fill=white,outer sep = 0,inner sep=2pt,rectangle,draw,text=black,thin,minimum size=1.2mm},
        rect1/.style={fill=white,outer sep = 0,inner sep=2pt,rectangle,draw,text=black,thin,minimum size=1.2mm},
        round2/.style={fill=red!10,outer sep = 0,inner sep=2pt,rounded corners=1mm,draw,text=black,thin,minimum size=1.2mm},
       round3/.style={fill=blue!10,outer sep = 0,inner sep=2pt,rounded corners=1mm,draw,text=black,thin,minimum size=1.2mm}, 
        rect2/.style={fill=black!10,outer sep = 0,inner sep=2pt,rectangle,draw,text=black,thin,minimum size=1.2mm},
        dot/.style={circle,fill=black,inner sep=0pt, minimum size=1.2mm},
        dotred/.style={circle,fill=black!50,inner sep=0pt, minimum size=2mm},
        var/.style={circle,fill=black!10,draw=black,inner sep=0pt, minimum size=3mm},
        kernel/.style={semithick,shorten >=2pt,shorten <=2pt},
         diag/.style={thin,shorten >=4pt,shorten <=4pt},
        kernel1/.style={thick},
        kernels/.style={snake=zigzag,shorten >=2pt,shorten <=2pt,segment amplitude=1pt,segment length=4pt,line before snake=2pt,line after snake=5pt,},
		kernels1/.style={snake=zigzag,segment amplitude=0.5pt,segment length=2pt},
		rho1/.style={densely dotted,semithick},
        rho/.style={densely dashed,semithick,shorten >=2pt,shorten <=2pt},
           testfcn/.style={dotted,semithick,shorten >=2pt,shorten <=2pt},
           visible/.style={draw, circle, fill, inner sep=0.25ex},
        renorm/.style={shape=circle,fill=white,inner sep=1pt},
        labl/.style={shape=rectangle,fill=white,inner sep=1pt},
        xic/.style={very thin,circle,fill=symbols,draw=black,inner sep=0pt,minimum size=1.2mm},
        xi/.style={very thin,circle,fill=blue!10,draw=black,inner sep=0pt,minimum size=1.2mm},
	xib/.style={very thin,circle,fill=blue!10,draw=black,inner sep=0pt,minimum size=1.6mm},
	xie/.style={very thin,circle,fill=green!50!black,draw=black,inner sep=0pt,minimum size=1mm},
	xid/.style={very thin,circle,fill=symbols,draw=black,inner sep=0pt,minimum size=1.6mm},
	edgetype/.style={very thin,circle,draw=black,inner sep=0pt,minimum size=5mm},
	nodetype/.style={very thick,circle,draw=black,inner sep=0pt,minimum size=5mm},
	kernels2/.style={very thick,draw=connection,segment length=12pt},
clean/.style={thin,circle,fill=black,inner sep=0pt,minimum size=1mm},	not/.style={thin,circle,fill=symbols,draw=connection,fill=connection,inner sep=0pt,minimum size=0.8mm},
	>=stealth,
        }

\makeatletter
\def\DeclareSymbol#1#2#3{%
	\expandafter\gdef\csname MH@symb@#1\endcsname{\tikzsetnextfilename{symbol#1}%
	\tikz[baseline=#2,scale=0.15,draw=symbols,line join=round]{#3}}%
	\expandafter\gdef\csname MH@symb@#1s\endcsname{\scalebox{0.75}{\tikzsetnextfilename{symbol#1}%
	\tikz[baseline=#2,scale=0.15,draw=symbols,line join=round]{#3}}}%
	\expandafter\gdef\csname MH@symb@#1ss\endcsname{\scalebox{0.65}{\tikzsetnextfilename{symbol#1}%
	\tikz[baseline=#2,scale=0.15,draw=symbols,line join=round]{#3}}}%
	}
\def\<#1>{\ifthenelse{\boolean{mmode}}{\mathchoice{\csname MH@symb@#1\endcsname}{\csname MH@symb@#1\endcsname}{\csname MH@symb@#1s\endcsname}{\csname MH@symb@#1ss\endcsname}}{\csname MH@symb@#1\endcsname}}
\makeatother

\DeclareSymbol{Xi22}{0.5}{\draw (0,0) node[xi] {} -- (-1,1) node[not] {} -- (0,2) node[xi] {};} 
\DeclareSymbol{Xi2}{-2}{\draw (-1,-0.25) node[xi] {} -- (0,1) node[xi] {};} 
\DeclareSymbol{Xi2b}{-2}{\draw (-1,-0.25) node[xic] {} -- (0,1) node[xic] {};} 
\DeclareSymbol{Xi2g}{-2}{\draw (-1,-0.25) node[xies] {} -- (0,1) node[xi] {};} 
\DeclareSymbol{Xi2g2}{-2}{\draw (-1,-0.25) node[xi] {} -- (0,1) node[xies] {};} 
\DeclareSymbol{cXi2}{-2}{\draw (0,-0.25) node[xi] {} -- (-1,1) node[xic] {};}
\DeclareSymbol{Xi3}{0}{\draw (0,0) node[xi] {} -- (-1,1) node[xi] {} -- (0,2) node[xi] {};}
\DeclareSymbol{XiIIXi}{0}{\draw (0,0) node[xi] {} -- (-1,1); \draw[kernels2] (-1,1) node[not] {} -- (0,2) node[xi] {};}

\DeclareSymbol{Xi4}{2}{\draw (0,0) node[xi] {} -- (-1,1) node[xi] {} -- (0,2) node[xi] {} -- (-1,3) node[xi] {};}
\DeclareSymbol{Xi4_1}{2}{\draw (0,0) node[xic] {} -- (-1,1) node[xic] {} -- (0,2) node[xi] {} -- (-1,3) node[xi] {};}
\DeclareSymbol{Xi4_2}{2}{\draw (0,0) node[xic] {} -- (-1,1) node[xi] {} -- (0,2) node[xi] {} -- (-1,3) node[xic] {};}
\DeclareSymbol{Xi2X}{-2}{\draw (0,-0.25) node[xi] {} -- (-1,1) node[xix] {};}
\DeclareSymbol{XXi2}{-2}{\draw (0,-0.25) node[xix] {} -- (-1,1) node[xi] {};}
\DeclareSymbol{IIXi}{0}{\draw (0,-0.25) node[not] {} -- (-1,1) node[xi] {} -- (0,2) node[xi] {};}
\DeclareSymbol{IXi^2}{-1}{\draw (-1,1) node[xi] {} -- (0,0) node[not] {} -- (1,1) node[xi] {};}
\DeclareSymbol{IIXi^2}{-4}{\draw (0,-1.5) node[not] {} -- (0,0);
\draw[kernels2] (-1,1) node[xi] {} -- (0,0) node[not] {} -- (1,1) node[xi] {};}
\DeclareSymbol{XiX}{-2.8}{\node[xibx] {};}
\DeclareSymbol{tauX}{-2.8}{ \node[X] {};}
\DeclareSymbol{Xi}{-2.8}{\node[xib] {};}

\DeclareSymbol{IXiX}{-1}{\draw (0,-0.25) node[not] {} -- (-1,1) node[xix] {};}
\DeclareSymbol{IXi3}{2}{\draw (0,-0.25) node[not] {} -- (-1,1) node[xi] {} -- (0,2) node[xi] {} -- (-1,3) node[xi] {};}
\DeclareSymbol{IXi}{-2}{\draw (0,-0.25) node[not] {} -- (-1,1) node[xi] {};}
\DeclareSymbol{XiI}{-2}{\draw (0,-0.25) node[xi] {} -- (-1,1) node[not] {};}

\DeclareSymbol{Xi4b}{0}{\draw(0,1.5) node[xi] {} -- (0,0); \draw (-1,1) node[xi] {} -- (0,0) node[xi] {} -- (1,1) node[xi] {};}
\DeclareSymbol{Xi4b'}{0}{\draw(0,1.5) node[xi] {} -- (0,-0.2); \draw (-1,1) node[xi] {} -- (0,-0.2) node[not] {} -- (1,1) node[xi] {};}
\DeclareSymbol{Xi4c}{0}{\draw (0,1) -- (0.8,2.2) node[xi] {};\draw (0,-0.25) node[xi] {} -- (0,1) node[xi] {} -- (-0.8,2.2) node[xi] {};}
\DeclareSymbol{Xi4d}{-4.5}{\draw (0,-1.5) node[not] {} -- (0,0); \draw (-1,1) node[xi] {} -- (0,0) node[xi] {} -- (1,1) node[xi] {};}
\DeclareSymbol{Xi4e}{0}{\draw (0,2) node[xi] {} -- (-1,1) node[xi] {} -- (0,0) node[xi] {} -- (1,1) node[xi] {};}
\DeclareSymbol{Xi4e'}{0}{\draw (0,2) node[xi] {} -- (-1,1) node[xi] {} -- (0,-0.2) node[not] {} -- (1,1) node[xi] {};}

\DeclareSymbol{Xitwo}
{0}{\draw[kernels2] (0,0) node[not] {} -- (-1,1) node[not] {}
-- (-2,2) node[not]{} -- (-3,3) node[xi]  {};
\draw[kernels2] (0,0) -- (1,1) node[xi] {};
\draw[kernels2] (-1,1) -- (0,2) node[xi] {};
\draw[kernels2] (-2,2) -- (-1,3) node[xi] {};}

\DeclareSymbol{IXitwo}
{0}{\draw (-.7,1.2) node[xi] {} -- (0,-0.2) -- (.7,1.2) node[xi] {};}
\DeclareSymbol{I1Xitwo}
{0}{\draw[kernels2] (0,0) node[not] {} -- (-1,1) node[xi] {};
\draw[kernels2] (0,0) -- (1,1) node[xi] {};}

\DeclareSymbol{I1Xitwobis}
{0}{\draw[kernels2] (0,0) node[not] {} -- (-1,1) node[xies] {};
\draw[kernels2] (0,0) -- (1,1) node[xies] {};}

\DeclareSymbol{I1Xitwog}
{0}{\draw[kernels2] (0,0) node[not] {} -- (-1,1) node[xies] {};
\draw[kernels2] (0,0) -- (1,1) node[xi] {};}

\DeclareSymbol{cI1Xitwo}
{0}{\draw[kernels2] (0,0) node[not] {} -- (-1,1) node[xic] {};
\draw[kernels2] (0,0) -- (1,1) node[xi] {};}

\DeclareSymbol{I1IXi3}{0}{\draw (0,0) node[xi] {} -- (-1,1) ; 
\draw[kernels2] (-1,1) node[not] {} -- (0,2) node[xi] {};
\draw[kernels2] (-1,1) node[not] {} -- (-2,2) node[xi] {};}

\DeclareSymbol{I1Xi3c}{-1}{\draw[kernels2](0,1.5) node[xi] {} -- (0,0) node[not] {}; \draw (-1,1) node[xi] {} -- (0,0) ; \draw[kernels2] (0,0) -- (1,1) node[xi] {};}

\DeclareSymbol{I1Xi3cbis}{-1}{\draw[kernels2](0,1.5) node[xies] {} -- (0,0) node[not] {}; \draw (-1,1) node[xies] {} -- (0,0) ; \draw[kernels2] (0,0) -- (1,1) node[xies] {};}

\DeclareSymbol{I1IXi3b}{0}{\draw[kernels2] (0,0) node[not] {} -- (-1,1) ; \draw[kernels2] (0,0)   -- (1,1) node[xi] {} ;
\draw (-1,1) node[xi] {} -- (0,2) node[xi] {};
}

\DeclareSymbol{I1IXi3c}{0}{\draw[kernels2] (0,0) node[not] {} -- (-1,1) ; \draw[kernels2] (0,0)   -- (1,1) node[xi] {} ;
\draw[kernels2] (-1,1) node[not] {} -- (0,2) node[xi] {};
\draw[kernels2] (-1,1) node[not] {} -- (-2,2) node[xi] {};}

\DeclareSymbol{I1IXi3cbis}{0}{\draw[kernels2] (0,0) node[not] {} -- (-1,1) ; \draw[kernels2] (0,0)   -- (1,1) node[xies] {} ;
\draw[kernels2] (-1,1) node[not] {} -- (0,2) node[xies] {};
\draw[kernels2] (-1,1) node[not] {} -- (-2,2) node[xies] {};}

\DeclareSymbol{I1Xi}{0}{\draw[kernels2] (0,0) node[not] {} -- (-1,1)  node[xi] {} ;}

\DeclareSymbol{I1Xi4a}{2}{\draw[kernels2] (0,0) node[not] {} -- (-1,1) ; \draw[kernels2] (0,0) node[not] {} -- (1,1) node[xi] {} ;
\draw (-1,1) node[xi] {} -- (0,2) node[xi] {} -- (-1,3) node[xi] {};}

\DeclareSymbol{cI1Xi4a}{2}{\draw[kernels2] (0,0) node[not] {} -- (-1,1) ; \draw[kernels2] (0,0) node[not] {} -- (1,1) node[xic] {} ;
\draw (-1,1) node[xic] {} -- (0,2) node[xi] {} -- (-1,3) node[xi] {};}

\DeclareSymbol{I1Xi4b}{2}{\draw (0,0) node[xi] {} -- (-1,1) node[xi] {} -- (0,2) ; \draw[kernels2] (0,2) node[not] {} -- (-1,3) node[xi] {};\draw[kernels2] (0,2)  -- (1,3) node[xi] {};
}

\DeclareSymbol{cI1Xi4b}{2}{\draw (0,0) node[xic] {} -- (-1,1) node[xic] {} -- (0,2) ; \draw[kernels2] (0,2) node[not] {} -- (-1,3) node[xi] {};\draw[kernels2] (0,2)  -- (1,3) node[xi] {};
}

\DeclareSymbol{I1Xi4c}{2}{\draw (0,0) node[xi] {} -- (-1,1) node[not] {}; \draw[kernels2] (-1,1) -- (0,2) ; 
\draw[kernels2] (-1,1) -- (-2,2) node[xi] {} ;
\draw (0,2) node[xi] {} -- (-1,3) node[xi] {};}

\DeclareSymbol{cI1Xi4c}{2}{\draw (0,0) node[xic] {} -- (-1,1) node[not] {}; \draw[kernels2] (-1,1) -- (0,2) ; 
\draw[kernels2] (-1,1) -- (-2,2) node[xic] {} ;
\draw (0,2) node[xi] {} -- (-1,3) node[xi] {};}

\DeclareSymbol{I1Xi4ab}{2}{\draw[kernels2] (0,0) node[not] {} -- (-1,1) ; \draw[kernels2] (0,0) node[not] {} -- (1,1) node[xi] {};\draw (-1,1) node[xi] {} -- (0,2) ; \draw[kernels2] (0,2) node[not] {} -- (-1,3) node[xi] {};\draw[kernels2] (0,2)  -- (1,3) node[xi] {}; }

\DeclareSymbol{cI1Xi4ab}{2}{\draw[kernels2] (0,0) node[not] {} -- (-1,1) ; \draw[kernels2] (0,0) node[not] {} -- (1,1) node[xic] {};\draw (-1,1) node[xic] {} -- (0,2) ; \draw[kernels2] (0,2) node[not] {} -- (-1,3) node[xi] {};\draw[kernels2] (0,2)  -- (1,3) node[xi] {}; }

\DeclareSymbol{I1Xi4bc}{2}{\draw (0,0) node[xi] {} -- (-1,1) node[not] {}; \draw[kernels2] (-1,1) -- (0,2) ; 
\draw[kernels2] (-1,1) -- (-2,2) node[xi] {} ; \draw[kernels2] (0,2) node[not] {} -- (-1,3) node[xi] {};\draw[kernels2] (0,2)  -- (1,3) node[xi] {};
}

\DeclareSymbol{cI1Xi4bc}{2}{\draw (0,0) node[xic] {} -- (-1,1) node[not] {}; \draw[kernels2] (-1,1) -- (0,2) ; 
\draw[kernels2] (-1,1) -- (-2,2) node[xic] {} ; \draw[kernels2] (0,2) node[not] {} -- (-1,3) node[xi] {};\draw[kernels2] (0,2)  -- (1,3) node[xi] {};
}

\DeclareSymbol{I1Xi4abcc1}{2}{\draw[kernels2] (0,0) node[not] {} -- (-1,1) node[not] {}
-- (-2,2) node[not]{} -- (-3,3) node[xic]  {};
\draw[kernels2] (0,0) -- (1,1) node[xic] {};
\draw[kernels2] (-1,1) -- (0,2) node[xi] {};
\draw[kernels2] (-2,2) -- (-1,3) node[xi] {};
}

\DeclareSymbol{I1Xi4abcc1b}{2}{\draw[kernels2] (0,0) node[not] {} -- (-1,1) node[not] {}
-- (-2,2) node[not]{} -- (-3,3) node[xi]  {};
\draw[kernels2] (0,0) -- (1,1) node[xic] {};
\draw[kernels2] (-1,1) -- (0,2) node[xic] {};
\draw[kernels2] (-2,2) -- (-1,3) node[xi] {};
}

\DeclareSymbol{I1Xi4abcc2}{2}{\draw[kernels2] (0,0) node[not] {} -- (-1,1) node[not] {}
-- (-2,2) node[not]{} -- (-3,3) node[xic]  {};
\draw[kernels2] (0,0) -- (1,1) node[xi] {};
\draw[kernels2] (-1,1) -- (0,2) node[xi] {};
\draw[kernels2] (-2,2) -- (-1,3) node[xic] {};
}

\DeclareSymbol{I1Xi4ac}{2}{\draw[kernels2] (0,0) node[not] {} -- (-1,1) ; \draw[kernels2] (0,0) node[not] {} -- (1,1) node[xi] {}; 
\draw[kernels2] (-1,1) node[not] {} -- (0,2) ; 
\draw[kernels2] (-1,1) -- (-2,2) node[xi] {} ;
\draw (0,2) node[xi] {} -- (-1,3) node[xi] {};}

\DeclareSymbol{cI1Xi4ac}{2}{\draw[kernels2] (0,0) node[not] {} -- (-1,1) ; \draw[kernels2] (0,0) node[not] {} -- (1,1) node[xic] {}; 
\draw[kernels2] (-1,1) node[not] {} -- (0,2) ; 
\draw[kernels2] (-1,1) -- (-2,2) node[xic] {} ;
\draw (0,2) node[xi] {} -- (-1,3) node[xi] {};}

\DeclareSymbol{I1Xi4acc1}{2}{\draw[kernels2] (0,0) node[not] {} -- (-1,1) ; \draw[kernels2] (0,0) node[not] {} -- (1,1) node[xic] {}; 
\draw[kernels2] (-1,1) node[not] {} -- (0,2) ; 
\draw[kernels2] (-1,1) -- (-2,2) node[xi] {} ;
\draw (0,2) node[xic] {} -- (-1,3) node[xi] {};}

\DeclareSymbol{I1Xi4acc2}{2}{\draw[kernels2] (0,0) node[not] {} -- (-1,1) ; \draw[kernels2] (0,0) node[not] {} -- (1,1) node[xic] {}; 
\draw[kernels2] (-1,1) node[not] {} -- (0,2) ; 
\draw[kernels2] (-1,1) -- (-2,2) node[xi] {} ;
\draw (0,2) node[xi] {} -- (-1,3) node[xic] {};}

\DeclareSymbol{2I1Xi4}{2}{\draw[kernels2] (0,0) node[not] {} -- (-1,1) node[not] {};
\draw[kernels2] (0,0) -- (1,1) node[not] {};
\draw[kernels2] (-1,1) -- (-1.5,2.5) node[xi] {};
\draw[kernels2] (-1,1) -- (-0.5,2.5) node[xi] {};
\draw[kernels2] (1,1) -- (0.5,2.5) node[xi] {};
\draw[kernels2] (1,1) -- (1.5,2.5) node[xi] {};
}

\DeclareSymbol{2I1Xi4dis}{2}{\draw[kernels2] (0,0) node[not] {} -- (-1,1) node[not] {};
\draw[kernels2] (0,0) -- (1,1) node[not] {};
\draw[kernels2] (-1,1) -- (-1.5,2.5) node[xies] {};
\draw[kernels2] (-1,1) -- (-0.5,2.5) node[xies] {};
\draw[kernels2] (1,1) -- (0.5,2.5) node[xies] {};
\draw[kernels2] (1,1) -- (1.5,2.5) node[xies] {};
}

\DeclareSymbol{2I1Xi4c1}{2}{\draw[kernels2] (0,0) node[not] {} -- (-1,1) node[not] {};
\draw[kernels2] (0,0) -- (1,1) node[not] {};
\draw[kernels2] (-1,1) -- (-1.5,2.5) node[xic] {};
\draw[kernels2] (-1,1) -- (-0.5,2.5) node[xi] {};
\draw[kernels2] (1,1) -- (0.5,2.5) node[xic] {};
\draw[kernels2] (1,1) -- (1.5,2.5) node[xi] {};
}

\DeclareSymbol{2I1Xi4c2}{2}{\draw[kernels2] (0,0) node[not] {} -- (-1,1) node[not] {};
\draw[kernels2] (0,0) -- (1,1) node[not] {};
\draw[kernels2] (-1,1) -- (-1.5,2.5) node[xic] {};
\draw[kernels2] (-1,1) -- (-0.5,2.5) node[xic] {};
\draw[kernels2] (1,1) -- (0.5,2.5) node[xi] {};
\draw[kernels2] (1,1) -- (1.5,2.5) node[xi] {};
}

\DeclareSymbol{2I1Xi4b}{2}{\draw[kernels2] (0,0) node[not] {} -- (-1,1) ;
\draw[kernels2] (0,0) -- (1,1);
\draw (-1,1) node[xi] {} -- (-1,2.5) node[xi] {};
\draw (1,1)  node[xi] {} -- (1,2.5) node[xi] {};
}

\DeclareSymbol{2I1Xi4bb}{2}{\draw[kernels2] (0,0) node[not] {} -- (-1,1) ;
\draw[kernels2] (0,0) -- (1,1);
\draw (-1,1) node[xi] {} -- (-1,2.5) node[xiesf] {};
\draw (1,1)  node[xi] {} -- (1,2.5) node[xic] {};
}

\DeclareSymbol{2I1Xi4c}{2}{\draw[kernels2] (0,0) node[not] {} -- (-1,1);
\draw[kernels2] (0,0) -- (1,1) node[not] {};
\draw (-1,1)  node[xi] {} -- (-1,2.5) node[xi] {};
\draw[kernels2] (1,1) -- (0.4,2.5) node[xi] {};
\draw[kernels2] (1,1) -- (1.6,2.5) node[xi] {};
}

\DeclareSymbol{2I1Xi4cc1}{2}{\draw[kernels2] (0,0) node[not] {} -- (-1,1);
\draw[kernels2] (0,0) -- (1,1) node[not] {};
\draw (-1,1)  node[xic] {} -- (-1,2.5) node[xi] {};
\draw[kernels2] (1,1) -- (0.4,2.5) node[xic] {};
\draw[kernels2] (1,1) -- (1.6,2.5) node[xi] {};
}

\DeclareSymbol{2I1Xi4cc2}{2}{\draw[kernels2] (0,0) node[not] {} -- (-1,1);
\draw[kernels2] (0,0) -- (1,1) node[not] {};
\draw (-1,1)  node[xic] {} -- (-1,2.5) node[xic] {};
\draw[kernels2] (1,1) -- (0.4,2.5) node[xi] {};
\draw[kernels2] (1,1) -- (1.6,2.5) node[xi] {};
}

\DeclareSymbol{Xi4ba}{0}{\draw(-0.5,1.5) node[xi] {} -- (0,0); \draw (-1.5,1) node[xi] {} -- (0,0) node[not] {}; \draw[kernels2] (0,0) -- (1.5,1) node[xi] {};
\draw[kernels2] (0,0) -- (0.5,1.5) node[xi] {} ;}

\DeclareSymbol{Xi4badis}{0}{\draw(-0.5,1.5) node[xies] {} -- (0,0); \draw (-1.5,1) node[xies] {} -- (0,0) node[not] {}; \draw[kernels2] (0,0) -- (1.5,1) node[xies] {};
\draw[kernels2] (0,0) -- (0.5,1.5) node[xies] {} ;}

\DeclareSymbol{Xi4ba1}{0}{\draw(-0.5,1.5) node[xi] {} -- (0,0); \draw (-1.5,1) node[xi] {} -- (0,0) node[not] {}; \draw[kernels2] (0,0) -- (1.5,1) node[xic] {};
\draw[kernels2] (0,0) -- (0.5,1.5) node[xic] {} ;}

\DeclareSymbol{Xi4ba1b}{0}{\draw(-0.5,1.5) node[xic] {} -- (0,0); \draw (-1.5,1) node[xic] {} -- (0,0) node[not] {}; \draw[kernels2] (0,0) -- (1.5,1) node[xi] {};
\draw[kernels2] (0,0) -- (0.5,1.5) node[xi] {} ;}

\DeclareSymbol{Xi4ba1bdiff}{0}{\draw(-0.5,1.5) node[xic] {} -- (0,0); \draw (-1.5,1) node[xic] {} -- (0,0) node[not] {}; \draw (0,0) -- (1.5,1) node[xi] {};
\draw (0,0) -- (0.5,1.5) node[xi] {};
\draw(0,0) node[diff] {};}

\DeclareSymbol{Xi4ba1bb}{0}{\draw(-0.5,1.5) node[xic] {} -- (0,0); \draw (-1.5,1) node[xiesf] {} -- (0,0) node[not] {}; \draw[kernels2] (0,0) -- (1.5,1) node[xi] {};
\draw[kernels2] (0,0) -- (0.5,1.5) node[xi] {} ;}

\DeclareSymbol{Xi4ba2}{0}{\draw(-0.5,1.5) node[xi] {} -- (0,0); \draw (-1.5,1) node[xic] {} -- (0,0) node[not] {}; \draw[kernels2] (0,0) -- (1.5,1) node[xi] {};
\draw[kernels2] (0,0) -- (0.5,1.5) node[xic] {} ;}

\DeclareSymbol{Xi4ba2b}{0}{\draw(-0.5,1.5) node[xi] {} -- (0,0); \draw (-1.5,1) node[xic] {} -- (0,0) node[not] {}; \draw[kernels2] (0,0) -- (1.5,1) node[xi] {};
\draw[kernels2] (0,0) -- (0.5,1.5) node[xiesf] {} ;}

\DeclareSymbol{Xi4ca}{0}{\draw (0,1) -- (-1,2.2) node[xi] {};\draw (0,-0.25) node[xi] {} -- (0,1) ; \draw[kernels2] (0,1) node[not] {} -- (1,2.2) node[xi] {};
\draw[kernels2] (0,1) {} -- (0,2.7) node[xi] {};
}

\DeclareSymbol{Xi4cb}{0}{\draw (-1,1) -- (-2,2) node[xi] {};\draw[kernels2] (0,0)  -- (-1,1) node[xi] {} ; \draw[kernels2] (0,0) node[not] {} -- (1,1) node[xi] {} ; 
\draw (-1,1) node[xi] {} -- (0,2) node[xi] {};}

\DeclareSymbol{Xi4cbb}{0}{\draw (-1,1) -- (-2,2) node[xiesf] {};\draw[kernels2] (0,0)  -- (-1,1) node[xi] {} ; \draw[kernels2] (0,0) node[not] {} -- (1,1) node[xi] {} ; 
\draw (-1,1) node[xi] {} -- (0,2) node[xic] {};}

\DeclareSymbol{Xi4cbc1}{0}{\draw (-1,1) -- (-2,2) node[xic] {};\draw[kernels2] (0,0)  -- (-1,1) node[xic] {} ; \draw[kernels2] (0,0) node[not] {} -- (1,1) node[xi] {} ; 
\draw (-1,1) node[xic] {} -- (0,2) node[xi] {};}

\DeclareSymbol{Xi4cbc2}{0}{\draw (-1,1) -- (-2,2) node[xi] {};\draw[kernels2] (0,0)  -- (-1,1) node[xi] {} ; \draw[kernels2] (0,0) node[not] {} -- (1,1) node[xic] {} ; 
\draw (-1,1) node[xic] {} -- (0,2) node[xi] {};}

\DeclareSymbol{Xi4cab}{0}{\draw (-1,1) -- (-2,2) node[xi] {};\draw[kernels2] (0,0)  -- (-1,1); \draw[kernels2] (0,0) node[not] {} -- (1,1) node[xi] {} ; 
\draw[kernels2] (-1,1)  {} -- (0,2) node[xi] {};
\draw[kernels2] (-1,1) node[not] {} -- (-1,2.5) node[xi] {};
}

\DeclareSymbol{Xi4cabdis}{0}{\draw (-1,1) -- (-2,2) node[xies] {};\draw[kernels2] (0,0)  -- (-1,1); \draw[kernels2] (0,0) node[not] {} -- (1,1) node[xies] {} ; 
\draw[kernels2] (-1,1)  {} -- (0,2) node[xies] {};
\draw[kernels2] (-1,1) node[not] {} -- (-1,2.5) node[xies] {};
}

\DeclareSymbol{Xi4cabc1}{0}{\draw (-1,1) -- (-2,2) node[xi] {};\draw[kernels2] (0,0)  -- (-1,1); \draw[kernels2] (0,0) node[not] {} -- (1,1) node[xic] {} ; 
\draw[kernels2] (-1,1)  {} -- (0,2) node[xic] {};
\draw[kernels2] (-1,1) node[not] {} -- (-1,2.5) node[xi] {};
}

\DeclareSymbol{Xi4cabc2}{0}{\draw (-1,1) -- (-2,2) node[xic] {};\draw[kernels2] (0,0)  -- (-1,1); \draw[kernels2] (0,0) node[not] {} -- (1,1) node[xic] {} ; 
\draw[kernels2] (-1,1)  {} -- (0,2) node[xi] {};
\draw[kernels2] (-1,1) node[not] {} -- (-1,2.5) node[xi] {};
}

\DeclareSymbol{Xi4ea}{1.5}{\draw (-1,2.5) node[xi] {} -- (-1,1) node[xi] {} -- (0,0); 
 \draw[kernels2] (0,0)  -- (1,1) node[xi] {};
\draw[kernels2] (0,0) node[not] {} -- (0,1.5) node[xi] {}; }

\DeclareSymbol{Xi4eac1}{1.5}{\draw (-1,2.5) node[xic] {} -- (-1,1) node[xi] {} -- (0,0); 
 \draw[kernels2] (0,0)  -- (1,1) node[xic] {};
\draw[kernels2] (0,0) node[not] {} -- (0,1.5) node[xi] {}; }

\DeclareSymbol{Xi4eac1b}{1.5}{\draw (-1,2.5) node[xic] {} -- (-1,1) node[xi] {} -- (0,0); 
 \draw[kernels2] (0,0)  -- (1,1) node[xiesf] {};
\draw[kernels2] (0,0) node[not] {} -- (0,1.5) node[xi] {}; }

\DeclareSymbol{Xi4eac2}{1.5}{\draw (-1,2.5) node[xic] {} -- (-1,1) node[xic] {} -- (0,0); 
 \draw[kernels2] (0,0)  -- (1,1) node[xi] {};
\draw[kernels2] (0,0) node[not] {} -- (0,1.5) node[xi] {}; }

\DeclareSymbol{Xi4eact1}{1.5}{\draw (-1,2.5) node[xic] {} -- (-1,1) node[xi] {} -- (0,0); 
 \draw (0,0)  -- (1,1) node[xic] {};
\draw[rho] (0,0) node[not] {} -- (0,1.5) node[xi] {}; }

\DeclareSymbol{Xi4eact2}{1.5}{\draw[rho] (-1,2.5) node[xic] {} -- (-1,1) node[xi] {} -- (0,0); 
 \draw (0,0)  -- (1,1) node[xic] {};
\draw (0,0) node[not] {} -- (0,1.5) node[xi] {}; }

\DeclareSymbol{Xi4eabis}{1.5}{\draw (-1,2.5) node[xi] {} -- (-1,1) ; \draw[kernels2] (-1,1) node[xi] {} -- (0,0); 
 \draw (0,0)  -- (1,1) node[xi] {};
\draw[kernels2] (0,0) node[not] {} -- (0,1.5) node[xi] {}; }

\DeclareSymbol{Xi4eabisc1}{1.5}{\draw (-1,2.5) node[xic] {} -- (-1,1) ; \draw[kernels2] (-1,1) node[xi] {} -- (0,0); 
 \draw (0,0)  -- (1,1) node[xi] {};
\draw[kernels2] (0,0) node[not] {} -- (0,1.5) node[xic] {}; }

\DeclareSymbol{Xi4eabisc1b}{1.5}{\draw (-1,2.5) node[xic] {} -- (-1,1) ; \draw[kernels2] (-1,1) node[xi] {} -- (0,0); 
 \draw (0,0)  -- (1,1) node[xi] {};
\draw[kernels2] (0,0) node[not] {} -- (0,1.5) node[xiesf] {}; }

\DeclareSymbol{Xi4eabisc1bis}{1.5}{\draw (-1,2.5) node[xi] {} -- (-1,1) ; \draw[kernels2] (-1,1) node[xi] {} -- (0,0); 
 \draw (0,0)  -- (1,1) node[xi] {};
\draw[kernels2] (0,0) node[not] {} -- (0,1.5) node[xi] {};
\draw (-2,1) node[] {\tiny{$i$}};
\draw (-2,2.5) node[] {\tiny{$\ell$}};
\draw (2,1) node[] {\tiny{$k$}};
\draw (0,2.5) node[] {\tiny{$j$}};
 }

\DeclareSymbol{Xi4eabisc1tris}{1.5}{\draw (-1,2.5) node[xi] {} -- (-1,1) ; \draw[kernels2] (-1,1) node[xi] {} -- (0,0); 
 \draw (0,0)  -- (1,1) node[xi] {};
\draw[kernels2] (0,0) node[not] {} -- (0,1.5) node[xi] {};
\draw (-2,1) node[] {\tiny{i}};
\draw (-2,2.5) node[] {\tiny{j}};
\draw (2,1) node[] {\tiny{j}};
\draw (0,2.5) node[] {\tiny{i}};
 }

\DeclareSymbol{Xi4eabisc1quater}{1.5}{\draw (-1,2.5) node[xic] {} -- (-1,1) ; \draw[kernels2] (-1,1) node[xi] {} -- (0,0); 
 \draw (0,0)  -- (1,1) node[xic] {};
\draw[kernels2] (0,0) node[not] {} -- (0,1.5) node[xi] {};
 }

\DeclareSymbol{Xi4eabisc2}{1.5}{\draw (-1,2.5) node[xic] {} -- (-1,1) ; \draw[kernels2] (-1,1) node[xi] {} -- (0,0); 
 \draw (0,0)  -- (1,1) node[xic] {};
\draw[kernels2] (0,0) node[not] {} -- (0,1.5) node[xi] {}; }

\DeclareSymbol{Xi4eabisc2l}{1.5}{\draw (-1,2.5) node[xiesf] {} -- (-1,1) ; \draw[kernels2] (-1,1) node[xi] {} -- (0,0); 
 \draw (0,0)  -- (1,1) node[xic] {};
\draw[kernels2] (0,0) node[not] {} -- (0,1.5) node[xi] {}; }

\DeclareSymbol{Xi4eabisc2r}{1.5}{\draw (-1,2.5) node[xic] {} -- (-1,1) ; \draw[kernels2] (-1,1) node[xi] {} -- (0,0); 
 \draw (0,0)  -- (1,1) node[xiesf] {};
\draw[kernels2] (0,0) node[not] {} -- (0,1.5) node[xi] {}; }

\DeclareSymbol{Xi4eabisc3}{1.5}{\draw (-1,2.5) node[xic] {} -- (-1,1) ; \draw[kernels2] (-1,1) node[xic] {} -- (0,0); 
 \draw (0,0)  -- (1,1) node[xi] {};
\draw[kernels2] (0,0) node[not] {} -- (0,1.5) node[xi] {}; }

\DeclareSymbol{Xi4eb}{0}{
\draw[kernels2] (0,2) node[xi] {} -- (-1,1) ; \draw[kernels2] (-2,2)  node[xi] {} -- (-1,1) ; \draw (-1,1)  node[not] {} -- (0,0); 
 \draw (0,0) node[xi] {}  -- (1,1) node[xi] {};
}

\DeclareSymbol{Xi4eab}{1.5}{\draw[kernels2] (-1,2.5) node[xi] {} -- (-1,1) ; \draw[kernels2] (-2,2)  node[xi] {} -- (-1,1) ; \draw (-1,1)  node[not] {} -- (0,0); 
 \draw[kernels2] (0,0)  -- (1,1) node[xi] {};
\draw[kernels2] (0,0) node[not] {} -- (0,1.5) node[xi] {}; 
}

\DeclareSymbol{Xi4eabdis}{1.5}{\draw[kernels2] (-1,2.5) node[xies] {} -- (-1,1) ; \draw[kernels2] (-2,2)  node[xies] {} -- (-1,1) ; \draw (-1,1)  node[not] {} -- (0,0); 
 \draw[kernels2] (0,0)  -- (1,1) node[xies] {};
\draw[kernels2] (0,0) node[not] {} -- (0,1.5) node[xies] {}; 
}

\DeclareSymbol{Xi4eabc1}{1.5}{\draw[kernels2] (-1,2.5) node[xic] {} -- (-1,1) ; \draw[kernels2] (-2,2)  node[xi] {} -- (-1,1) ; \draw (-1,1)  node[not] {} -- (0,0); 
 \draw[kernels2] (0,0)  -- (1,1) node[xic] {};
\draw[kernels2] (0,0) node[not] {} -- (0,1.5) node[xi] {}; 
}

\DeclareSymbol{Xi4eabc2}{1.5}{\draw[kernels2] (-1,2.5) node[xi] {} -- (-1,1) ; \draw[kernels2] (-2,2)  node[xi] {} -- (-1,1) ; \draw (-1,1)  node[not] {} -- (0,0); 
 \draw[kernels2] (0,0)  -- (1,1) node[xic] {};
\draw[kernels2] (0,0) node[not] {} -- (0,1.5) node[xic] {}; 
}

\DeclareSymbol{Xi4eabbis}{1.5}{\draw[kernels2] (-1,2.5) node[xi] {} -- (-1,1) ; \draw[kernels2] (-2,2)  node[xi] {} -- (-1,1) ; \draw[kernels2] (-1,1)  node[not] {} -- (0,0); 
 \draw (0,0)  -- (1,1) node[xi] {};
\draw[kernels2] (0,0) node[not] {} -- (0,1.5) node[xi] {}; 
}

\DeclareSymbol{Xi4eabbisc1}{1.5}{\draw[kernels2] (-1,2.5) node[xic] {} -- (-1,1) ; \draw[kernels2] (-2,2)  node[xi] {} -- (-1,1) ; \draw[kernels2] (-1,1)  node[not] {} -- (0,0); 
 \draw (0,0)  -- (1,1) node[xic] {};
\draw[kernels2] (0,0) node[not] {} -- (0,1.5) node[xi] {}; 
}

\DeclareSymbol{Xi4eabbisc1perm}{1.5}{\draw[kernels2] (-1,2.5) node[xi] {} -- (-1,1) ; \draw[kernels2] (-2,2)  node[xic] {} -- (-1,1) ; \draw[kernels2] (-1,1)  node[not] {} -- (0,0); 
 \draw (0,0)  -- (1,1) node[xic] {};
\draw[kernels2] (0,0) node[not] {} -- (0,1.5) node[xi] {}; 
}

\DeclareSymbol{Xi4eabbisc2}{1.5}{\draw[kernels2] (-1,2.5) node[xi] {} -- (-1,1) ; \draw[kernels2] (-2,2)  node[xi] {} -- (-1,1) ; \draw[kernels2] (-1,1)  node[not] {} -- (0,0); 
 \draw (0,0)  -- (1,1) node[xic] {};
\draw[kernels2] (0,0) node[not] {} -- (0,1.5) node[xic] {}; 
}

\DeclareSymbol{Xi2cbis}{0}{\draw[kernels2] (0,1) -- (0.8,2.2) node[xi] {};\draw[kernels2] (0,-0.25) node[not] {} -- (0,1); \draw[kernels2] (0,1) node[not] {} -- (-0.8,2.2) node[xi] {};}

\DeclareSymbol{Xi2cbis1}{0}{\draw (0,1) -- (-0.8,2.2) node[xi] {};\draw[kernels2] (0,-0.25) node[not] {} -- (0,1) node[xi] {}; }

\DeclareSymbol{Xi2Xbis}{-2}{\draw[kernels2] (0,-0.25)  -- (-1,1) ; \draw (-1,1) node[xix] {};
\draw[kernels2] (0,-0.25) node[not] {} -- (1,1) node[xi] {};}

\DeclareSymbol{XXi2bis}{-2}{\draw[kernels2] (0,-0.25) -- (-1,1) node[xi] {};
\draw[kernels2] (0,-0.25) node[X] {} -- (1,1) node[xi] {};}

\DeclareSymbol{I1XiIXi}{0}{\draw[kernels2] (0,-0.25) -- (1,1) node[xi] {};
\draw (0,-0.25) node[not] {} -- (-1,1) node[xi] {};}

\DeclareSymbol{I1XiIXib}{0}{\draw  (0,-0.25) node[xi] {} -- (0,1) node[not] {};
\draw[kernels2] (0,1) -- (0,2.25) ; \draw (0,2.25) node[xi]{}; }

\DeclareSymbol{I1XiIXic}{0}{
\draw[kernels2] (0,0) -- (1,1) node[xi] {} ; 
\draw[kernels2] (0,0) node[not] {}  -- (-1,1) node[not] {} -- (0,2) node[xi] {};
}

\DeclareSymbol{thin}{1.4}{\draw[pagebackground] (-0.3,0) -- (0.3,0); \draw  (0,0) -- (0,2);}
\DeclareSymbol{thin2}{1.4}{\draw[pagebackground] (-0.3,0) -- (0.3,0); \draw[tinydots]  (0,0) -- (0,2);}

\DeclareSymbol{thick}{1.4}{\draw[pagebackground] (-0.3,0) -- (0.3,0); \draw[kernels2]  (0,0) -- (0,2);}

\DeclareSymbol{thick2}{1.4}{\draw[pagebackground] (-0.3,0) -- (0.3,0); \draw[kernels2,tinydots]  (0,0) -- (0,2);}

\DeclareSymbol{Xi4ind}{2}{\draw (0,0) node[xi,label={[label distance=-0.2em]right: \scriptsize  $ i $}]  { } -- (-1,1) node[xi,label={[label distance=-0.2em]left: \scriptsize  $ j $}] {} -- (0,2) node[xi,label={[label distance=-0.2em]right: \scriptsize  $ k $}] {} -- (-1,3) node[xi,label={[label distance=-0.2em]left: \scriptsize  $ \ell $}] {};}

\DeclareSymbol{Xi4c1}{2}{\draw (0,0) node[xic] {} -- (-1,1) node[xi] {} -- (0,2) node[xic] {} -- (-1,3) node[xi] {};} 
\DeclareSymbol{IXi2ex}{0}{\draw (0,-0.25) node[xie] {} -- (-1,1) node[xi] {} ; \draw (0,-0.25)-- (1,1) node[xi] {};}
\DeclareSymbol{IXi2ex1}{0}{\draw (0,-0.25) node[xie] {} -- (-1,1) node[xi] {} -- (0,2) node[xi] {};}

\DeclareSymbol{Xi4b1}{0}{\draw(0,1.5) node[xic] {} -- (0,0); \draw (-1,1) node[xic] {} -- (0,0) node[xi] {} -- (1,1) node[xi] {};}

\DeclareSymbol{Xi4ec1}{0}{\draw (0,2) node[xi] {} -- (-1,1) node[xic] {} -- (0,0) node[xic] {} -- (1,1) node[xi] {};}
\DeclareSymbol{Xi4ec2}{0}{\draw (0,2) node[xic] {} -- (-1,1) node[xi] {} -- (0,0) node[xic] {} -- (1,1) node[xi] {};}
\DeclareSymbol{Xi4ec3}{0}{\draw (0,2) node[xic] {} -- (-1,1) node[xic] {} -- (0,0) node[xi] {} -- (1,1) node[xi] {};}

\DeclareSymbol{I1Xi4ac1}{2}{\draw[kernels2] (0,0) node[not] {} -- (-1,1) ; \draw[kernels2] (0,0) node[not] {} -- (1,1) node[xic] {} ;
\draw (-1,1) node[xi] {} -- (0,2) node[xic] {} -- (-1,3) node[xi] {};}

\DeclareSymbol{I1Xi4ac2}{2}{\draw[kernels2] (0,0) node[not] {} -- (-1,1) ; \draw[kernels2] (0,0) node[not] {} -- (1,1) node[xic] {} ;
\draw (-1,1) node[xi] {} -- (0,2) node[xi] {} -- (-1,3) node[xic] {};}

\DeclareSymbol{I1Xi4bp}{2}{\draw (0,0) node[not] {} -- (-1,1) node[xi] {} -- (0,2) ; \draw[kernels2] (0,2) node[not] {} -- (-1,3) node[xi] {};\draw[kernels2] (0,2)  -- (1,3) node[xi] {};
}

\DeclareSymbol{I1Xi4bc1}{2}{\draw (0,0) node[xic] {} -- (-1,1) node[xi] {} -- (0,2) ; \draw[kernels2] (0,2) node[not] {} -- (-1,3) node[xi] {};\draw[kernels2] (0,2)  -- (1,3) node[xic] {};
}

\DeclareSymbol{I1Xi4bc2}{2}{\draw (0,0) node[xic] {} -- (-1,1) node[xi] {} -- (0,2) ; \draw[kernels2] (0,2) node[not] {} -- (-1,3) node[xic] {};\draw[kernels2] (0,2)  -- (1,3) node[xi] {};
}

\DeclareSymbol{I1Xi4cp}{2}{\draw (0,0) node[not] {} -- (-1,1) node[not] {}; \draw[kernels2] (-1,1) -- (0,2) ; 
\draw[kernels2] (-1,1) -- (-2,2) node[xi] {} ;
\draw (0,2) node[xi] {} -- (-1,3) node[xi] {};}

\DeclareSymbol{I1Xi4cc1}{2}{\draw (0,0) node[xic] {} -- (-1,1) node[not] {}; \draw[kernels2] (-1,1) -- (0,2) ; 
\draw[kernels2] (-1,1) -- (-2,2) node[xi] {} ;
\draw (0,2) node[xic] {} -- (-1,3) node[xi] {};}

\DeclareSymbol{I1Xi4cc2}{2}{\draw (0,0) node[xic] {} -- (-1,1) node[not] {}; \draw[kernels2] (-1,1) -- (0,2) ; 
\draw[kernels2] (-1,1) -- (-2,2) node[xi] {} ;
\draw (0,2) node[xi] {} -- (-1,3) node[xic] {};}

\DeclareSymbol{I1Xi4abc1}{2}{\draw[kernels2] (0,0) node[not] {} -- (-1,1) ; \draw[kernels2] (0,0) node[not] {} -- (1,1) node[xic] {};\draw (-1,1) node[xi] {} -- (0,2) ; \draw[kernels2] (0,2) node[not] {} -- (-1,3) node[xic] {};\draw[kernels2] (0,2)  -- (1,3) node[xi] {}; }

\DeclareSymbol{I1Xi4abc2}{2}{\draw[kernels2] (0,0) node[not] {} -- (-1,1) ; \draw[kernels2] (0,0) node[not] {} -- (1,1) node[xic] {};\draw (-1,1) node[xi] {} -- (0,2) ; \draw[kernels2] (0,2) node[not] {} -- (-1,3) node[xi] {};\draw[kernels2] (0,2)  -- (1,3) node[xic] {}; }

\DeclareSymbol{R1}{0}{\draw (-1,1) node[xi] {} -- (0,0) node[not] {};
\draw[kernels2] (0,1.5) node[xic] {} -- (0,0) -- (1,1) node[xic] {};}
\DeclareSymbol{R2}{0}{\draw (-1,1) node[xic] {} -- (0,0) node[not] {};
\draw[kernels2] (0,1.5)  {} -- (0,0) -- (1,1)  {};
\draw (0,1.5) node[xi] {};
\draw (1,1) node[xic] {};
}
\DeclareSymbol{R3}{1}{\draw[kernels2] (-1,1.5)  {} -- (0,0) node[not] {} -- (1,1.5);
\draw (-1,1.5) node[xi] {};
\draw[kernels2] (0,3) {} -- (1,1.5) -- (2,3)  {};
\draw  (0,3) node[xic] {} ;
\draw (2,3) node[xic] {};}
\DeclareSymbol{R4}{1}{\draw[kernels2] (-1,1.5) node[xic] {} -- (0,0) node[not] {} -- (1,1.5);
\draw[kernels2] (0,3) {} -- (1,1.5) -- (2,3) node[xic] {};
\draw (0,3) node[xi] {};}

\DeclareSymbol{I1Xi4bcp}{2}{\draw (0,0) node[not] {} -- (-1,1) node[not] {}; \draw[kernels2] (-1,1) -- (0,2) ; 
\draw[kernels2] (-1,1) -- (-2,2) node[xi] {} ; \draw[kernels2] (0,2) node[not] {} -- (-1,3) node[xi] {};\draw[kernels2] (0,2)  -- (1,3) node[xi] {};
}

\DeclareSymbol{I1Xi4bcc1}{2}{\draw (0,0) node[xic] {} -- (-1,1) node[not] {}; \draw[kernels2] (-1,1) -- (0,2) ; 
\draw[kernels2] (-1,1) -- (-2,2) node[xi] {} ; \draw[kernels2] (0,2) node[not] {} -- (-1,3) node[xi] {};\draw[kernels2] (0,2)  -- (1,3) node[xic] {};
}

\DeclareSymbol{I1Xi4bcc2}{2}{\draw (0,0) node[xic] {} -- (-1,1) node[not] {}; \draw[kernels2] (-1,1) -- (0,2) ; 
\draw[kernels2] (-1,1) -- (-2,2) node[xi] {} ; \draw[kernels2] (0,2) node[not] {} -- (-1,3) node[xic] {};\draw[kernels2] (0,2)  -- (1,3) node[xi] {};
} 

\DeclareSymbol{2I1Xi4bc1}{2}{\draw[kernels2] (0,0) node[not] {} -- (-1,1) ;
\draw[kernels2] (0,0) -- (1,1);
\draw (-1,1) node[xic] {} -- (-1,2.5) node[xi] {};
\draw (1,1)  node[xic] {} -- (1,2.5) node[xi] {};
}

\DeclareSymbol{2I1Xi4bc2}{2}{\draw[kernels2] (0,0) node[not] {} -- (-1,1) ;
\draw[kernels2] (0,0) -- (1,1);
\draw (-1,1) node[xi] {} -- (-1,2.5) node[xic] {};
\draw (1,1)  node[xic] {} -- (1,2.5) node[xi] {};
}

\DeclareSymbol{diff2I1Xi4bc2}{2}{\draw (0,0) node[diff] {} -- (-1,1) ;
\draw (0,0) -- (1,1);
\draw (-1,1) node[xi] {} -- (-1,2.5) node[xic] {};
\draw (1,1)  node[xic] {} -- (1,2.5) node[xi] {};
}

\DeclareSymbol{2I1Xi4bc3}{2}{\draw[kernels2] (0,0) node[not] {} -- (-1,1) ;
\draw[kernels2] (0,0) -- (1,1);
\draw (-1,1) node[xic] {} -- (-1,2.5) node[xic] {};
\draw (1,1)  node[xi] {} -- (1,2.5) node[xi] {};
}

\DeclareSymbol{Xi41}{0}{\draw (0,1) -- (0.8,2.2) node[xic] {};\draw (0,-0.25) node[xi] {} -- (0,1) node[xi] {} -- (-0.8,2.2) node[xic] {};} 

\DeclareSymbol{Xi42}{0}{\draw (0,1) -- (0.8,2.2) node[xi] {};\draw (0,-0.25) node[xic] {} -- (0,1) node[xi] {} -- (-0.8,2.2) node[xic] {};}

\DeclareSymbol{Xi4ca1}{0}{\draw (0,1) -- (-1,2.2) node[xic] {};\draw (0,-0.25) node[xi] {} -- (0,1) ; \draw[kernels2] (0,1) node[not] {} -- (1,2.2) node[xic] {};
\draw[kernels2] (0,1) {} -- (0,2.7) node[xi] {};
}

\DeclareSymbol{Xi4ca2}{0}{\draw (0,1) -- (-1,2.2) node[xi] {};\draw (0,-0.25) node[xi] {} -- (0,1) ; \draw[kernels2] (0,1) node[not] {} -- (1,2.2) node[xic] {};
\draw[kernels2] (0,1) {} -- (0,2.7) node[xic] {};
}

\DeclareSymbol{Xi4cap}{0}{\draw (0,1) -- (-1,2.2) node[xi] {};\draw (0,-0.25) node[not] {} -- (0,1) ; \draw[kernels2] (0,1) node[not] {} -- (1,2.2) node[xi] {};
\draw[kernels2] (0,1) {} -- (0,2.7) node[xi] {};
}

\DeclareSymbol{Xi3a}{0}{
 \draw (-1,1)  node[xi] {} -- (0,0); 
 \draw (0,0) node[xi] {}  -- (1,1) node[xi] {};
}

\DeclareSymbol{Xi4ebc1}{0}{
\draw[kernels2] (0,2) node[xi] {} -- (-1,1) ; \draw[kernels2] (-2,2)  node[xic] {} -- (-1,1) ; \draw (-1,1)  node[not] {} -- (0,0); 
 \draw (0,0) node[xic] {}  -- (1,1) node[xi] {};
}

\DeclareSymbol{Xi4ebc2}{0}{
\draw[kernels2] (0,2) node[xi] {} -- (-1,1) ; \draw[kernels2] (-2,2)  node[xi] {} -- (-1,1) ; \draw (-1,1)  node[not] {} -- (0,0); 
 \draw (0,0) node[xic] {}  -- (1,1) node[xic] {};
}

\DeclareSymbol{Xi2cbispex}{0}{\draw[kernels2] (0,1) -- (0.8,2.2) node[xi] {};\draw (0,-0.25) node[xie] {} -- (0,1); \draw[kernels2] (0,1) node[not] {} -- (-0.8,2.2) node[xi] {};}

\DeclareSymbol{Xi2cbis1p}{0}{\draw (0,1) -- (-0.8,2.2) node[xi] {};\draw (0,-0.25) node[not] {} -- (0,1) node[xi] {}; }

\DeclareSymbol{Xi2Xp}{-2}{\draw (0,-0.25) node[not] {} -- (-1,1) node[xix] {};} 

\DeclareSymbol{I1XiIXib}{0}{\draw  (0,-0.25) node[xi] {} -- (0,1) node[not] {};
\draw[kernels2] (0,1) -- (0,2.25) ; \draw (0,2.25) node[xi]{}; }

\DeclareSymbol{IXi2b}{0}{\draw  (0,-0.25) node[xi] {} -- (0,1) node[not] {};
\draw (0,1) -- (0,2.25) ; \draw (0,2.25) node[xi]{}; }

\DeclareSymbol{IXi2bex}{0}{\draw  (0,-0.25) node[xi] {} -- (0,1) node[xie] {};
\draw (0,1) -- (0,2.25) ; \draw (0,2.25) node[xi]{}; }

 \def\1{\mathbf{\symbol{1}}}

\def\one{\mathbf{1}}

\DeclareSymbol{diff}{0}{
\draw (0,0.5) node[diff] {};
}

\DeclareSymbol{diff1}{0}{
\draw (0,0.5) node[diff1] {};
}

\DeclareSymbol{diff2}{0}{
\draw (0,0.5) node[diff2] {};
}

\DeclareSymbol{geo}{0}{
\draw (0,0) node[diff] {};
\draw (0.3,0) node[diff] {};
}

\DeclareSymbol{generic}{0}{
\draw (0,0.6) node[xi] {};
}

\DeclareSymbol{g}{0}{
\draw (0,0.6) node[g] {};
}

\DeclareSymbol{Ito}{0}{
\draw (0,0.6) node[xies] {};
}

\DeclareSymbol{Itob}{0}{
\draw (0,0.6) node[xiesf] {};
}

\DeclareSymbol{greycirc}{0}{
\draw (0,0.3) node[xi] {};
}

\DeclareSymbol{not}{0}{
\draw (0,0.6) node[not] {};
\draw[tinydots] (0,0.6) circle (0.8);
}

\DeclareSymbol{genericb}{0}{
\draw (0,0.6) node[xic] {};
}

\DeclareSymbol{bluecirc}{0}{
\draw (0,0.3) node[xic] {};
}

\DeclareSymbol{genericxix}{0}{
\draw (0,0.6) node[xix] {};
}

\DeclareSymbol{genericX}{0}{
\draw (0,0.6) node[X] {};
}

\DeclareSymbol{diffIto}{1}{
\draw  (0,2.5) -- (0,0) ;
\draw (0,-0.1) node[diff] {};
\draw (0,2.5) node[xies] {};
}
\DeclareSymbol{Itodiff}{2}{
\draw(0,2.9) -- (0,-0.2);
\draw (0,2.9) node[diff] {};
\draw (0,-0.1) node[xies] {};
}

\DeclareSymbol{diffgeneric}{1}{
\draw  (0,2.5) -- (0,0) ;
\draw (0,-0.1) node[diff] {};
\draw (0,2.5) node[xi] {};
}

\DeclareSymbol{genericdiff}{2}{
\draw(0,2.9) -- (0,-0.2);
\draw (0,2.9) node[diff] {};
\draw (0,-0.1) node[xi] {};
}

\DeclareSymbol{diffdot}{2}{
\draw  (0,3) -- (0,-0.1) ;
\draw (0,3) node[not] {};
\draw (0,-0.1) node[diff] {};
}

\DeclareSymbol{diffdotmini}{0}{
\draw  (0,0) -- (0,1.2) ;
\draw (0,1.2) node[not] {};
\draw (0,0) node[diffmini] {};
}

\DeclareSymbol{dotdiff}{2}{
\draw[kernelsmod]  (0,3) -- (0,-0.1) ;
\draw (0,3) node[diff] {};
\draw (0,-0.1) node[not] {};
}

\DeclareSymbol{dotdiff1}{2}{
\draw[kernelsmod]  (0,3) -- (0,-0.1) ;
\draw (0,3) node[diff1] {};
\draw (0,-0.1) node[not] {};
}

\DeclareSymbol{dotdiff1mini}{0}{
\draw[kernelsmod]  (0,1.2) -- (0,0) ;
\draw (0,1.2) node[diffmini] {};
\draw (0,0) node[not] {};
}

\DeclareSymbol{dotdiff2}{2}{
\draw (0,3) -- (0,-0.1) ;
\draw (0,3) node[diff] {};
\draw (0,-0.1) node[not] {};
}

\DeclareSymbol{dotdiff2mini}{0}{
\draw (0,1.2) -- (0,0) ;
\draw (0,1.2) node[diffmini] {};
\draw (0,0) node[not] {};
}

\DeclareSymbol{dotdiffstraight}{0}{
\draw  (0,3) -- (0,-0.1) ;
\draw (0,3) node[diff] {};
\draw (0,-0.1) node[not] {};
}

\DeclareSymbol{arbre1}{0}{
\draw  (0,0) -- (1.5,1.5) ;
\draw (1.5,1.5) node[not] {};
\draw (0,0) node[not] {};
}

\DeclareSymbol{arbre2}{0}{
\draw  (0,0) -- (1.5,1.5) ;
\draw[kernelsmod] (0,0) -- (-1.5,1.5);
\draw (1.5,1.5) node[not] {};
\draw (0,0) node[not] {};
\draw (-1.5,1.5) node[xi] {};
}

\DeclareSymbol{arbre3}{0}{
\draw  (0,0) -- (1.5,1.5) ;
\draw[kernelsmod] (1.5,1.5) -- (0,3);
\draw (0,0) node[not] {};
\draw (1.5,1.5) node[not] {};
\draw (0,3) node[xi] {};
}

\DeclareSymbol{treeeval}{0}{
\draw (0,0) -- (1,1);
\draw (0,0) node[xi] {};
\draw (1.25,1.25) node[xi] {};
\draw (-0.6,0.6) node[]{\tiny{$i$}};
\draw (0.65,1.85) node[]{\tiny{$j$}};
}

\DeclareSymbol{testeval}{0}{
\draw (0,0) -- (1,1);
\draw (0,0) -- (-1,1);
\draw (0,0) node[xi] {};
\draw (1.25,1.25) node[xi] {};
\draw (-1.25,1.25) node[xi] {};
\draw (-0.6,-0.6) node[]{\tiny{$i$}};
\draw (0.65,1.85) node[]{\tiny{$j$}};
\draw (-1.95,1.85) node[]{\tiny{$k$}};
}

\DeclareSymbol{treeeval2}{0}{
\draw[kernelsmod] (-0.25,-1) -- (1,0.5) ;
\draw[kernelsmod] (1,0.5) -- (-0.25,2);
\draw (1,0.5) node[diff2] {};
\draw (-0.25,-1) node[not] {};
\draw (-0.25,2) node[xi] {};
\draw (-0.6,1.2) node[]{\tiny{1}};
}

\DeclareSymbol{arbreact}{1}{
\draw (0,0) node[not] {};
\draw[kernelsmod] (0,0) -- (1,1);
\draw[kernelsmod] (0,0) -- (-1,1);
\draw (-1,1) node[xic] {};
\draw  (0,2) -- (1,1) ;
\draw (0,2) node[xic] {};
\draw (1,1) node[xi] {};
}

\DeclareSymbol{arbreact1}{0}{
\draw (0,-1.5) -- (0,0);
\draw[kernelsmod] (0,0) -- (1,1);
\draw[kernelsmod] (0,0) -- (-1,1);
\draw  (0,2) -- (1,1) ;
\draw (0,-1.5) node[diff] {};
\draw (0,0) node[not] {};
\draw (-1,1) node[xic] {};
\draw (0,2) node[xic] {};
\draw (1,1) node[xi] {};
}

\DeclareSymbol{arbreact2}{0}{
\draw (0,-0.75) -- (-1,0.5); 
\draw (0,-0.75) -- (1,0.5);
\draw (0,1.5) -- (1,0.5);
\draw (0,1.5) node[xic] {};
\draw (1,0.5) node[xi] {};
\draw (-1,0.5) node[xic] {};
\draw (0,-0.75) node[diff] {};
}

\DeclareSymbol{arbreact3}{0}{
\draw[kernelsmod] (0,-0.75) -- (-1,0.5); 
\draw[kernelsmod] (0,-0.75) -- (1,0.5);
\draw (0,1.75) -- (1,0.5);
\draw (2,1.75) -- (1,0.5);
\draw (0,1.75) node[xic] {};
\draw (1,0.5) node[diff] {};
\draw (-1,0.5) node[xic] {};
\draw (2,1.75) node[xi] {};
\draw (0,-0.75) node[not] {};
}

\DeclareSymbol{pre_im_I1Xitwo}{0}{
\draw[kernels2] (0,-0.3) node[not] {} -- (-0.6,0.7) ;
\draw[kernels2] (0,-0.3) -- (0.6,0.7);
\draw (0,0.9) node[g] {};
}

\DeclareSymbol{pre_im_cI1Xi4ab}{2}{
\draw[kernels2] (0,-1) node[not] {} -- (-0.6,0) ;
\draw[kernels2] (0,-1) -- (0.6,0);
\draw (0,0.2) node[g] {};
\draw (0,0.6) -- (0,1.5);
\draw[kernels2] (0,1.5) node[not] {} -- (-0.6,2.5) ;
\draw[kernels2] (0,1.5) -- (0.6,2.5);
\draw (0,2.7) node[g] {};
}

\DeclareSymbol{pre_im_I1Xi4acc2}{0}{
\draw[kernels2] (-1,-0.5) node[not] {} -- (-1.6,0.5) ;
\draw[kernels2] (-1,-0.5) -- (-0.4,0.5);
\draw[kernels2] (-1,-0.5) -- (0.2,-1.5) node[not] {} ;
\draw (-1,1.1) -- (-1,2);
\draw[kernels2] (0.2,-1.5) -- (0.2,2);
\draw (-1,0.7) node[g] {};
\draw (-0.3,2.2) node[g] {};
}

\DeclareSymbol{pre_im_I1Xi4abcc2}{2}{
\draw[kernels2] (0,-1) node[not] {} -- (-1,0) node[not] {};
\draw[kernels2] (-1,1.2) node[not] {} -- (-1,0);
\draw[kernels2] (-1,1.2) -- (-1.5,2.5);
\draw[kernels2] (-1,1.2) -- (-0.5,2.5);
\draw[kernels2] (-1,0) -- (0.7,2.5);
\draw[kernels2] (0,-1) -- (1.5,2.5);
\draw (-1,2.7) node[g] {};
\draw (1,2.7) node[g] {};
}

\DeclareSymbol{pre_im_2I1Xi4c1}{2}{
\draw[kernels2] (0,-0.5) node[not] {} -- (-1,0.5) node[not] {};
\draw[kernels2] (0,-0.5) -- (1,0.5) node[not] {};
\draw[kernels2] (-1,0.5) node[not] {}-- (-1.7,2);
\draw[kernels2]  (-1,2) -- (1,0.5);
\draw[kernels2] (-1,0.5) -- (1,2);
\draw[kernels2] (1,0.5) -- (1.7,2);
\draw (-1.2,2.2) node[g] {};
\draw (1.2,2.2) node[g] {};
}

\DeclareSymbol{pre_im_Xi4eabisc2}{2}{
\draw[kernels2] (1.2,-0.5) node[not] {} -- (-0.7,0.8) ;
\draw[kernels2] (1.2,-0.5) -- (0.4,0.8);
\draw (0,1.4)  -- (0,2.2);
\draw (1.2,2.2) -- (1.2,-0.6);
\draw (0,1) node[g] {};
\draw (0.6,2.4) node[g] {};
}

\DeclareSymbol{pre_im_Xi4eabisc22}{2}{
\draw (1.2,-0.5) node[not] {} -- (-0.7,0.8) ;
\draw[kernels2] (1.2,-0.5) -- (0.4,0.8);
\draw (0,1.4)  -- (0,2.2);
\draw[kernels2] (1.2,2.2) -- (1.2,-0.6);
\draw (0,1) node[g] {};
\draw (0.6,2.4) node[g] {};
}

\DeclareSymbol{pre_im_Xi4eabisc222}{2}{
\draw[kernels2] (0.4,-0.5) node[not] {} -- (-0.6,1) ;
\draw[kernels2] (1.2,0) -- (0.3,1);
\draw (0,1.1)  -- (0,2.5);
\draw[kernels2] (1.2,2.5) -- (1.2,0) node[not] {} -- (0.4,-0.6);
\draw (0,1.2) node[g] {};
\draw (0.6,2.5) node[g] {};
}

\DeclareSymbol{pre_im_Xi4eabc2}{2}{
\draw (0,-0.5) node[not] {} -- (-1,0.5) node[not] {};
\draw[kernels2] (-1,0.5) -- (-1.5,2);
\draw[kernels2] (0,-0.5)  -- (0.7,2);
\draw[kernels2] (-1,0.5) -- (-0.5,2);
\draw[kernels2] (0,-0.5) -- (1.5,2);
\draw (-1,2.2) node[g] {};
\draw (1,2.2) node[g] {};
}

\DeclareSymbol{pre_im_Xi4eabbisc2}{2}{
\draw[kernels2] (0,-0.5) node[not] {} -- (-1,0.5) node[not] {};
\draw[kernels2] (-1,0.5) -- (-1.5,2);
\draw[kernels2] (0,-0.5)  -- (0.7,2);
\draw[kernels2] (-1,0.5) -- (-0.5,2);
\draw (0,-0.5) -- (1.5,2);
\draw (-1,2.2) node[g] {};
\draw (1,2.2) node[g] {};
}

\DeclareSymbol{pre_im_I1Xi4abcc1}{2}{
\draw[kernels2] (0,-1) node[not] {} -- (-1,0) node[not] {};
\draw[kernels2] (0,1.1) node[not] {} -- (-1,0);
\draw[kernels2] (-1,0) -- (-1.5,2.5);
\draw[kernels2] (0,1.1) node[not] {} -- (-0.5,2.5);
\draw[kernels2] (0,1.1) -- (0.5,2.5);
\draw[kernels2] (0,-1) -- (1.5,2.5);
\draw (-1,2.7) node[g] {};
\draw (1,2.7) node[g] {};
}

\DeclareSymbol{pre_im_Xi4eabc1}{2}{
\draw (0,-0.5) node[not] {} -- (-1,0.5) node[not] {};
\draw[kernels2] (-1,0.5) -- (-1.5,2);
\draw[kernels2] (0,-0.5)  -- (-0.5,2);
\draw[kernels2] (-1,0.5) -- (0.8,2);
\draw[kernels2] (0,-0.5) -- (1.5,2);
\draw (-1,2.2) node[g] {};
\draw (1,2.2) node[g] {};
}

\DeclareSymbol{pre_im_Xi4ba1b}{2}{
\draw[kernels2] (0,0) node[not] {}  -- (1.8,1.5);
\draw[kernels2] (0,0) -- (0.8,1.5);
\draw (0,-0.1) -- (-1.8,1.5);
\draw (0,-0.1) -- (-0.8,1.5);
\draw (-1,1.7) node[g] {};
\draw (1,1.7) node[g] {};
}

\DeclareSymbol{pre_im_Xi4ba2}{2}{
\draw (0,-0.1) node[not] {}  -- (1.8,1.5);
\draw[kernels2] (0,0) -- (0.8,1.5);
\draw (0,-0.1) -- (-1.8,1.5);
\draw[kernels2] (0,0) -- (-0.8,1.5);
\draw (-1,1.7) node[g] {};
\draw (1,1.7) node[g] {};
}

\DeclareSymbol{pre_im_Xi4cabc2}{2}{
\draw[kernels2] (0,-0.5) node[not] {} -- (-1,0.5) node[not] {};
\draw[kernels2] (-1,0.5) -- (-1.5,2);
\draw (-1,0.5)  -- (0.7,2);
\draw[kernels2] (-1,0.5) -- (-0.5,2);
\draw[kernels2] (0,-0.5) -- (1.5,2);
\draw (-1,2.2) node[g] {};
\draw (1,2.2) node[g] {};
}

\DeclareSymbol{pre_im_Xi4cabc1}{2}{
\draw[kernels2] (0,-0.5) node[not] {} -- (-1,0.5) node[not] {};
\draw (-1,0.5) -- (-1.5,2);
\draw[kernels2] (-1,0.5)  -- (0.7,2);
\draw[kernels2] (-1,0.5) -- (-0.5,2);
\draw[kernels2] (0,-0.5) -- (1.5,2);
\draw (-1,2.2) node[g] {};
\draw (1,2.2) node[g] {};
}

\DeclareSymbol{pre_im_Xi4eabbisc1}{2}{
\draw[kernels2] (0,-0.5) node[not] {} -- (-1,0.5) node[not] {};
\draw[kernels2] (-1,0.5) -- (-1.5,2);
\draw[kernels2] (0,-0.5)  -- (-0.5,2);
\draw[kernels2] (-1,0.5) -- (0.8,2);
\draw (0,-0.5) -- (1.5,2);
\draw (-1,2.2) node[g] {};
\draw (1,2.2) node[g] {};
}

\DeclareSymbol{pre_im_1}{0}{
\draw[kernels2] (0,-0.5) node[not] {} -- (-0.6,0.5) ;
\draw[kernels2] (0,-0.5) -- (0.6,0.5);
\draw (0,1.1)  -- (-0.55,2);
\draw (0,1.1)  -- (0.55,2);
\draw (0,0.7) node[g] {};
\draw (0,2.2) node[g] {};
}

\DeclareSymbol{disconnect}{0}{
\draw[kernels2] (0,-0.5) node[not] {} -- (-0.6,0.5) ;
\draw[kernels2] (0,-0.5) -- (0.6,0.5);
\draw (-0.55,1.1)  -- (-0.55,2.3);
\draw (0.55,2.3) -- (0.55,1.5) -- (1.2,1.5) -- (1.2,3.5) -- (0.55,3.5) -- (0.55,2.7);
\draw (0,0.7) node[g] {};
\draw (0,2.5) node[g] {};
}

\DeclareSymbol{pre_im_2}{2}{\draw[kernels2] (0,0) node[not] {} -- (-1,1) node[not] {};
\draw[kernels2] (0,0) -- (1,1) node[not] {};
\draw[kernels2] (-1,1) -- (-1.5,2.5);
\draw[kernels2] (-1,1) -- (-0.5,2.5);
\draw[kernels2] (1,1) -- (0.5,2.5);
\draw[kernels2] (1,1) -- (1.5,2.5);
\draw (-1,2.7) node[g] {};
\draw (1,2.7) node[g] {};
}

\DeclareSymbol{CX_rec}{0}{
\draw [black] (-0.3,1) to (-0.3,-0.3);
\draw [black] (0.3,1) to (0.3,-0.3);
\draw [black] (-0.3,1) to (-0.3,2.3);
\draw [black] (0.3,1) to (0.3,2.3);
\draw (0,1) node[rec] {};
}

\DeclareSymbol{CX_cerc}{0}{
\draw [black] (0,1) to (0,-0.3);
\draw (0,1) node[cerc] {};
}

\DeclareMathAlphabet{\mathpzc}{OT1}{pzc}{m}{it}

\def\eqref#1{(\ref{#1})}

\makeatletter 
\newcommand*{\bigcdot}{}
\DeclareRobustCommand*{\bigcdot}{%
  \mathbin{\mathpalette\bigcdot@{}}%
}
\newcommand*{\bigcdot@scalefactor}{.5}
\newcommand*{\bigcdot@widthfactor}{1.15}
\newcommand*{\bigcdot@}[2]{%
  \sbox0{$#1\vcenter{}$}
  \sbox2{$#1\cdot\m@th$}%
  \hbox to \bigcdot@widthfactor\wd2{%
    \hfil
    \raise\ht0\hbox{%
      \scalebox{\bigcdot@scalefactor}{%
        \lower\ht0\hbox{$#1\bullet\m@th$}%
      }%
    }%
    \hfil
  }%
}
\makeatother

\tcbset
{colframe=boxcolor,colback=symbols!7!pagebackground,coltext=pageforeground,
fonttitle=\bfseries,nobeforeafter,center title,size=fbox,boxsep=1.5pt,
top=0mm,bottom=0mm,boxsep=0mm,tcbox raise base}

\def\two{{\<generic>\kern0.05em\<genericb>}}
\def\twoI{{\<Ito>\kern0.05em\<Itob>}}

\def\mail#1{\burlalt{#1}{mailto:#1}}

\usepackage{thmtools} 

\begin{document}

\title{Derivation of normal forms for dispersive PDEs via arborification}
\author{Yvain Bruned}
\institute{ 
 IECL (UMR 7502), Université de Lorraine
  \\
Email:\ \begin{minipage}[t]{\linewidth}
\mail{yvain.bruned@univ-lorraine.fr}.
\end{minipage}}

\maketitle

\begin{abstract}
In this work, we propose a systematic derivation of normal forms for dispersive equations using decorated trees introduced in \cite{BS}. The key tool is the arborification map which is a morphism from the Butcher-Connes-Kreimer Hopf algebra to the Shuffle Hopf algebra. It originates from Ecalle's approach to dynamical systems with singularities. This natural map has been used in many applications ranging from algebra, numerical analysis and rough paths.  This connection shows that Hopf algebras also appear naturally in the context of dispersive equations and provide insights into some crucial decomposition.
\end{abstract}

\setcounter{tocdepth}{1}
\tableofcontents

\section{Introduction}

Singular stochastic partial differential equations (SPDEs) have been developed in the previous years reaching a high degree of generality. A large class of singular parabolic SPDEs are covered by the theory of Regularity Structures invented by Martin Hairer in \cite{reg}. Surveys on this theory can be found in \cite{FrizHai,BaiHos}. One main component of this success is the use of Hopf algebras on decorated trees that are at the core of the Black Box used for the solution theory in \cite{BHZ,CH16,BCCH}. They are viewed as extensions of Hopf algebras coming from algebraic renormalisation  in quantum field theory \cite{CK1,CK2} and numerical analysis \cite{Butcher,CHV10}.
Another main component is to extend the theory of Rough Paths \cite{Lyons98,Gubinelli2004,Gub06} based on iterated integrals for locally expanding the solution of singular stochastic differential equations. In Rough Paths, one sees already Hopf algebra on words and trees similar to the ones cited above.

In recent years, dispersive equations have seen quick growth partially thanks to the use of sophisticated combinatorics. One can mention the work \cite{DNY22} introducing the random tensors seen by their authors as an analogue of Regularity Structures for solving dispersive equations with random initial data. Another direction is
Wave turbulence in \cite{DH23} where the authors have derived the wave kinetic equation from the cubic nonlinear Schrödinger (NLS) equation at the kinetic timescale. Both works used trees and Feynman diagrams for conducting advanced estimates but without Hopf algebras.
Trees like expansions for dispersive equations appeared also before these works such as in \cite{C07} and have been used in various contexts such as norm inflation in \cite{T17,K19}, normal forms \cite{GKO13,K19,KOY20} and rough solutions \cite{Gub11}. Let us mention that an alternative theory to Regularity Structures for singular SPDEs in \cite{GIP19} is to use paracontrolled calculus whose high order expansion is described in \cite{BB19}. So far, no Hopf algebras structures have been found for this decomposition.

The present work aims to make a connection between the combinatorics developed for dispersive equations and the one for singular SPDEs. We use therefore the formalism introduced in \cite{BS} where decorated trees combine the decorations from \cite{BHZ} and \cite{C07}. These decorated trees provide a systematic  way to derive high order low regularity schemes for dispersive equations. This formalism has been further used for generalised resonance schemes in \cite{ABBS22}, 
discretisation in Wave turbulence, \cite{ABBS24} symmetric schemes  in \cite{ABBMS23} and resonance schemes for SPDEs in \cite{AGB23}.
  Our focus is the derivation of normal forms for dispersive equations with these decorated trees. Poincaré-Dulac normal forms are used for rewriting a PDE of the form
  \begin{equs}
  \partial_t x = Ax + F(x) = Ax +
  \sum_{j=0}^{\infty}
  f_j(x)
  	\end{equs}
  into resonant monomials on which one can get suitable analytical estimates (see \cite{Arnold}). Here, $ A $ is a differential operator with distinct eigenvalues. One first performs a change of variable $ \tilde{x}(t) = e^{-tA} x(t) $ to get
  \begin{equs}
\partial_t \tilde{x} = e^{-tA} F(e^{tA}\tilde{x}).
  \end{equs}
Then, one has a change of variables $ \tilde{z}_k = \tilde{x} + \sum_{j=a}^k \tilde{y}_j $ such that
\begin{equs}
	\partial_t \tilde{z}_k 
	=  e^{-tA} G_k(e^{tA}\tilde{z}_k).
	\end{equs}
where the $ G_k $ maps are called resonant monomials. The final decomposition can be written into
\begin{equs}
	\tilde{x}(t) = \tilde{x}(0) - \sum_{j=1}^k \left( \tilde{y}_j(t) - \tilde{y}_j(0)\right)  + 
\int_0^t e^{-sA} G_k(e^{s A} \tilde{z}_k(s))  ds
\end{equs}
and then one can pass to the limit $ k \rightarrow + \infty $.  The point of the classical Poincaré-Dulac normal form is to renormalise the flow so that it
is expressed in terms of resonant terms. In the context of dispersive equation, one starts for Nonlinear Schrödinger eqution (NLS) with an equation of the form
\begin{equs} \label{reso_decomp}
		\begin{split}
			\partial_t v_k 
			& = - i 
			\sum_{\substack{k = -k_1 +k_2 + k_3\\ k_2\ne k_1, k_3} }
			e^{ i \Phi(\bar{k})t } 
			\bar{v}_{k_1} v_{k_2} v_{k_3}
		- 2  i 
		\sum_{k_1 \in \mathbb{Z} } 	\bar{v}_{k_1} v_{k_1} v_{k}.
		\end{split}
\end{equs}
where the $ v_k $ are the Fourier coefficients of $ v = e^{-i t \partial_x^2} u $ and $u$ is the solution of 
\begin{equs}
	i \partial_t u + \partial_x^2 u  =|u|^2 u, \quad u(0) = u_0, \quad (t,x) \in \mathbb{R} \times \mathbb{T}.
\end{equs}
The phase $ \Phi(\bar{k}) $ is given by
$	\Phi(\bar{k}): = \Phi(k, k_1, k_2, k_3) = k^2 + k_1^2 - k_2^2- k_3^2$. One has already performed a splitting according to the resonance in the equation \eqref{reso_decomp}: The first term is non-resonant with $ 	\Phi(\bar{k}) \neq 0 $ while the second term is resonant. Then, the main idea is to proceed by integration by parts on the non-resonant terms via the identity 
  \begin{equs}
	\Phi(\bar{k}) = \partial_t \left( \frac{e^{i \Phi(\bar{k}) t}}{i\Phi(\bar{k})}  \right).
\end{equs}
 One iterates this trick to produce the full normal form decomposition.
This approach has been very successful in proving unconditional well-posedness in \cite{GKO13,K19,KOY20} and also for the study of quasi-invariant Gaussian measures in \cite{OST18}. Let us mention that normal forms have been also used in the stochastic context in \cite{HRSZ23} for the three dimensional Zakharov system but with only one iteration of the procedure. Scattering of this system is obtained in \cite{GN14} from the same normal form inspired from \cite{S85}.

In fact, one keyword gives us the main algebraic insight on this decomposition which is integration by parts. Indeed, in the context of iterated integrals integration by parts leads to the concept of geometric Rough Paths that are characters for the shuffle Hopf algebra which is the natural Hopf algebra on the words. When one iterates Duhamel's formula for dispersive equations, one gets iterated integrals described by trees and the integration by parts allows one to look at words. This is reflected by the arborification map which is a surjective morphism between decorated trees and some words whose letters are the nodes of the previous trees. Such a map has been extensively used in	 
dynamical systems for classifying singularities via resurgent functions introduced by Jean Ecalle (see \cite{Ecalle1,Ecalle2,EV04,Cr09,FM}).
 Arborification played a crucial role in mould's expansions that correspond to expansions with analytical objects indexed by words. The algebraic properties of the Hopf algebra on the moulds have been explained in \cite{EFFM17}.
 At the level of Rough Paths, arborification appears in \cite{HK15} for establishing a connection between geometric and branched Rough Paths, in \cite{Br173} for renormalising rough SDEs, in \cite{U10,FU13} for describing an algorithm called Fourier normal ordering on iterated integrals. Let us mention that the chronicle trees introduced in \cite{GKO13} are very similar to the heap-ordered trees in \cite{FU13}. They are just a different way to describe the arborification. 
 The arborification also appears in numerical analysis in \cite{Murua2006,Murua2017}
   for efficient calculations involving Lie series in problems of control theory.
   
   Our main result is to provide a derivation of the normal form introduced in \cite{GKO13} combining the Fourier decorated trees introduced in \cite{BS} with Butcher series formalism, to the arborification coming from Ecalle's theory on the mould (see \cite{EV04,FM}). We are able in Theorem \ref{theorem_dev_N}  to provide a combinatorial insight on the main elements of the normal form. Let us mention that the proof of this main result relies on a new identity (up to the knowledge of the author) of the arborification in Proposition \ref{alternative_arb}. We also make a correspondence with a key character $\Psi$  presented in \cite{Cr09,FM} (see also Proposition \ref{character_prop})  whose modification $ \tilde{\Psi} $ is at the core of the normal form. This connection is not so surprising as iterating Duhamel's formula in Fourier space generates iterated integrals in time that are connected to the shuffle Hopf algebra on words. Integration by parts is well encoded via the arborification which is its translation into combinatorics. We hope to see more links in this direction connecting expansion for dispersive expansion and Hopf algebra techniques.

	Finally, let us outline the paper by summarising the content of its sections. In Section~\ref{Sec::2}, we introduce the shuffle Hopf algebra $ T(A) $ over an alphabet $A$. Its main analytical applications are iterated integrals over a simplex given in \eqref{iterated_integrals}. The character property of these integrals in Proposition~\ref{character_property} is the result of the smoothness of the path which allows integrations by parts for rewriting products of integrals on the simplex into an integral over the simplex. This mechanism is crucial for the sequel and we illustrate it with a simple example. Then, we also present the Butcher-Connes-Kreimer Hopf algebra $ \CH^A_{\text{\tiny{BCK}}} $ on forest decorated by $A$ together with the grafting product on decorated trees and the Butcher-Connes-Kreimer coproduct $\Delta_{\text{\tiny{BCK}}} $ via recursive and non-recursive formulae. The main definition of the section is the arborifaction map $\mathfrak{a}$ in \eqref{def_arborification}, 
 a surjective Hopf morphism between $ \CH^A_{\text{\tiny{BCK}}} $ and $ T(A) $.
 We propose an alternative recursive definition of this map in Proposition~\ref{alternative_arb} derived from the adjoint of the grafting product. It is  obtained from some type of co-morphism in Proposition~\ref{identi_c}. This definition is compared with the one of the coproduct $\Delta_{\text{\tiny{BCK}}} $ in \eqref{new_identity_arb}.	
 
 In Section~\ref{Sec::3}, we recall the Fourier decorated trees introduced in \cite{BS} in the context of numerical analysis. They are used to describe iteration of Duhamel's formula see Proposition~\ref{tree_series}. To state this result, one has to introduce various important definitions such as Definition\ref{decorated_trees} for the decorated trees, \eqref{space_deco} for the set of decorated trees coming from the dispersive equation, the symmetry factor in \eqref{symmetry_factor}, the character $\Pi$ in \eqref{def_Pi} that maps a decorated trees onto an oscillatory integral and the definition of the coefficients/elementary differentials in \eqref{def_Upsilon}. An important definition is Definition\ref{dom_freq} which introduces a recursive way to compute the oscillations associated with a decorated tree. We adapt the definition of the Butcher-Connes-Kreimer coproduct and the arborification in this new context see \eqref{BCK_new}  and \eqref{arbo_new}. We finish the section with the introduction of a map $ \tilde{\Psi} $ on these decorated trees in \eqref{def_char} which is actually a character for the shuffle product see Proposition \ref{character_prop}. A modification of this character denoted by $ \Psi $ in \eqref{Phi} is crucial for encoding the normal form.   
 
	In Section~\ref{Sec::4},  we first recall the normal form \eqref{N12} introduced in \cite{GKO13}.  We explain its derivation at the second order using an integration by parts. It allows us to make appear an oscillation at the denominator. Before stating our main result, in Proposition \ref{morphism_grafting}, we propose a reinterpretation of the derivation in time of the elementary differential as a combinatorial operation involving the grafting product on decorates trees previously introduced. Thanks to this Proposition, we can prove Theorem \ref{theorem_dev_N} which provides  combinatorial formulae   for the terms coming from the normal form \eqref{N12}. It involves the Butcher-type series defined on Fourier decorated trees, the arborification map and the map $ \Psi $ previously introduced.

 \subsection*{Acknowledgements}
 
 {\small
 	Y. B. thanks the organisers of the meeting of the ANR SMOOTH in  Brest May 2024, especially Ismaël Bailleul and Nicolay Tzvetkov whose questions and discussions provide motivation for this work.  
 	Y. B. gratefully acknowledges funding support from the European Research Council (ERC) through the ERC Starting Grant Low Regularity Dynamics via Decorated Trees (LoRDeT), grant agreement No.\ 101075208.
 	Views and opinions expressed are however those of the author(s) only and do not necessarily reflect those of the European Union or the European Research Council. Neither the European Union nor the granting authority can be held responsible for them.
 } 

\subsection*{Data Availibility} Data sharing not applicable to this article as no datasets were
generated or analysed during the current study.

\subsection*{Conflict of interest} The authors declare that they have no conflict of interest.

\section{Shuffle Hopf algebra}

\label{Sec::2}

  Given an alphabet $ A $, we consider the linear span of the words on this alphabet denoted by $ T(A) $. We set $ \varepsilon  $ as the empty word. The product on $ T(A)$ is the shuffle product defined by
  \begin{equs}
  \varepsilon \shuffle v=v\shuffle \varepsilon  =v, \quad (au\shuffle bv) = a(u\shuffle bv) + b(au\shuffle v)
  \end{equs}
for all $u,v\in T(A)$ and $a,b\in A$.
 The coproduct $\Delta:T(A)\to T(A)\otimes T(A)$ is the deconcatenation of words: 
\[ \Delta(a_1\dotsm a_n) = a_1\dotsm a_n\otimes \varepsilon  + \varepsilon \otimes a_1\dotsm a_n + \sum_{k=1}^{n-1}a_1\dotsm a_k\otimes a_{k+1}\dotsm a_n.\]
Equipped with the shuffle product and the deconcatenation coproduct, $ T(A) $  is a Hopf algebra whose antipode $ \CA $ is given by
\begin{equs}
	 \mathcal{M}_{\shuffle}\left( \CA \otimes \id \right) \Delta = \mathcal{M}_{\shuffle}\left( \id \otimes \CA \right) \Delta = \varepsilon \varepsilon^{*} 
\end{equs}
where  $ \mathcal{M}_{\shuffle}(u \otimes w)  = u \shuffle w$ and $ \varepsilon^* : T(A) \rightarrow \mathbb{R} $ is the counit satisfying $\varepsilon^*(u) = 1$ for $u= \varepsilon$ otherwise it is zero.

 The grading of $ T(A) $
 is given by the length of a word that is the number of its letters: $\ell(a_1\dotsm a_n) = n$. We denote by $ \CG_A $ the group of characters associated to $ T(A) $:
 \begin{equs}
 	\CG_A = \left\lbrace g \in (T(A))^*, \quad g( u \shuffle v ) = g(u) g(v)  \right\rbrace. 
 \end{equs}
The product associated to this group is   the convolution product $  *  $ given by
\begin{equs}
	g * f = \left( g \otimes f \right) \Delta
	\end{equs}
and the inverse is computed via the antipode $ \CA $
\begin{equs}
	g^{-1} = g(\CA \cdot).
\end{equs}
One main example of shuffle Hopf algebra is iterated integrals defined on the simplex $ 0 < t_1 < ... < t_n < t $. Given a smooth path $ t \mapsto X^a_t $ indexed by the letter $a \in A$, one can define the following map:
\begin{equs} \label{iterated_integrals}
	X_{st}(a_1 \cdots a_n) = \int_{s < t_1 < \cdots <t_n < t}  dX_{t_1}^{a_1} \cdots dX_{t_n}^{a_n}.
\end{equs}
One checks easily the following proposition
\begin{proposition} \label{character_property}
	The map $ X_{st} $ is a character on the shuffle Hopf algebra $T(A)$.
	\end{proposition}
The main idea of the proof is the fact that we can proceed with integration by parts. We explain this mechanism below for simple iterated integrals. Let's consider $a,b \in A$. Then, one has
\begin{equs}
	X_{st}(a) 	X_{st}(b) & = \int_{s}^t  dX_{r}^{a} \int_{s}^t  dX_{r}^{b}
	=  \int_{s}^t \int_{s}^t  dX_{r_1}^{a}   dX_{r_2}^{b}
	\\ & = \int_{s}^t \int_{s}^{r_2}  dX_{r_1}^{a}   dX_{r_2}^{b} + 
	\int_{s}^t \int_{r_2}^{t}  dX_{r_1}^{a}   dX_{r_2}^{b}.
\end{equs}  
By using integration by parts, one has
\begin{equs}
	\int_{s}^t \int_{r_2}^{t}  dX_{r_1}^{a}   dX_{r_2}^{b} & = -  X_{s}^{b}  \int_{s}^{t}  dX_{r_1}^{a} 	- \int_{s}^t   X_{r_2}^{b}  d \left( \int_{r_2}^{t} d X_{r_1}^{a} \right)   
	\\ & = -  X_{s}^{b}  \int_{s}^{t}  dX_{r_1}^{a} 	+ \int_{s}^t   X_{r_2}^{b}  d  X_{r_2}^{a} 
	\\ &=    \int_{s}^t \int_{s}^{r_2}  d X_{r_1}^{b}  d  X_{r_2}^{a}. 
\end{equs}
So in the end,
\begin{equs}
X_{st}(a) 	X_{st}(b) & =\int_{s}^t \int_{s}^{r_2}  dX_{r_1}^{a}   dX_{r_2}^{b} + 
\int_{s}^t \int_{s}^{r_2} dX_{r_1}^{b}  dX_{r_2}^{a}   
\\ & = 	X_{st}(ab) + 	X_{st}(ba) = 	X_{st}(a \shuffle b).
\end{equs} 
The integration by parts is possible because the path $ X^b$ is smooth. If $X^b$ does not satisfy the integration by parts, one can encode the product of iterated integrals via decorated trees. Therefore, we consider $ \CT^A $ the set of rooted trees with nodes decorated by $A$. 
We denote by 
$\CF^{A}$ the set of  forests composed of trees in $ \CT^A $. Any element $ F \in \CF^{A} $ can be decomposed into
\begin{equs}
	F = \tau_1 \cdot ... \cdot \tau_n
\end{equs} 
where the decorated trees $ \tau_i $ belong to $\CT^A$ and $ \cdot $ is the forest product, a commutative product.
  Any rooted tree $\tau \in \CT^A$, different from the empty tree $\one$, can be written in terms of the $B^{a}_+$-operators, $ a \in  A$. Indeed,  we have that $$\tau = B^{a}_+(\tau_1\cdot ... \cdot\tau_n)$$ where
  the operator $ B^{a}_+ $  connects the roots of the trees in the forest $\tau_1 \cdots \tau_n \in \CF^A$ to a new root decorated by $ a $.
We define $\CH^A = \langle \CF^A \rangle$ as the linear span of $\CF^A$. One can endow this vector space with a Hopf algebra structure where the product is given by the forest product. The coproduct is given by the Butcher-Connes-Kreimer coproduct defined recursively by
\begin{equs} \label{Connes-Kreimer}
	\Delta_{\text{\tiny{BCK}}} \tau 
	= \tau \otimes \mathbf{1}  + ( \mathrm{id} \otimes B^{a}_+)\Delta_{\text{\tiny{BCK}}} (\tau_1 \cdot ... \cdot \tau_n),
\end{equs}
and extended multiplicatively for the forest product.
One can also provide a non-recursive formula via admissible cuts given by
\begin{equs} \label{Connes_Kreimer_1}
\Delta_{\tiny{\text{BCK}}} \tau =  \tau \otimes \mathbf{1} + \sum_{c \in \scriptsize{\text{Adm}}(h) }   P^c(\tau) \otimes R^c(\tau).
\end{equs} 
Here, we have used $P^c(h)$ to denote the pruned forest that is formed by collecting all the edges at or above the cut. The term $R^c(\tau)$  corresponds to the "trunk", that is the subforest formed by the edges not lying above the ones upon which the cut was performed.   
In the sequel, we use the notation $ \CH^A_{\text{\tiny{BCK}}} $ when one equips $ \CH^A $ with the Butcher-Connes-Kreimer Hopf algebra structure. We consider an important product connected to this Hopf algebra. We define the grafting product $ \curvearrowright $ as
\begin{equation}
	\label{grafting_a}
	\sigma \curvearrowright \tau:=\sum_{v\in  N_{\tau} } \sigma \curvearrowright_v  \tau,
\end{equation}
where  $\sigma $ and $\tau$ are two decorated rooted trees, $ N_{\tau} $ is the set of nodes of $ \tau $ and where $\sigma \curvearrowright_v \tau$ is obtained by grafting the tree $\sigma$ on the tree $\tau$ at vertex $v$ by means of a new edge. We extend $ \curvearrowright $ to the empty tree $ \mathbf{1} $ by setting
\begin{equs}
	\sigma \curvearrowright \mathbf{1} = \mathbf{1}  \curvearrowright \sigma.
\end{equs}
 We will need to use the adjoint map of the grafting denoted by $ \curvearrowright^* $ and given by
\begin{equs}
\left\langle 	\curvearrowright^* \tau, \tau_1 \otimes \tau_2 \right\rangle  := \left\langle 	 \tau, \tau_1 \curvearrowright \tau_2 \right\rangle
\end{equs}
where the inner product is defined by
\begin{equs}
	\left\langle  \sigma, \tau \right\rangle = \delta_{\sigma, \tau} S(\tau).
\end{equs}
Here, $ \delta_{\sigma, \tau} $ is equal to one only if $ \sigma = \tau $ otherwise it is equal to zero. The symmetry factor $  S(\tau) $ is given for $ \sigma = B_+^a(\tau) $ and $ \tau = \tau_1^{\beta_1}\cdot ... \cdot\tau_n^{\beta_1} $ by
\begin{equs}
S(\sigma) = S(\tau), \quad	S(\tau) = \prod_{i=1}^n (\beta_i!) S(\tau_i)^{\beta_i}
\end{equs}
where the $\tau_i$ are pairwise disjoint and the  $ \tau_i^{\beta_i} $ correspond to the forests with the decorated tree $\tau_i$ repeated $ \beta_i $ times. The factor $S(\tau)$ is the number of automorphisms preserving the structure of the forest $ \tau $.
Then, one defines a map for moving from trees to words called Arborification. The process of
arborification is given by a surjective Hopf algebra morphism from
the Butcher–Connes–Kreimer Hopf algebra $ \mathcal{H}^A_{\text{\tiny{BCK}}} $
onto the
shuffle Hopf algebra 
defined over the alphabet $A$. 
The arborification morphism $  \mathfrak{a} :  \mathcal{H}^A_{\text{\tiny{BCK}}} \rightarrow T(A)$ 
is defined by
\begin{equs} \label{def_arborification}
	\mathfrak{a}(  B_+^a(\tau_1,...,\tau_n) )
	=  \left(  	\mathfrak{a}(\tau_1 ) \shuffle \cdots \shuffle 	\mathfrak{a}(\tau_n) \right) a
\end{equs}
As the arborification is a Hopf morphism, one has 
\begin{equs} \label{Hopf_morphism}
	\left( \mathfrak{a} \otimes \mathfrak{a} \right)	\Delta_{\text{\tiny{BCK}}} = \Delta \mathfrak{a}.
\end{equs}
Moreover, one has
\begin{equs}
	\mathfrak{a}(\sigma \cdot \tau ) = \mathfrak{a}(\sigma) \shuffle \mathfrak{a}(\tau).
\end{equs}
In the next proposition, we prove an identity of the type \eqref{Hopf_morphism} but for $\curvearrowright^*$. 
\begin{proposition} \label{identi_c}
	One has the following identity
	\begin{equs}
\left( P_{\bullet} \otimes \mathfrak{a} \right)	\curvearrowright^* = 	\Delta^{\!1}_c \mathfrak{a} 
	\end{equs}
where  $ P_{\bullet} $ is the projection onto the vector space generated by elements of the form $ \bullet_a $, $a \in A$ and sends $ \bullet_a $ to the letter $a$. The map $ \Delta^{\!1}_c $ is given by
\begin{equs}
	\Delta^{\!1}_c au = a \otimes u.
\end{equs}
	\end{proposition}
\begin{proof}
We proceed by induction on the size of $\tau$.	The property is obvious for $\tau = \bullet_a$.
	One first observes that the map $ \curvearrowright^* $ as the following recursive formula
	\begin{equs}
 \curvearrowright^{*} \tau = \sum_{j=1}^n \left( \id  \otimes B_+^a(\tau_1,...,\tau_{j-1}, \cdot, \tau_{j+1},\tau_n)  \right) \curvearrowright^{*} \tau_j + \tau \otimes \one.
	\end{equs}
This formula means that we perform a cut with $ \curvearrowright^{*} $ and therefore we select a branch inside one of the $\tau_j$.
Then,  one obtains
\begin{equs}
	&\left( P_{\bullet} \otimes \mathfrak{a} \right)	\curvearrowright^* \tau   = \sum_{j=1}^n \left(  P_{\bullet}  \otimes \mathfrak{a}( B_+^a(\tau_1,...,\tau_{j-1}, \cdot, \tau_{j+1},\tau_n) ) \right) \curvearrowright^{*} \tau_j
	\\ & =  \sum_{j=1}^n \left( \id \otimes   (\mathfrak{a}(\tau_1) \shuffle ... \shuffle \mathfrak{a}(\tau_{j-1}) \shuffle \cdot \shuffle \mathfrak{a}(\tau_{j+1}) \shuffle  \mathfrak{a}(\tau_n)  )a \right) \left( P_{\bullet} \otimes  \mathfrak{a} \right)\curvearrowright^{*} \tau_j.
	\end{equs}
We apply the induction hypothesis to $ \tau_j $
\begin{equs}
	\left( P_{\bullet} \otimes \mathfrak{a} \right) 	\curvearrowright^*\tau_j = 	\Delta^{\!1}_c \mathfrak{a}(\tau_j).
	\end{equs}
We conclude from the fact that
\begin{equs}
\Delta^{\!1}_c \mathfrak{a}(\tau) =	\sum_{j=1}^n \left( \id \otimes  (\mathfrak{a}(\tau_1) \shuffle ... \shuffle \mathfrak{a}(\tau_{j-1}) \shuffle \cdot \shuffle \mathfrak{a}(\tau_{j+1})  \shuffle  \mathfrak{a}(\tau_n)  )a \right) 	\Delta^{\!1}_c \mathfrak{a}(\tau_j).
\end{equs}
Indeed, $ \Delta^{\!1}_c  $ selects a letter among the different blocks $\mathfrak{a}(\tau_i)$.
	\end{proof}

We want to provide an alternative definition of the arborification in the next proposition. This definition is crucial for the sequel.
 
\begin{proposition} \label{alternative_arb}
	The arborification morphism $  \mathfrak{a} :  \mathcal{H}^A_{\text{\tiny{BCK}}}  \rightarrow T(A)$ 
satisfies the following identity
	\begin{equs} \label{def_arborification_2}
		\mathfrak{a}(  \tau )
		= \mathcal{M}_{\tiny{\text{c}}}\left(  P_{\bullet} \otimes \mathfrak{a}   \right) \curvearrowright^{*} \tau.
	\end{equs}
where $ \mathcal{M}_{\tiny{\text{c}}}(u \otimes v)  = uv$ with $ u,v \in  T(A) $.
	`\end{proposition}
\begin{proof} 
	From Proposition~\ref{identi_c}, one has
	\begin{equs}
		\mathcal{M}_{\tiny{\text{c}}}\left(  P_{\bullet} \otimes \mathfrak{a}   \right) \curvearrowright^{*} \tau = \mathcal{M}_c \Delta^{\!1}_c \mathfrak{a}(\tau).
	\end{equs}
We conclude from the identity
\begin{equs}
	\mathcal{M}_c \Delta^{\!1}_c = \id.
	\end{equs}
	\end{proof}

The characterisation \eqref{def_arborification_2} corresponds to start the word with the left most letter. At the level of the decorated tree, it corresponds to cut one leaf that will be that letter. This is different from the original definition of the arborification which starts from the right most letter. A coproduct version of \eqref{def_arborification} is given by
\begin{equs} \label{arb}
		\mathfrak{a}(  \tau )
	= \mathcal{M}_{\tiny{\text{c}}}\left( \mathfrak{a}  \otimes P_{\bullet}   \right) \Delta_{\text{\tiny{BCK}}}  \tau,
\end{equs} 
as $  P_{\bullet} $ forces the trunk to be formed of only one node. Together with \eqref{def_arborification_2}, one has the following identity:
\begin{equs} \label{new_identity_arb}
	\mathcal{M}_{\tiny{\text{c}}}\left( \mathfrak{a}  \otimes P_{\bullet}   \right) \Delta_{\text{\tiny{BCK}}} = \mathcal{M}_{\tiny{\text{c}}}\left(  P_{\bullet} \otimes \mathfrak{a}   \right) \curvearrowright^{*} 
\end{equs}
which seems absent from the literature. This identity can be interpreted as when one flips $ \mathfrak{a} $ and $ P_{\bullet} $, one has to replace $ \Delta_{\text{\tiny{BCK}}} $ with $\curvearrowright^{*} $ which is a move from several cuts to a single cut.

In fact, the definition \eqref{arb} can be adapted to produce more morphisms between  $\mathcal{H}^A_{\text{\tiny{BCK}}}$ and   $T(B)$ where $B$ is some alphabet. One of these known morphisms is the Hairer-Kelly map $ \Psi_{\text{\tiny{HK}}}  $ introduced for moving from geometric to branched Rough Paths  in \cite{HK15}. 
  This map has been written in \cite[Def. 4, Sec. 6]{Br173} as the unique Hopf algebra morphism from $\mathcal{H}^A_{\text{\tiny{BCK}}}$ to the shuffle Hopf algebra $T(\CT^A)$ obeying: 
\begin{equs} \label{Hairer_Kelly}	\Psi_{\text{\tiny{HK}}} = 	\mathcal{M}_{\tiny{\text{c}}}(\Psi_{\text{\tiny{HK}}} \otimes P_{\one}) \Delta_{\text{\tiny{BCK}}} 
\end{equs}
where  $P_{\one}:=\mathrm{id}-\one^{*}$ is the augmentation projector, here $ \one^{*}(\one)  =1 $ and zero otherwise. This projector guarantees that the trunk is not empty and that one branch is cut. The main difference between the arborification and the Hairer-Kelly map is the size of the alphabet where one replaces $ A$ by $ \CT^A $. A similar proof as for the arborification allows to rewrite \eqref{Hairer_Kelly} into
\begin{equs}
	\Psi_{\text{\tiny{HK}}} =  	\mathcal{M}_{\tiny{\text{c}}}(  P_{\one} \otimes\Psi_{\text{\tiny{HK}}}) \curvearrowright^{*}.
\end{equs}

\section{Fourier decorated trees}

\label{Sec::3}

We recall the decorated trees introduced in \cite{BS}. Let  $\Lab$ a finite set  and  frequencies $ k_1,...,k_m \in \mathbb{Z}^{d}$. Here, $ \mathfrak{L} $ parametrises a family of polynomials $ (P_{\Labhom})_{\Labhom \in \Lab} $ in the frequencies that represent operators with constant coefficients in Fourier space.  
\begin{definition}		\label{decorated_trees}
	We define the set of decorated trees $\mcT$ as elements of the form $T_\mfe^{\mff}$ where 	
	\begin{itemize}	
		\item $ T $ is a non-planar rooted tree with root node $ \varrho $, edge set $ E_T $ and node set $ N_T $. We consider only planted trees which means that $ T $ has only one edge connecting its root $ \varrho $.		
		\item $ \mfe:E_T\rightarrow \mfL\times \{0,1 \} $ is the set of edge decorations $ \mfe(\cdot) = (\mft(\cdot), \mfp(\cdot)) $ where the first component selects the correct polynomial $ P_{\mft} $ when $ \mft \in \Lab $.		
		\item $ \mff:N_T\setminus\{\varrho\}\rightarrow \mathbb{Z}^d $ are node decorations  satisfying the relation for every inner node $ u $
		\begin{equs} \label{frequencies_identity}
			(-1)^{\mfp(e_u)}\mff(u)=\sum_{e = (u,v)\in E_T}(-1)^{\mfp(e)}\mff(v)
		\end{equs}
		where $ e_u$ is the edge outgoing $u$ of the form $ (w,u) $. From this definition, one
		can see that the node decorations  $(\mff(u))_{u \in L_T}$ determine the decorations of the inner nodes. We assume
		that the node decorations at the leaves are linear combinations of the $k_i$ with coefficients in $ \lbrace -1,0,1 \rbrace$.
	\end{itemize}
\end{definition}

We define $ \CH $ as the linear span of $H$ the forests containing decorated trees in $ \CT $. The empty forest is denoted by $ \one $.
Before presenting the iterated integrals associated to these trees, we introduce a symbolic notation. An  edge decorated by $ o = (\mft, \mfp) $ with $ \mft \in \Lab $ is denoted by $ \mathcal{I}_{o} $. The symbol $ \mathcal{I}_{o} (\lambda_{k}
\cdot) : \CH \rightarrow \CH $ is viewed as the operation that merges all the roots of the trees composing the forest into one node decorated by $ k \in \mathbb{Z}^d$. We obtain
a decorated tree which is then grafted onto a new root with no decoration. If the condition \eqref{frequencies_identity} is not
satisfied on the argument, then $ \mathcal{I}_{o} (\lambda_{k}
\cdot) $  gives zero.
Below, we illustrate the symbolic notation on various examples. We suppose that $ \mfL = \{\mft_1,\mft_2\}$ and  $P_{\mathfrak{t}_1}(\lambda) = -\lambda^2$ and $P_{\mathfrak{t}_2}(\lambda) = \lambda^2$. Then, one has 
\begin{equs} \label{exemple_1}
	T =  \begin{tikzpicture}[scale=0.2,baseline=-5]
		\coordinate (root) at (0,0);
		\coordinate (tri) at (0,-2);
		\coordinate (t1) at (-2,2);
		\coordinate (t2) at (2,2);
		\coordinate (t3) at (0,3);
		\draw[kernels2,tinydots] (t1) -- (root);
		\draw[kernels2] (t2) -- (root);
		\draw[kernels2] (t3) -- (root);
		\draw[symbols] (root) -- (tri);
		\node[not] (rootnode) at (root) {};t
		\node[not] (trinode) at (tri) {};
		\node[var] (rootnode) at (t1) {\tiny{$ k_{\tiny{1}} $}};
		\node[var] (rootnode) at (t3) {\tiny{$ k_{\tiny{2}} $}};
		\node[var] (trinode) at (t2) {\tiny{$ k_3 $}};
	\end{tikzpicture}
=	  \CI_{(\mft_2,0)}(\lambda_{k}\CI_{(\mft_1,1)}(\lambda_{k_1})\CI_{(\mft_1,0)}(\lambda_{k_2}\CI_{(\mft_1,0)}(\lambda_{k_3})
\end{equs}
where $ k = -k_1 + k_2 + k_3 $. A blue edge encodes $ (\mathfrak{t}_2,0) $, a brown edge is used for $ (\mathfrak{t}_1,0) $ and a dashed brown edge is for $ (\mathfrak{t}_2,1) $. 
The frequency decorations appear on the leaves of the previous tree. One does not have to make explicit the frequency decoration  for the inner nodes as they are determined by the ones coming from the leaves.  We have also omitted the forest product $ \cdot $ as $ \CI_{(\mft_1,1)}(\lambda_{k_1})\CI_{(\mft_1,0)}(\lambda_{k_2}) $ is a shorthand notation for $ \CI_{(\mft_1,1)}(\lambda_{k_1}) \cdot \CI_{(\mft_1,0)}(\lambda_{k_2}) $.

\begin{definition} \label{dom_freq} We recursively define $\mathscr{F} :  H \rightarrow \mathbb{R}[\mathbb{Z}^d]$ as:
	\begin{equs}
\,	&	\mathscr{F}(\one)   = 0, \quad
		\mathscr{F}
		(F \cdot \bar F)  =\mathscr{F}(F) + \mathscr{F}(\bar F), \\
	&	\mathscr{F}\left( \CI_{(\Labhom,\Labp)}(  \lambda_{k}F) \right)   =     P_{(\Labhom,\Labp)}(k) +\mathscr{F}(F)  
	\end{equs}
where one has $
	P_{(\Labhom,\Labp)}(k) = (-1)^{\Labp} P_{\Labhom}((-1)^{\Labp}k).$
\end{definition}

We continue the example given in \eqref{exemple_1} and we compute $ \mathscr{F}(T) $:
\begin{equs}
	\mathscr{F}(T) & =  (-k_1+k_2+k_3)^{2} + (-k_1)^{2} - k_2^{2} - k_3^2.  
\end{equs}
Given a decorated tree  $ T_\mfe^{\mff} $, we define the order of a tree denoted by $ \vert \cdot \vert_{\text{\tiny{ord}}} $
by
\begin{equation*}
	\vert  T_\mfe^{\mff} \vert_{\text{\tiny{ord}}} = \sum_{e \in E_T }   \one_{\lbrace \mathfrak{t}(e) \in \Lab_{2} \rbrace}.
\end{equation*} 
which corresponds to the number of blue edges in a tree. These edges are associated to a subset $  \Lab_{2} \subset  \Lab $. In the sequel, we suppose that we have always $ \mfL = \{\mft_1,\mft_2\}$ and $ \Lab_{2} = \{\mft_2\} $. We denote also by $ o_i $ edge decorations of the form $(\mft_i,p_i)$ for $i \in \{1,2\}$.  For example in \eqref{exemple_1}, one has $ \vert  T \vert_{\text{\tiny{ord}}} = 1 $.
We define recursively the following set of decorated trees:
\begin{equs}
	& \CT_0 = \lbrace \CI_{(\Labhom_1,0)}( \lambda_k \CI_{(\Labhom_2,0)}( \lambda_k    T_1 \cdot T_2 \cdot  \tilde T_3  ) ), \CI_{(\Labhom_1,0)}(\lambda_k) \,  \\ & : \,T_1, T_2 \in \CT_0, \, \tilde{T}_3 \in \CT_1, \, k \in \mathbb{Z}^{d}  \rbrace \\
	& \CT_1  = \lbrace \CI_{(\Labhom_1,1)}( \lambda_k \CI_{(\Labhom_2,1)}( \lambda_k   T_1 \cdot T_2  \cdot \tilde T_3  ) ), \CI_{(\Labhom_1,1)}(\lambda_k) \,  \\ & \quad : \,T_1, T_2 \in \CT_1, \, \tilde{T}_3 \in \CT_0, \, k \in \mathbb{Z}^{d}  \rbrace.
\end{equs}
We define $   \CT_0^k  $ as the subspace of $ \CT_0 $ such that the frequency decoration on the nodes connected to the root is $ k $. For $ r\in\mathbb{N} $ and $ j \in \lbrace 0,1 \rbrace $ we set
\begin{equs} \label{space_deco}
	\mcT^{\leq r,k}_j= \cup_{m=0}^r \mcT^{m,k}_j, \quad \mcT^{m,k}_j=\{T_\mfe^\mff\in \mcT^k_j ,|T^\mff_\mfe|_{\text{\tiny{ord}}} =  m \}.
\end{equs}

In the sequel, we will also consider decorated trees that start with an edge decorated by $ \mathfrak{t}_2 $. They are given by
\begin{equs}
	& \hat{\CT}_0 = \lbrace  \CI_{(\Labhom_2,0)}( \lambda_k    T_1 \cdot T_2 \cdot  \tilde T_3  ) \in \CT_{\text{\tiny{non-reso}}} \, : \,T_1, T_2 \in \CT_0, \, \tilde{T}_3 \in \CT_1, \, k \in \mathbb{Z}^{d}  \rbrace \\
	& \hat{\CT}_1  = \lbrace  \CI_{(\Labhom_2,1)}( \lambda_k   T_1 \cdot T_2  \cdot \tilde T_3   ) \, \in \CT_{\text{\tiny{non-reso}}} : \,T_1, T_2 \in \CT_1, \, \tilde{T}_3 \in \CT_0, \, k \in \mathbb{Z}^{d}  \rbrace.
\end{equs}
where $ \CT_{\text{\tiny{non-reso}}} $ is the set of non-resonant trees. These trees are decorated trees satisfying the following extra conditions:
 For every inner node $ u $, one has
 \begin{equs} \label{reso_condition}
 	P_{(\Labhom(e_u),\Labp(e_u))}(\mff(u)) +	 \sum_{e = (u,v)\in E_T} P_{(\Labhom(e),\Labp(e))}(\mff(v))
 	\neq 0.
 \end{equs}
which correspond to a non-resonant condition.
As the same as before, we define the spaces
$ 	\hat{\mcT}^{\leq m,k}_j $ and $ \hat{\mcT}^{m,k}_j $ from 
$ \hat{\mcT}^{k}_j $. We also introduce the notation $ \hat{\CT}^{(m)}_1$ for the decorated trees in $ \hat{\CT}^{m}_j $ with $m$ blue edges but with no specific constraint on the frequency decoration at the node connected to the root.  

	We will now introduce the elementary differentials, which are required to present the expansion of dispersive equations solutions in terms of iterated integrals. We start by defining the symmetry factor of the tree $ T^{\mff}_{\mfe}\in\mcT $ by considering only the edge decoration, i.e, setting $ T^{\mff}_\mfe := T_\mfe $ and then setting $ S(\boldsymbol{1}) = 1 $ and working inductively defining 
\begin{equs} \label{symmetry_factor}
	S(T):=\prod_{i,j}S(T_{i,j})^{\gamma_{i,j}}\gamma_{i,j}!
\end{equs}
for the tree $ \prod_{i,j}\CI_{(\mft_{t_i}, p_i)}(T_{i,j})^{\gamma_{i,j}} $ where $T_{i,j} \neq T_{i,\ell}$ for $j \neq \ell$.
We now define the elementary differentials denoted by $ \Upsilon(T)(v) $ from two given polynomials $ p_0(v,\bar{v}) $ $ p_1(v,\bar{v}) $ as 
\begin{equs}
	T  & = 
	\CI_{(\Labhom_1,a)}\left( \lambda_k \CI_{(\Labhom_2,a)}( \lambda_k   \prod_{i=1}^n \CI_{(\Labhom_1,0)}( \lambda_{k_i} T_i) \prod_{j=1}^m \CI_{(\Labhom_1,1)}( \lambda_{\tilde k_j} \tilde T_j)  ) \right), \quad a \in \lbrace 0,1 \rbrace  
\end{equs}
by
\begin{equs} \label{def_Upsilon}
	\begin{aligned}
	\Upsilon(T)(v) \, & { :=}  \partial_v^{n} \partial_{\bar v}^{m} p_a(v,\bar v) \prod_{i=1}^n  \Upsilon( \CI_{(\Labhom_1,0)}\left( \lambda_{k_i}  T_i \right) )(v)  \\ & \prod_{j=1}^m \Upsilon( \CI_{(\Labhom_1,1)}( \lambda_{\tilde k_j}\tilde T_j ) )(v)
	\end{aligned}
\end{equs}
where 
\begin{equs} 
	\label{mid_point_rule}
	\begin{aligned}
		\Upsilon(\CI_{(\Labhom_1,0)}( \lambda_{k})  )(v)  \,  { :=}  v_k, \quad 
		\Upsilon(\CI_{(\Labhom_1,1)}( \lambda_{k})  )(v)  \, & { :=}  \bar{v}_k.
	\end{aligned}
\end{equs}
Above, we have used the notation:
\begin{equs}
	p_0(v,\bar{v}) = p(v,\bar{v}), \quad p_{1}(v,\bar{v}) = \overline{p(v, \bar{v})}. 
\end{equs}
The definition is similar if $ T$ has a similar form  starting by $	 \CI_{(\Labhom_2,a)}( \lambda_k \cdot)$ instead of $	\CI_{(\Labhom_1,a)}\left( \lambda_k \CI_{(\Labhom_2,a)}( \lambda_k \cdot) \right)$

 One can define the iterated integrals recursively based on the inductive construction of the decorated trees via a map $ \Pi :  \CH \rightarrow \mathcal{C} $. The space $  \mathcal{C} $ contains  functions of the form $ z \mapsto \sum_j Q_j(z) e^{i z P_j(k_1,...,k_n)} $  and the $ Q_j(z) $ are polynomials in
$ z $, the $ P_j $ are polynomials in  $ k_1,...,k_n $, and each $ Q_j $ depends implicitly on $k_i$. 
The map $ \Pi $ is defined for every forest $ F $ by
\begin{equs}
	\label{def_Pi}
	\begin{aligned}
		\Pi (\mathcal{I}_{o_1}(\lambda_{k} F))(t) &= e^{i t  P_{o_1}(k)}  (\Pi \bar{F})(t) \\
		\Pi (\mathcal{I}_{o_2}(\lambda_{k} F))(t) &= - i |\nabla|^\alpha(k)\int_0^t e^{i s P_{o_2}(k)} \Pi(F)(s)d s 
	\end{aligned}
\end{equs} 
for some fixed differential operator $ |\nabla|^\alpha(k) $ in Fourier space.
From \cite[Prop. 4.3]{BS},  the previous iterated integrals are used for expanding the following Duhamel's formula:
\begin{equs}
	u(t) = e^{i t \Delta} u(0) - i e^{it \Delta} \int_0^t e^{- i s \Delta} \left( \bar{u}(s) u(s)^2 \right) ds
\end{equs}
which is given in Fourier space as
\begin{equs} \label{duhamel_Fourier}	
	\begin{aligned}
			u_k(t) & = e^{-i t k^2} u_k(0) \\ &  - i \sum_{k=-k_1 + k_2 + k_3} e^{-it k^2} \int_0^t e^{ i s k^2} \left( \bar{u}_{k_1}(s) u_{k_2}(s) u_{k_3}(s) \right) ds
			\end{aligned}
\end{equs}
where the operator $ e^{it \Delta} $ is sent to $ e^{-it k^2} $ in Fourier space. Also, pointwise products are sent to convolution products on the frequencies. Here, one has
\begin{equs}
	p(v,\bar{v}) = v^2 \bar{v}, \quad \alpha =0.
\end{equs}
We start an expansion by:
\begin{equs}
		u_k(t) = e^{-i t k^2} u_k(0) + \mathcal{O}(t)
\end{equs}
and then we obtain
\begin{equs}
	\begin{aligned}
		u_k(t) & = e^{-i t k^2} u_k(0) \\ &  - i \sum_{k=-k_1 + k_2 + k_3} e^{-it k^2} \int_0^t e^{ i s (k^2+k_1^2-k_2^2-k_3^2)} \left( \bar{u}_{k_1}(0) u_{k_2}(0) u_{k_3}(0) \right) ds + \mathcal{O}(t^2)
	\end{aligned}
\end{equs}
In the next proposition, we recall \cite[Prop. 4.3]{BS} that shows how to encode iterations of Duhamel's formula with the decorated trees previously introduced.
\begin{proposition} \label{tree_series}
	The tree series given by 
	\begin{equs}\label{genscheme}
		U_{k}^{r}(t, v) =   \sum_{T \in \CT^{\leq r,k}_{0}} \frac{\Upsilon( T)(v)}{S(T)} (\Pi   T )(t)
	\end{equs}
	is the $k$th Fourier coefficient of a solution of $ \eqref{duhamel_Fourier} $ up to order $ r $ which means that one has
	\begin{equs}
		u_{k}(t) - 	U_{k}^{r}(t, v) = \mathcal{O}(t^{r+1}).
	\end{equs}
\end{proposition}
As an example, we can list the elements for $ \CT_0^{\leq 1,k} $ and $ \CT_0^{\leq 2,k} $ below:
\begin{equs}
	\CT^{\leq 1,k}_{0} & = \left\lbrace T_0, T_1, \, k_i \in \mathbb{Z}^d \right\rbrace, \quad T_0 =  \begin{tikzpicture}[scale=0.2,baseline=-5]
		\coordinate (root) at (0,1);
		\coordinate (tri) at (0,-1);
		\draw[kernels2] (tri) -- (root);
		\node[var] (rootnode) at (root) {\tiny{$ k $}};
		\node[not] (trinode) at (tri) {};
	\end{tikzpicture} , \quad T_1 =  \begin{tikzpicture}[scale=0.2,baseline=-5]
		\coordinate (root) at (0,2);
		\coordinate (tri) at (0,0);
		\coordinate (trib) at (0,-2);
		\coordinate (t1) at (-2,4);
		\coordinate (t2) at (2,4);
		\coordinate (t3) at (0,5);
		\draw[kernels2,tinydots] (t1) -- (root);
		\draw[kernels2] (t2) -- (root);
		\draw[kernels2] (t3) -- (root);
		\draw[kernels2] (trib) -- (tri);
		\draw[symbols] (root) -- (tri);
		\node[not] (rootnode) at (root) {};
		\node[not] (trinode) at (tri) {};
		\node[var] (rootnode) at (t1) {\tiny{$ k_{\tiny{1}} $}};
		\node[var] (rootnode) at (t3) {\tiny{$ k_{\tiny{2}} $}};
		\node[var] (trinode) at (t2) {\tiny{$ k_{\tiny{3}} $}};
		\node[not] (trinode) at (trib) {};
	\end{tikzpicture} 
\\
	\CT^{\leq 2,k}_{0} & = \left\lbrace T_0,T_1,T_2,T_3, \, k_i \in \mathbb{Z}^d \right\rbrace, \quad T_2 = \begin{tikzpicture}[scale=0.2,baseline=-5]
		\coordinate (root) at (0,2);
		\coordinate (tri) at (0,0);
		\coordinate (trib) at (0,-2);
		\coordinate (t1) at (-2,4);
		\coordinate (t2) at (2,4);
		\coordinate (t3) at (0,4);
		\coordinate (t4) at (0,6);
		\coordinate (t41) at (-2,8);
		\coordinate (t42) at (2,8);
		\coordinate (t43) at (0,10);
		\draw[kernels2,tinydots] (t1) -- (root);
		\draw[kernels2] (t2) -- (root);
		\draw[kernels2] (t3) -- (root);
		\draw[symbols] (root) -- (tri);
		\draw[symbols] (t3) -- (t4);
		\draw[kernels2,tinydots] (t4) -- (t41);
		\draw[kernels2] (t4) -- (t42);
		\draw[kernels2] (t4) -- (t43);
		\draw[kernels2] (trib) -- (tri);
		\node[not] (trinode) at (trib) {};
		\node[not] (rootnode) at (root) {};
		\node[not] (rootnode) at (t4) {};
		\node[not] (rootnode) at (t3) {};
		\node[not] (trinode) at (tri) {};
		\node[var] (rootnode) at (t1) {\tiny{$ k_{\tiny{4}} $}};
		\node[var] (rootnode) at (t41) {\tiny{$ k_{\tiny{1}} $}};
		\node[var] (rootnode) at (t42) {\tiny{$ k_{\tiny{3}} $}};
		\node[var] (rootnode) at (t43) {\tiny{$ k_{\tiny{2}} $}};
		\node[var] (trinode) at (t2) {\tiny{$ k_5 $}};
	\end{tikzpicture}, \quad T_3 = \begin{tikzpicture}[scale=0.2,baseline=-5]
		\coordinate (root) at (0,2);
		\coordinate (tri) at (0,0);
		\coordinate (trib) at (0,-2);
		\coordinate (t1) at (-2,4);
		\coordinate (t2) at (2,4);
		\coordinate (t3) at (0,4);
		\coordinate (t4) at (0,6);
		\coordinate (t41) at (-2,8);
		\coordinate (t42) at (2,8);
		\coordinate (t43) at (0,10);
		\draw[kernels2] (t1) -- (root);
		\draw[kernels2] (t2) -- (root);
		\draw[kernels2,tinydots] (t3) -- (root);
		\draw[symbols] (root) -- (tri);
		\draw[symbols,tinydots] (t3) -- (t4);
		\draw[kernels2] (t4) -- (t41);
		\draw[kernels2,tinydots] (t4) -- (t42);
		\draw[kernels2,tinydots] (t4) -- (t43);
		\draw[kernels2] (trib) -- (tri);
		\node[not] (trinode) at (trib) {};
		\node[not] (rootnode) at (root) {};
		\node[not] (rootnode) at (t4) {};
		\node[not] (rootnode) at (t3) {};
		\node[not] (trinode) at (tri) {};
		\node[var] (rootnode) at (t1) {\tiny{$ k_{\tiny{4}} $}};
		\node[var] (rootnode) at (t41) {\tiny{$ k_{\tiny{1}} $}};
		\node[var] (rootnode) at (t42) {\tiny{$ k_{\tiny{3}} $}};
		\node[var] (rootnode) at (t43) {\tiny{$ k_{\tiny{2}} $}};
		\node[var] (trinode) at (t2) {\tiny{$ k_5 $}};
	\end{tikzpicture}.
\end{equs}

Then, one can compute the various components of the Butcher-type series given in \eqref{genscheme}. For example, the symmetry factor $S$ is given by:
\begin{equs}
	S(T_0) = 1, \quad S(T_1) = S(T_2) = 2, \quad S(T_3) = 4.
	\end{equs}
Let us stress that the symmetry factor does not take into account the frequency decorations $k_i$ but only the edge decorations. For the elementary differentials, one has
\begin{equs}
	\Upsilon[T_0](v) = v_k, \quad \Upsilon[T_1](v) = 2\bar{v}_{k_1} v_{k_2} v_{k_3}, \quad \Upsilon[T_3](v) = 4 \bar{v}_{k_1} v_{k_2} v_{k_3} \bar{v}_{k_4} v_{k_5}  ,
\end{equs}
where the factors $2$ and $4$ come from the derivation of the monomial $ u^2 $. For the iterated integrals, one has recursively: 
\begin{equs}
	(\Pi  \begin{tikzpicture}[scale=0.2,baseline=-5]
		\coordinate (root) at (0,1);
		\coordinate (tri) at (0,-1);
		\draw[kernels2] (tri) -- (root);
		\node[var] (rootnode) at (root) {\tiny{$ k_2 $}};
		\node[not] (trinode) at (tri) {};
	\end{tikzpicture}) (t) & = e^{-i t k_2^2}, \quad (\Pi  \begin{tikzpicture}[scale=0.2,baseline=-5]
		\coordinate (root) at (0,1);
		\coordinate (tri) at (0,-1);
		\draw[kernels2,tinydots] (tri) -- (root);
		\node[var] (rootnode) at (root) {\tiny{$ k_1 $}};
		\node[not] (trinode) at (tri) {};
	\end{tikzpicture}) (t) = e^{i t k_1^2}, \quad 
	(\Pi \begin{tikzpicture}[scale=0.2,baseline=-5]
		\coordinate (root) at (0,-1);
		\coordinate (t1) at (-2,1);
		\coordinate (t2) at (2,1);
		\coordinate (t3) at (0,2);
		\draw[kernels2,tinydots] (t1) -- (root);
		\draw[kernels2] (t2) -- (root);
		\draw[kernels2] (t3) -- (root);
		\node[not] (rootnode) at (root) {};t
		\node[var] (rootnode) at (t1) {\tiny{$ k_{\tiny{1}} $}};
		\node[var] (rootnode) at (t3) {\tiny{$ k_{\tiny{2}} $}};
		\node[var] (trinode) at (t2) {\tiny{$ k_3 $}};
	\end{tikzpicture}  )(t) = e^{i t (k_1^2 - k_2^2 - k_3^2)},
	\\ ( \Pi \begin{tikzpicture}[scale=0.2,baseline=-5]
		\coordinate (root) at (0,0);
		\coordinate (tri) at (0,-2);
		\coordinate (t1) at (-2,2);
		\coordinate (t2) at (2,2);
		\coordinate (t3) at (0,3);
		\draw[kernels2,tinydots] (t1) -- (root);
		\draw[kernels2] (t2) -- (root);
		\draw[kernels2] (t3) -- (root);
		\draw[symbols] (root) -- (tri);
		\node[not] (rootnode) at (root) {};t
		\node[not] (trinode) at (tri) {};
		\node[var] (rootnode) at (t1) {\tiny{$ k_{\tiny{1}} $}};
		\node[var] (rootnode) at (t3) {\tiny{$ k_{\tiny{2}} $}};
		\node[var] (trinode) at (t2) {\tiny{$ k_3 $}};
	\end{tikzpicture}) (t) & = -i \int^{t}_0 e^{is (-k_1 + k_2 + k_3)^2} e^{i s (k_1^2 - k_2^2 - k_3^2)} ds.
\end{equs}
with these computations, one has
\begin{equs}
	(\Pi T_0)(t) & = e^{-it k^2}, \\	(\Pi T_1)(t) & = e^{-it k^2} (\Pi \begin{tikzpicture}[scale=0.2,baseline=-5]
		\coordinate (root) at (0,0);
		\coordinate (tri) at (0,-2);
		\coordinate (t1) at (-2,2);
		\coordinate (t2) at (2,2);
		\coordinate (t3) at (0,3);
		\draw[kernels2,tinydots] (t1) -- (root);
		\draw[kernels2] (t2) -- (root);
		\draw[kernels2] (t3) -- (root);
		\draw[symbols] (root) -- (tri);
		\node[not] (rootnode) at (root) {};t
		\node[not,label= {[label distance=-0.2em]below: \scriptsize  $  $}] (trinode) at (tri) {};
		\node[var] (rootnode) at (t1) {\tiny{$ k_{\tiny{1}} $}};
		\node[var] (rootnode) at (t3) {\tiny{$ k_{\tiny{2}} $}};
		\node[var] (trinode) at (t2) {\tiny{$ k_3 $}};
	\end{tikzpicture})(t)  =  -i e^{-it k^2} \int^{\tau}_0 e^{is (-k_1 + k_2 + k_3)^2} e^{i s (k_1^2 - k_2^2 - k_3^2)} ds.
\end{equs}
In the sequel, we will need extra sets of decorated trees that we denote by $ 	\hat{\mcT}^{\leq r,k}_{\text{\tiny{res}},j} $ and $ \hat{\mcT}^{m,k}_{\text{\tiny{res}},j} $. These sets are very similar to the sets $ 	\hat{\mcT}^{\leq r,k}_j $ and $ \hat{\mcT}^{m,k}_j $ except that exactly one  inner node connected to  three leaves does not satisfy the condition \eqref{reso_condition}. For example, one has
\begin{equs}
	\begin{tikzpicture}[scale=0.2,baseline=-5]
		\coordinate (root) at (0,0);
		\coordinate (tri) at (0,-2);
		\coordinate (t1) at (-2,2);
		\coordinate (t2) at (2,2);
		\coordinate (t3) at (0,3);
		\draw[kernels2,tinydots] (t1) -- (root);
		\draw[kernels2] (t2) -- (root);
		\draw[kernels2] (t3) -- (root);
		\draw[symbols] (root) -- (tri);
		\node[not] (rootnode) at (root) {};t
		\node[not,label= {[label distance=-0.2em]below: \scriptsize  $  $}] (trinode) at (tri) {};
		\node[var] (rootnode) at (t1) {\tiny{$ k_{\tiny{1}} $}};
		\node[var] (rootnode) at (t3) {\tiny{$ k_{\tiny{2}} $}};
		\node[var] (trinode) at (t2) {\tiny{$ k_3 $}};
	\end{tikzpicture} \in \hat{\mcT}^{1,k}_0, \quad \begin{tikzpicture}[scale=0.2,baseline=-5]
	\coordinate (root) at (0,0);
	\coordinate (tri) at (0,-2);
	\coordinate (t1) at (-2,2);
	\coordinate (t2) at (2,2);
	\coordinate (t3) at (0,3);
	\draw[kernels2,tinydots] (t1) -- (root);
	\draw[kernels2] (t2) -- (root);
	\draw[kernels2] (t3) -- (root);
	\draw[symbols] (root) -- (tri);
	\node[not] (rootnode) at (root) {};t
	\node[not,label= {[label distance=-0.2em]below: \scriptsize  $  $}] (trinode) at (tri) {};
	\node[var] (rootnode) at (t1) {\tiny{$ k_1 $}};
	\node[var] (rootnode) at (t3) {\tiny{$ k_1 $}};
	\node[var] (trinode) at (t2) {\tiny{$ k $}};
\end{tikzpicture} \in \hat{\mcT}^{1,k}_{\text{\tiny{res}},0}
\end{equs}
with $k=-k_1 + k_2 + k_3$ and $ k _1\neq k_2, k_3$ for the first tree. One has also
\begin{equs}
	\begin{tikzpicture}[scale=0.2,baseline=-5]
		\coordinate (root) at (0,0);
		\coordinate (tri) at (0,-2);
		\coordinate (t1) at (-2,2);
		\coordinate (t2) at (2,2);
		\coordinate (t3) at (0,3);
		\coordinate (t4) at (4,4);
		\coordinate (t41) at (2,6);
		\coordinate (t42) at (6,6);
		\coordinate (t43) at (4,8);
		\coordinate (t4l) at (-4,4);
		\coordinate (t41l) at (-2,6);
		\coordinate (t42l) at (-6,6);
		\coordinate (t43l) at (-4,8);
		\draw[kernels2,tinydots] (t1) -- (root);
		\draw[kernels2] (t2) -- (root);
		\draw[kernels2] (t3) -- (root);
		\draw[symbols] (root) -- (tri);
		\draw[symbols] (t2) -- (t4);
		\draw[kernels2,tinydots] (t4) -- (t41);
		\draw[kernels2] (t4) -- (t42);
		\draw[kernels2] (t4) -- (t43);\draw[symbols,tinydots] (t1) -- (t4l);
		\draw[kernels2,tinydots] (t4l) -- (t41l);
		\draw[kernels2] (t4l) -- (t42l);
		\draw[kernels2,tinydots] (t4l) -- (t43l);
		
		\node[not] (rootnode) at (root) {};
		\node[not] (rootnode) at (t4) {};
		\node[var] (rootnode) at (t3) {\tiny{$ k_{\tiny{4}} $}};
		\node[not,label= {[label distance=-0.2em]below: \scriptsize  $  $}] (trinode) at (tri) {};
		\node[not] (rootnode) at (t1) {};
		\node[var] (rootnode) at (t41) {\tiny{$ k_{\tiny{5}} $}};
		\node[var] (rootnode) at (t42) {\tiny{$ k_{\tiny{7}} $}};
		\node[var] (rootnode) at (t43) {\tiny{$ k_{\tiny{6}} $}};
		\node[var] (rootnode) at (t41l) {\tiny{$ k_{\tiny{3}} $}};
		\node[var] (rootnode) at (t42l) {\tiny{$ k_{\tiny{1}} $}};
		\node[var] (rootnode) at (t43l) {\tiny{$ k_{\tiny{2}} $}};
		\node[not] (trinode) at (t2) {};
	\end{tikzpicture} \in \hat{\mcT}^{1,k}_0, \quad 	\begin{tikzpicture}[scale=0.2,baseline=-5]
	\coordinate (root) at (0,0);
	\coordinate (tri) at (0,-2);
	\coordinate (t1) at (-2,2);
	\coordinate (t2) at (2,2);
	\coordinate (t3) at (0,3);
	\coordinate (t4) at (4,4);
	\coordinate (t41) at (2,6);
	\coordinate (t42) at (6,6);
	\coordinate (t43) at (4,8);
	\coordinate (t4l) at (-4,4);
	\coordinate (t41l) at (-2,6);
	\coordinate (t42l) at (-6,6);
	\coordinate (t43l) at (-4,8);
	\draw[kernels2,tinydots] (t1) -- (root);
	\draw[kernels2] (t2) -- (root);
	\draw[kernels2] (t3) -- (root);
	\draw[symbols] (root) -- (tri);
	\draw[symbols] (t2) -- (t4);
	\draw[kernels2,tinydots] (t4) -- (t41);
	\draw[kernels2] (t4) -- (t42);
	\draw[kernels2] (t4) -- (t43);\draw[symbols,tinydots] (t1) -- (t4l);
	\draw[kernels2,tinydots] (t4l) -- (t41l);
	\draw[kernels2] (t4l) -- (t42l);
	\draw[kernels2,tinydots] (t4l) -- (t43l);
	
	\node[not] (rootnode) at (root) {};
	\node[not] (rootnode) at (t4) {};
	\node[var] (rootnode) at (t3) {\tiny{$ k_{\tiny{4}} $}};
	\node[not,label= {[label distance=-0.2em]below: \scriptsize  $  $}] (trinode) at (tri) {};
	\node[not] (rootnode) at (t1) {};
	\node[var] (rootnode) at (t41) {\tiny{$ k_{\tiny{5}} $}};
	\node[var] (rootnode) at (t42) {\tiny{$ k_{\tiny{7}} $}};
	\node[var] (rootnode) at (t43) {\tiny{$ k_{\tiny{6}} $}};
	\node[var] (rootnode) at (t41l) {\tiny{$ \ell_{\tiny{1}} $}};
	\node[var] (rootnode) at (t42l) {\tiny{$ k_{\tiny{1}} $}};
	\node[var] (rootnode) at (t43l) {\tiny{$ k_{\tiny{1}} $}};
	\node[not] (trinode) at (t2) {};
\end{tikzpicture}, \, 	\begin{tikzpicture}[scale=0.2,baseline=-5]
\coordinate (root) at (0,0);
\coordinate (tri) at (0,-2);
\coordinate (t1) at (-2,2);
\coordinate (t2) at (2,2);
\coordinate (t3) at (0,3);
\coordinate (t4) at (4,4);
\coordinate (t41) at (2,6);
\coordinate (t42) at (6,6);
\coordinate (t43) at (4,8);
\coordinate (t4l) at (-4,4);
\coordinate (t41l) at (-2,6);
\coordinate (t42l) at (-6,6);
\coordinate (t43l) at (-4,8);
\draw[kernels2,tinydots] (t1) -- (root);
\draw[kernels2] (t2) -- (root);
\draw[kernels2] (t3) -- (root);
\draw[symbols] (root) -- (tri);
\draw[symbols] (t2) -- (t4);
\draw[kernels2,tinydots] (t4) -- (t41);
\draw[kernels2] (t4) -- (t42);
\draw[kernels2] (t4) -- (t43);\draw[symbols,tinydots] (t1) -- (t4l);
\draw[kernels2,tinydots] (t4l) -- (t41l);
\draw[kernels2] (t4l) -- (t42l);
\draw[kernels2,tinydots] (t4l) -- (t43l);

\node[not] (rootnode) at (root) {};
\node[not] (rootnode) at (t4) {};
\node[var] (rootnode) at (t3) {\tiny{$ k_{\tiny{4}} $}};
\node[not,label= {[label distance=-0.2em]below: \scriptsize  $  $}] (trinode) at (tri) {};
\node[not] (rootnode) at (t1) {};
\node[var] (rootnode) at (t41) {\tiny{$ k_{\tiny{5}} $}};
\node[var] (rootnode) at (t42) {\tiny{$ \ell_{\tiny{5}} $}};
\node[var] (rootnode) at (t43) {\tiny{$ k_{\tiny{5}} $}};
\node[var] (rootnode) at (t41l) {\tiny{$ k_{\tiny{3}} $}};
\node[var] (rootnode) at (t42l) {\tiny{$ k_{\tiny{1}} $}};
\node[var] (rootnode) at (t43l) {\tiny{$ k_{\tiny{2}} $}};
\node[not] (trinode) at (t2) {};
\end{tikzpicture} \in \hat{\mcT}^{1,k}_{\text{\tiny{res}},0}
	\end{equs}
where $ -\ell_1  = k_1 -k_2 -k_3  $, $ \ell_5  = -k_5 + k_6 + k_7 $, $ k_1\neq k_2, k_3$,  $ k_5\neq k_6, k_7$ and $ \ell_1 \neq k_4, \ell_5$ for the first tree.

The Butcher-Connes-Kreimer coproduct, we will work with is slightly different from the previous section as we cut only certain types of  edges and the edges will be kept in this operation. We first consider the space $\CH_2$ as a subspace of $\CH$ containing forests with planted trees of the form $ \CI_{o_2}(\lambda_k F) $. The Butcher-Connes-Kreimer type coproduct $ \Delta_{\text{\tiny{BCK}}} : \CH_2 \rightarrow \CH_2 \otimes \CH_2 $, a simple version of the one introduced in \cite{BS}, is defined recursively by
\begin{equs} \label{BCK_new}
	\begin{aligned}
	\Delta_{\text{\tiny{BCK}}} \CI_{o_1}( \lambda_{k}  F ) & = \left( \id \, \otimes \CI_{o_1}( \lambda_{k}  \cdot )  \right) \Delta_{\text{\tiny{BCK}}} F,  \\ 
	\Delta_{\text{\tiny{BCK}}} \CI_{o_2}( \lambda_{k}  F ) & = \left( \id \, \otimes \CI_{o_2}( \lambda_{k}  \cdot ) \right) \Delta_{\text{\tiny{BCK}}} F +  \CI_{o_2}( \lambda_{k}  F ) \otimes \one.
	\end{aligned}
\end{equs}
and then extended multiplicatively for the forest product. One has also a similar formula with the admissible cuts where one cuts only edges decorated by $o_2$.
Below, we provide some examples of computations:
\begin{equs}
	\\ & \Delta_{\text{\tiny{BCK}}} \begin{tikzpicture}[scale=0.2,baseline=-5]
		\coordinate (root) at (0,0);
		\coordinate (tri) at (0,-2);
		\coordinate (t1) at (-2,2);
		\coordinate (t2) at (2,2);
		\coordinate (t3) at (0,3);
		\draw[kernels2,tinydots] (t1) -- (root);
		\draw[kernels2] (t2) -- (root);
		\draw[kernels2] (t3) -- (root);
		\draw[symbols] (root) -- (tri);
		\node[not] (rootnode) at (root) {};t
		\node[not,label= {[label distance=-0.2em]below: \scriptsize  $ $}] (trinode) at (tri) {};
		\node[var] (rootnode) at (t1) {\tiny{$ k_{\tiny{1}} $}};
		\node[var] (rootnode) at (t3) {\tiny{$ k_{\tiny{2}} $}};
		\node[var] (trinode) at (t2) {\tiny{$ k_3 $}};
	\end{tikzpicture}  =   \begin{tikzpicture}[scale=0.2,baseline=-5]
		\coordinate (root) at (0,0);
		\coordinate (tri) at (0,-2);
		\coordinate (t1) at (-2,2);
		\coordinate (t2) at (2,2);
		\coordinate (t3) at (0,3);
		\draw[kernels2,tinydots] (t1) -- (root);
		\draw[kernels2] (t2) -- (root);
		\draw[kernels2] (t3) -- (root);
		\draw[symbols] (root) -- (tri);
		\node[not] (rootnode) at (root) {};t
		\node[not,label= {[label distance=-0.2em]below: \scriptsize  $ $}] (trinode) at (tri) {};
		\node[var] (rootnode) at (t1) {\tiny{$ k_{\tiny{1}} $}};
		\node[var] (rootnode) at (t3) {\tiny{$ k_{\tiny{2}} $}};
		\node[var] (trinode) at (t2) {\tiny{$ k_3 $}};
	\end{tikzpicture} \otimes \one + \one \otimes  \begin{tikzpicture}[scale=0.2,baseline=-5]
		\coordinate (root) at (0,0);
		\coordinate (tri) at (0,-2);
		\coordinate (t1) at (-2,2);
		\coordinate (t2) at (2,2);
		\coordinate (t3) at (0,3);
		\draw[kernels2,tinydots] (t1) -- (root);
		\draw[kernels2] (t2) -- (root);
		\draw[kernels2] (t3) -- (root);
		\draw[symbols] (root) -- (tri);
		\node[not] (rootnode) at (root) {};t
		\node[not,label= {[label distance=-0.2em]below: \scriptsize  $ $}] (trinode) at (tri) {};
		\node[var] (rootnode) at (t1) {\tiny{$ k_{\tiny{1}} $}};
		\node[var] (rootnode) at (t3) {\tiny{$ k_{\tiny{2}} $}};
		\node[var] (trinode) at (t2) {\tiny{$ k_3 $}};
	\end{tikzpicture} 
\\
	& \Delta_{\text{\tiny{BCK}}}  \begin{tikzpicture}[scale=0.2,baseline=-5]
		\coordinate (root) at (0,0);
		\coordinate (tri) at (0,-2);
		\coordinate (t1) at (-2,2);
		\coordinate (t2) at (2,2);
		\coordinate (t3) at (0,2);
		\coordinate (t4) at (0,4);
		\coordinate (t41) at (-2,6);
		\coordinate (t42) at (2,6);
		\coordinate (t43) at (0,8);
		\draw[kernels2,tinydots] (t1) -- (root);
		\draw[kernels2] (t2) -- (root);
		\draw[kernels2] (t3) -- (root);
		\draw[symbols] (root) -- (tri);
		\draw[symbols] (t3) -- (t4);
		\draw[kernels2,tinydots] (t4) -- (t41);
		\draw[kernels2] (t4) -- (t42);
		\draw[kernels2] (t4) -- (t43);
		\node[not] (rootnode) at (root) {};
		\node[not] (rootnode) at (t4) {};
		\node[not] (rootnode) at (t3) {};
		\node[not,label= {[label distance=-0.2em]below: \scriptsize  $  $}] (trinode) at (tri) {};
		\node[var] (rootnode) at (t1) {\tiny{$ k_{\tiny{4}} $}};
		\node[var] (rootnode) at (t41) {\tiny{$ k_{\tiny{1}} $}};
		\node[var] (rootnode) at (t42) {\tiny{$ k_{\tiny{3}} $}};
		\node[var] (rootnode) at (t43) {\tiny{$ k_{\tiny{2}} $}};
		\node[var] (trinode) at (t2) {\tiny{$ k_5 $}};
	\end{tikzpicture}  =
	\begin{tikzpicture}[scale=0.2,baseline=-5]
		\coordinate (root) at (0,0);
		\coordinate (tri) at (0,-2) ;
		\coordinate (t1) at (-2,2);
		\coordinate (t2) at (2,2);
		\coordinate (t3) at (0,2);
		\coordinate (t4) at (0,4);
		\coordinate (t41) at (-2,6);
		\coordinate (t42) at (2,6);
		\coordinate (t43) at (0,8);
		\draw[kernels2,tinydots] (t1) -- (root);
		\draw[kernels2] (t2) -- (root);
		\draw[kernels2] (t3) -- (root);
		\draw[symbols] (root) -- (tri);
		\draw[symbols] (t3) -- (t4);
		\draw[kernels2,tinydots] (t4) -- (t41);
		\draw[kernels2] (t4) -- (t42);
		\draw[kernels2] (t4) -- (t43);
		\node[not] (rootnode) at (root) {};
		\node[not] (rootnode) at (t4) {};
		\node[not] (rootnode) at (t3) {};
		\node[not,label= {[label distance=-0.2em]below: \scriptsize  $  $} ] (trinode) at (tri) {};
		\node[var] (rootnode) at (t1) {\tiny{$ k_{\tiny{4}} $}};
		\node[var] (rootnode) at (t41) {\tiny{$ k_{\tiny{1}} $}};
		\node[var] (rootnode) at (t42) {\tiny{$ k_{\tiny{3}} $}};
		\node[var] (rootnode) at (t43) {\tiny{$ k_{\tiny{2}} $}};
		\node[var] (trinode) at (t2) {\tiny{$ k_5 $}};
	\end{tikzpicture} \otimes \one 
	+ \one \otimes \begin{tikzpicture}[scale=0.2,baseline=-5]
		\coordinate (root) at (0,0);
		\coordinate (tri) at (0,-2);
		\coordinate (t1) at (-2,2);
		\coordinate (t2) at (2,2);
		\coordinate (t3) at (0,2);
		\coordinate (t4) at (0,4);
		\coordinate (t41) at (-2,6);
		\coordinate (t42) at (2,6);
		\coordinate (t43) at (0,8);
		\draw[kernels2,tinydots] (t1) -- (root);
		\draw[kernels2] (t2) -- (root);
		\draw[kernels2] (t3) -- (root);
		\draw[symbols] (root) -- (tri);
		\draw[symbols] (t3) -- (t4);
		\draw[kernels2,tinydots] (t4) -- (t41);
		\draw[kernels2] (t4) -- (t42);
		\draw[kernels2] (t4) -- (t43);
		\node[not] (rootnode) at (root) {};
		\node[not] (rootnode) at (t4) {};
		\node[not] (rootnode) at (t3) {};
		\node[not,label= {[label distance=-0.2em]below: \scriptsize  $  $}] (trinode) at (tri) {};
		\node[var] (rootnode) at (t1) {\tiny{$ k_{\tiny{4}} $}};
		\node[var] (rootnode) at (t41) {\tiny{$ k_{\tiny{1}} $}};
		\node[var] (rootnode) at (t42) {\tiny{$ k_{\tiny{3}} $}};
		\node[var] (rootnode) at (t43) {\tiny{$ k_{\tiny{2}} $}};
		\node[var] (trinode) at (t2) {\tiny{$ k_5 $}};
	\end{tikzpicture}  + \begin{tikzpicture}[scale=0.2,baseline=-5]
	\coordinate (root) at (0,0);
	\coordinate (tri) at (0,-2);
	\coordinate (t1) at (-2,2);
	\coordinate (t2) at (2,2);
	\coordinate (t3) at (0,3);
	\draw[kernels2,tinydots] (t1) -- (root);
	\draw[kernels2] (t2) -- (root);
	\draw[kernels2] (t3) -- (root);
	\draw[symbols] (root) -- (tri);
	\node[not] (rootnode) at (root) {};t
	\node[not,label= {[label distance=-0.2em]below: \scriptsize  $ $}] (trinode) at (tri) {};
	\node[var] (rootnode) at (t1) {\tiny{$ k_{\tiny{1}} $}};
	\node[var] (rootnode) at (t3) {\tiny{$ k_{\tiny{2}} $}};
	\node[var] (trinode) at (t2) {\tiny{$ k_3 $}};
\end{tikzpicture} \otimes
	\begin{tikzpicture}[scale=0.2,baseline=-5]
		\coordinate (root) at (0,0);
		\coordinate (tri) at (0,-2);
		\coordinate (t1) at (-2,2);
		\coordinate (t2) at (2,2);
		\coordinate (t3) at (0,3);
		\draw[kernels2,tinydots] (t1) -- (root);
		\draw[kernels2] (t2) -- (root);
		\draw[kernels2] (t3) -- (root);
		\draw[symbols] (root) -- (tri);
		\node[not] (rootnode) at (root) {};t
		\node[not,label= {[label distance=-0.2em]below: \scriptsize  $  $}] (trinode) at (tri) {};
		\node[var] (rootnode) at (t1) {\tiny{$ k_{\tiny{4}} $}};
		\node[var] (rootnode) at (t3) {\tiny{$ \ell_1 $}};
		\node[var] (trinode) at (t2) {\tiny{$ k_5 $}};
	\end{tikzpicture}  
\end{equs}
where $ \ell_1 = -k_1 + k_2 + k_3 $. For the first line, one can see that the decorated tree is primitive because the coproduct is of the form $ \tau \otimes \one + \one \otimes \tau $. The first term corresponds to the cut of the blue edge connected to the root and the second term is the empty cut. The second line of computations is about a decorated tree with two blue edges as these edges belong to the same path to the root, they cannot be cut simultaneously and therefore one gets only trees on both sides of the tensor products. Below, we provide an example where two edges can be cut at the same time producing a forest 
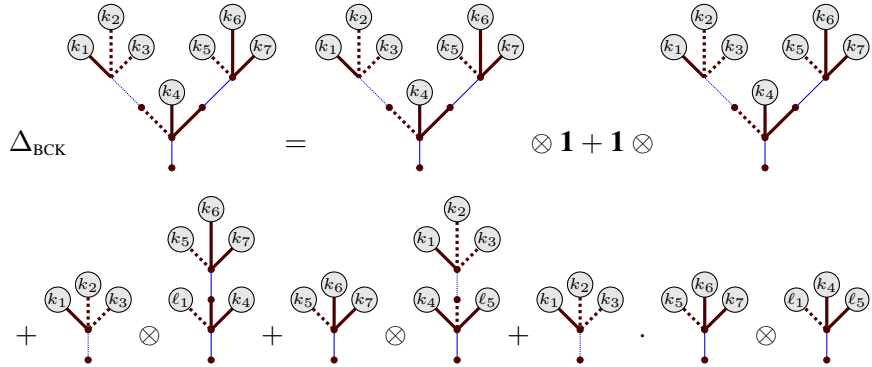
\begin{equs}
\,	 & \Delta_{\text{\tiny{BCK}}}  \begin{tikzpicture}[scale=0.2,baseline=-5]
	\coordinate (root) at (0,0);
	\coordinate (tri) at (0,-2);
	\coordinate (t1) at (-2,2);
	\coordinate (t2) at (2,2);
	\coordinate (t3) at (0,3);
	\coordinate (t4) at (4,4);
	\coordinate (t41) at (2,6);
	\coordinate (t42) at (6,6);
	\coordinate (t43) at (4,8);
	\coordinate (t4l) at (-4,4);
	\coordinate (t41l) at (-2,6);
	\coordinate (t42l) at (-6,6);
	\coordinate (t43l) at (-4,8);
	\draw[kernels2,tinydots] (t1) -- (root);
	\draw[kernels2] (t2) -- (root);
	\draw[kernels2] (t3) -- (root);
	\draw[symbols] (root) -- (tri);
	\draw[symbols] (t2) -- (t4);
	\draw[kernels2,tinydots] (t4) -- (t41);
	\draw[kernels2] (t4) -- (t42);
	\draw[kernels2] (t4) -- (t43);\draw[symbols,tinydots] (t1) -- (t4l);
	\draw[kernels2,tinydots] (t4l) -- (t41l);
	\draw[kernels2] (t4l) -- (t42l);
	\draw[kernels2,tinydots] (t4l) -- (t43l);
	
	\node[not] (rootnode) at (root) {};
	\node[not] (rootnode) at (t4) {};
	\node[var] (rootnode) at (t3) {\tiny{$ k_{\tiny{4}} $}};
	\node[not,label= {[label distance=-0.2em]below: \scriptsize  $  $}] (trinode) at (tri) {};
	\node[not] (rootnode) at (t1) {};
	\node[var] (rootnode) at (t41) {\tiny{$ k_{\tiny{5}} $}};
	\node[var] (rootnode) at (t42) {\tiny{$ k_{\tiny{7}} $}};
	\node[var] (rootnode) at (t43) {\tiny{$ k_{\tiny{6}} $}};
	\node[var] (rootnode) at (t41l) {\tiny{$ k_{\tiny{3}} $}};
	\node[var] (rootnode) at (t42l) {\tiny{$ k_{\tiny{1}} $}};
	\node[var] (rootnode) at (t43l) {\tiny{$ k_{\tiny{2}} $}};
	\node[not] (trinode) at (t2) {};
\end{tikzpicture}  = \begin{tikzpicture}[scale=0.2,baseline=-5]
\coordinate (root) at (0,0);
\coordinate (tri) at (0,-2);
\coordinate (t1) at (-2,2);
\coordinate (t2) at (2,2);
\coordinate (t3) at (0,3);
\coordinate (t4) at (4,4);
\coordinate (t41) at (2,6);
\coordinate (t42) at (6,6);
\coordinate (t43) at (4,8);
\coordinate (t4l) at (-4,4);
\coordinate (t41l) at (-2,6);
\coordinate (t42l) at (-6,6);
\coordinate (t43l) at (-4,8);
\draw[kernels2,tinydots] (t1) -- (root);
\draw[kernels2] (t2) -- (root);
\draw[kernels2] (t3) -- (root);
\draw[symbols] (root) -- (tri);
\draw[symbols] (t2) -- (t4);
\draw[kernels2,tinydots] (t4) -- (t41);
\draw[kernels2] (t4) -- (t42);
\draw[kernels2] (t4) -- (t43);\draw[symbols,tinydots] (t1) -- (t4l);
\draw[kernels2,tinydots] (t4l) -- (t41l);
\draw[kernels2] (t4l) -- (t42l);
\draw[kernels2,tinydots] (t4l) -- (t43l);

\node[not] (rootnode) at (root) {};
\node[not] (rootnode) at (t4) {};
\node[var] (rootnode) at (t3) {\tiny{$ k_{\tiny{4}} $}};
\node[not,label= {[label distance=-0.2em]below: \scriptsize  $  $}] (trinode) at (tri) {};
\node[not] (rootnode) at (t1) {};
\node[var] (rootnode) at (t41) {\tiny{$ k_{\tiny{5}} $}};
\node[var] (rootnode) at (t42) {\tiny{$ k_{\tiny{7}} $}};
\node[var] (rootnode) at (t43) {\tiny{$ k_{\tiny{6}} $}};
\node[var] (rootnode) at (t41l) {\tiny{$ k_{\tiny{3}} $}};
\node[var] (rootnode) at (t42l) {\tiny{$ k_{\tiny{1}} $}};
\node[var] (rootnode) at (t43l) {\tiny{$ k_{\tiny{2}} $}};
\node[not] (trinode) at (t2) {};
\end{tikzpicture} \otimes \one + \one \otimes \begin{tikzpicture}[scale=0.2,baseline=-5]
\coordinate (root) at (0,0);
\coordinate (tri) at (0,-2);
\coordinate (t1) at (-2,2);
\coordinate (t2) at (2,2);
\coordinate (t3) at (0,3);
\coordinate (t4) at (4,4);
\coordinate (t41) at (2,6);
\coordinate (t42) at (6,6);
\coordinate (t43) at (4,8);
\coordinate (t4l) at (-4,4);
\coordinate (t41l) at (-2,6);
\coordinate (t42l) at (-6,6);
\coordinate (t43l) at (-4,8);
\draw[kernels2,tinydots] (t1) -- (root);
\draw[kernels2] (t2) -- (root);
\draw[kernels2] (t3) -- (root);
\draw[symbols] (root) -- (tri);
\draw[symbols] (t2) -- (t4);
\draw[kernels2,tinydots] (t4) -- (t41);
\draw[kernels2] (t4) -- (t42);
\draw[kernels2] (t4) -- (t43);\draw[symbols,tinydots] (t1) -- (t4l);
\draw[kernels2,tinydots] (t4l) -- (t41l);
\draw[kernels2] (t4l) -- (t42l);
\draw[kernels2,tinydots] (t4l) -- (t43l);

\node[not] (rootnode) at (root) {};
\node[not] (rootnode) at (t4) {};
\node[var] (rootnode) at (t3) {\tiny{$ k_{\tiny{4}} $}};
\node[not,label= {[label distance=-0.2em]below: \scriptsize  $  $}] (trinode) at (tri) {};
\node[not] (rootnode) at (t1) {};
\node[var] (rootnode) at (t41) {\tiny{$ k_{\tiny{5}} $}};
\node[var] (rootnode) at (t42) {\tiny{$ k_{\tiny{7}} $}};
\node[var] (rootnode) at (t43) {\tiny{$ k_{\tiny{6}} $}};
\node[var] (rootnode) at (t41l) {\tiny{$ k_{\tiny{3}} $}};
\node[var] (rootnode) at (t42l) {\tiny{$ k_{\tiny{1}} $}};
\node[var] (rootnode) at (t43l) {\tiny{$ k_{\tiny{2}} $}};
\node[not] (trinode) at (t2) {};
\end{tikzpicture} 
\\& +  
\begin{tikzpicture}[scale=0.2,baseline=-5]
	\coordinate (root) at (0,0);
	\coordinate (tri) at (0,-2);
	\coordinate (t1) at (-2,2);
	\coordinate (t2) at (2,2);
	\coordinate (t3) at (0,3);
	\draw[kernels2] (t1) -- (root);
	\draw[kernels2,tinydots] (t2) -- (root);
	\draw[kernels2,tinydots] (t3) -- (root);
	\draw[symbols,tinydots] (root) -- (tri);
	\node[not] (rootnode) at (root) {};t
	\node[not,label= {[label distance=-0.2em]below: \scriptsize  $  $}] (trinode) at (tri) {};
	\node[var] (rootnode) at (t1) {\tiny{$ k_{\tiny{1}} $}};
	\node[var] (rootnode) at (t3) {\tiny{$ k_2 $}};
	\node[var] (trinode) at (t2) {\tiny{$ k_3 $}};
\end{tikzpicture}  \otimes \begin{tikzpicture}[scale=0.2,baseline=-5]
	\coordinate (root) at (0,0);
	\coordinate (tri) at (0,-2);
	\coordinate (t1) at (-2,2);
	\coordinate (t2) at (2,2);
	\coordinate (t3) at (0,2);
	\coordinate (t4) at (0,4);
	\coordinate (t41) at (-2,6);
	\coordinate (t42) at (2,6);
	\coordinate (t43) at (0,8);
	\draw[kernels2,tinydots] (t1) -- (root);
	\draw[kernels2] (t2) -- (root);
	\draw[kernels2] (t3) -- (root);
	\draw[symbols] (root) -- (tri);
	\draw[symbols] (t3) -- (t4);
	\draw[kernels2,tinydots] (t4) -- (t41);
	\draw[kernels2] (t4) -- (t42);
	\draw[kernels2] (t4) -- (t43);
	\node[not] (rootnode) at (root) {};
	\node[not] (rootnode) at (t4) {};
	\node[not] (rootnode) at (t3) {};
	\node[not,label= {[label distance=-0.2em]below: \scriptsize  $  $}] (trinode) at (tri) {};
	\node[var] (rootnode) at (t1) {\tiny{$ \ell_{\tiny{1}} $}};
	\node[var] (rootnode) at (t41) {\tiny{$ k_{\tiny{5}} $}};
	\node[var] (rootnode) at (t42) {\tiny{$ k_{\tiny{7}} $}};
	\node[var] (rootnode) at (t43) {\tiny{$ k_{\tiny{6}} $}};
	\node[var] (trinode) at (t2) {\tiny{$ k_4 $}};
\end{tikzpicture} + \begin{tikzpicture}[scale=0.2,baseline=-5]
\coordinate (root) at (0,0);
\coordinate (tri) at (0,-2);
\coordinate (t1) at (-2,2);
\coordinate (t2) at (2,2);
\coordinate (t3) at (0,3);
\draw[kernels2,tinydots] (t1) -- (root);
\draw[kernels2] (t2) -- (root);
\draw[kernels2] (t3) -- (root);
\draw[symbols] (root) -- (tri);
\node[not] (rootnode) at (root) {};t
\node[not,label= {[label distance=-0.2em]below: \scriptsize  $  $}] (trinode) at (tri) {};
\node[var] (rootnode) at (t1) {\tiny{$ k_{\tiny{5}} $}};
\node[var] (rootnode) at (t3) {\tiny{$ k_6 $}};
\node[var] (trinode) at (t2) {\tiny{$ k_7 $}};
\end{tikzpicture}  \otimes \begin{tikzpicture}[scale=0.2,baseline=-5]
\coordinate (root) at (0,0);
\coordinate (tri) at (0,-2);
\coordinate (t1) at (-2,2);
\coordinate (t2) at (2,2);
\coordinate (t3) at (0,2);
\coordinate (t4) at (0,4);
\coordinate (t41) at (-2,6);
\coordinate (t42) at (2,6);
\coordinate (t43) at (0,8);
\draw[kernels2] (t1) -- (root);
\draw[kernels2] (t2) -- (root);
\draw[kernels2,tinydots] (t3) -- (root);
\draw[symbols] (root) -- (tri);
\draw[symbols,tinydots] (t3) -- (t4);
\draw[kernels2] (t4) -- (t41);
\draw[kernels2,tinydots] (t4) -- (t42);
\draw[kernels2,tinydots] (t4) -- (t43);
\node[not] (rootnode) at (root) {};
\node[not] (rootnode) at (t4) {};
\node[not] (rootnode) at (t3) {};
\node[not,label= {[label distance=-0.2em]below: \scriptsize  $  $}] (trinode) at (tri) {};
\node[var] (rootnode) at (t1) {\tiny{$ k_{\tiny{4}} $}};
\node[var] (rootnode) at (t41) {\tiny{$ k_{\tiny{1}} $}};
\node[var] (rootnode) at (t42) {\tiny{$ k_{\tiny{3}} $}};
\node[var] (rootnode) at (t43) {\tiny{$ k_{\tiny{2}} $}};
\node[var] (trinode) at (t2) {\tiny{$ \ell_5 $}};
\end{tikzpicture} + \begin{tikzpicture}[scale=0.2,baseline=-5]
\coordinate (root) at (0,0);
\coordinate (tri) at (0,-2);
\coordinate (t1) at (-2,2);
\coordinate (t2) at (2,2);
\coordinate (t3) at (0,3);
\draw[kernels2] (t1) -- (root);
\draw[kernels2,tinydots] (t2) -- (root);
\draw[kernels2,tinydots] (t3) -- (root);
\draw[symbols,tinydots] (root) -- (tri);
\node[not] (rootnode) at (root) {};t
\node[not,label= {[label distance=-0.2em]below: \scriptsize  $  $}] (trinode) at (tri) {};
\node[var] (rootnode) at (t1) {\tiny{$ k_{\tiny{1}} $}};
\node[var] (rootnode) at (t3) {\tiny{$ k_2 $}};
\node[var] (trinode) at (t2) {\tiny{$ k_3 $}};
\end{tikzpicture}  \; \cdot \; \begin{tikzpicture}[scale=0.2,baseline=-5]
\coordinate (root) at (0,0);
\coordinate (tri) at (0,-2);
\coordinate (t1) at (-2,2);
\coordinate (t2) at (2,2);
\coordinate (t3) at (0,3);
\draw[kernels2,tinydots] (t1) -- (root);
\draw[kernels2] (t2) -- (root);
\draw[kernels2] (t3) -- (root);
\draw[symbols] (root) -- (tri);
\node[not] (rootnode) at (root) {};t
\node[not,label= {[label distance=-0.2em]below: \scriptsize  $  $}] (trinode) at (tri) {};
\node[var] (rootnode) at (t1) {\tiny{$ k_{\tiny{5}} $}};
\node[var] (rootnode) at (t3) {\tiny{$ k_6 $}};
\node[var] (trinode) at (t2) {\tiny{$ k_7 $}};
\end{tikzpicture} \otimes  \begin{tikzpicture}[scale=0.2,baseline=-5]
\coordinate (root) at (0,0);
\coordinate (tri) at (0,-2);
\coordinate (t1) at (-2,2);
\coordinate (t2) at (2,2);
\coordinate (t3) at (0,3);
\draw[kernels2,tinydots] (t1) -- (root);
\draw[kernels2] (t2) -- (root);
\draw[kernels2] (t3) -- (root);
\draw[symbols] (root) -- (tri);
\node[not] (rootnode) at (root) {};t
\node[not,label= {[label distance=-0.2em]below: \scriptsize  $ $}] (trinode) at (tri) {};
\node[var] (rootnode) at (t1) {\tiny{$ \ell_{\tiny{1}} $}};
\node[var] (rootnode) at (t3) {\tiny{$ k_{\tiny{4}} $}};
\node[var] (trinode) at (t2) {\tiny{$ \ell_5 $}};
\end{tikzpicture}    
\end{equs}
where $ -\ell_1 = k_1 -k_2 -k_3 $ and $\ell_5 = -k_5 + k_6 + k_7 $. The difference with the original Butcher-Connes-Kreiemer coproduct is the cut of the blue edge and one can observe that we kept them as they carry some analytical information. For the grafting product $ \curvearrowright $, one is allowed to graft only on the leaves and the frequency decorations must respect \eqref{frequencies_identity}. For example, one has
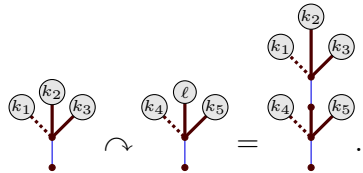
\begin{equs}
\begin{tikzpicture}[scale=0.2,baseline=-5]
	\coordinate (root) at (0,0);
	\coordinate (tri) at (0,-2);
	\coordinate (t1) at (-2,2);
	\coordinate (t2) at (2,2);
	\coordinate (t3) at (0,3);
	\draw[kernels2,tinydots] (t1) -- (root);
	\draw[kernels2] (t2) -- (root);
	\draw[kernels2] (t3) -- (root);
	\draw[symbols] (root) -- (tri);
	\node[not] (rootnode) at (root) {};t
	\node[not,label= {[label distance=-0.2em]below: \scriptsize  $ $}] (trinode) at (tri) {};
	\node[var] (rootnode) at (t1) {\tiny{$ k_{\tiny{1}} $}};
	\node[var] (rootnode) at (t3) {\tiny{$ k_{\tiny{2}} $}};
	\node[var] (trinode) at (t2) {\tiny{$ k_3 $}};
\end{tikzpicture}	\curvearrowright	\begin{tikzpicture}[scale=0.2,baseline=-5]
		\coordinate (root) at (0,0);
		\coordinate (tri) at (0,-2);
		\coordinate (t1) at (-2,2);
		\coordinate (t2) at (2,2);
		\coordinate (t3) at (0,3);
		\draw[kernels2,tinydots] (t1) -- (root);
		\draw[kernels2] (t2) -- (root);
		\draw[kernels2] (t3) -- (root);
		\draw[symbols] (root) -- (tri);
		\node[not] (rootnode) at (root) {};t
		\node[not,label= {[label distance=-0.2em]below: \scriptsize  $  $}] (trinode) at (tri) {};
		\node[var] (rootnode) at (t1) {\tiny{$ k_{\tiny{4}} $}};
		\node[var] (rootnode) at (t3) {\tiny{$ \ell $}};
		\node[var] (trinode) at (t2) {\tiny{$ k_5 $}};
	\end{tikzpicture}  = \begin{tikzpicture}[scale=0.2,baseline=-5]
	\coordinate (root) at (0,0);
	\coordinate (tri) at (0,-2);
	\coordinate (t1) at (-2,2);
	\coordinate (t2) at (2,2);
	\coordinate (t3) at (0,2);
	\coordinate (t4) at (0,4);
	\coordinate (t41) at (-2,6);
	\coordinate (t42) at (2,6);
	\coordinate (t43) at (0,8);
	\draw[kernels2,tinydots] (t1) -- (root);
	\draw[kernels2] (t2) -- (root);
	\draw[kernels2] (t3) -- (root);
	\draw[symbols] (root) -- (tri);
	\draw[symbols] (t3) -- (t4);
	\draw[kernels2,tinydots] (t4) -- (t41);
	\draw[kernels2] (t4) -- (t42);
	\draw[kernels2] (t4) -- (t43);
	\node[not] (rootnode) at (root) {};
	\node[not] (rootnode) at (t4) {};
	\node[not] (rootnode) at (t3) {};
	\node[not,label= {[label distance=-0.2em]below: \scriptsize  $  $}] (trinode) at (tri) {};
	\node[var] (rootnode) at (t1) {\tiny{$ k_{\tiny{4}} $}};
	\node[var] (rootnode) at (t41) {\tiny{$ k_{\tiny{1}} $}};
	\node[var] (rootnode) at (t42) {\tiny{$ k_{\tiny{3}} $}};
	\node[var] (rootnode) at (t43) {\tiny{$ k_{\tiny{2}} $}};
	\node[var] (trinode) at (t2) {\tiny{$ k_5 $}};
\end{tikzpicture}.
\end{equs}
One can notice that there were two other spots available for the grafting decorated by $k_4$ and $k_5$ but the condition \eqref{frequencies_identity} will not have been satisfied and therefore these terms are sent to zero.
The grafting product is defined only for planted decorated trees having a blue edge connected to the root.
 Now, we need to specify the alphabet $A$ for being able to define the arborification map
  $\mathfrak{a}$.
  
  We consider $A$ the alphabet whose letters are given by decorated trees of the form:
  
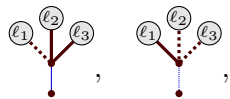
\begin{equs}
  	\begin{tikzpicture}[scale=0.2,baseline=-5]
  		\coordinate (root) at (0,0);
  		\coordinate (tri) at (0,-2);
  		\coordinate (t1) at (-2,2);
  		\coordinate (t2) at (2,2);
  		\coordinate (t3) at (0,3);
  		\draw[kernels2,tinydots] (t1) -- (root);
  		\draw[kernels2] (t2) -- (root);
  		\draw[kernels2] (t3) -- (root);
  		\draw[symbols] (root) -- (tri);
  		\node[not] (rootnode) at (root) {};t
  		\node[not,label= {[label distance=-0.2em]below: \scriptsize  $ $}] (trinode) at (tri) {};
  		\node[var] (rootnode) at (t1) {\tiny{$ \ell_{\tiny{1}} $}};
  		\node[var] (rootnode) at (t3) {\tiny{$ \ell_{\tiny{2}} $}};
  		\node[var] (trinode) at (t2) {\tiny{$ \ell_3 $}};
  	\end{tikzpicture}, \quad \begin{tikzpicture}[scale=0.2,baseline=-5]
  	\coordinate (root) at (0,0);
  	\coordinate (tri) at (0,-2);
  	\coordinate (t1) at (-2,2);
  	\coordinate (t2) at (2,2);
  	\coordinate (t3) at (0,3);
  	\draw[kernels2] (t1) -- (root);
  	\draw[kernels2,tinydots] (t2) -- (root);
  	\draw[kernels2,tinydots] (t3) -- (root);
  	\draw[symbols,tinydots] (root) -- (tri);
  	\node[not] (rootnode) at (root) {};t
  	\node[not,label= {[label distance=-0.2em]below: \scriptsize  $ $}] (trinode) at (tri) {};
  	\node[var] (rootnode) at (t1) {\tiny{$ \ell_{\tiny{1}} $}};
  	\node[var] (rootnode) at (t3) {\tiny{$ \ell_{\tiny{2}} $}};
  	\node[var] (trinode) at (t2) {\tiny{$ \ell_3 $}};
  \end{tikzpicture},
  \end{equs}
where the $\ell_i$ are linear combinations of the $k_i$ with coefficients in $ \lbrace-1, 0, 1 \rbrace $. We also assume that the previous trees are non-resonant satisfying the assumption \eqref{reso_condition}. Then, the arborification map is given  by
\begin{equs} \label{arbo_new}		\mathfrak{a}(  \tau )
	= \mathcal{M}_{\tiny{\text{c}}}\left( \mathfrak{a}  \otimes P_{A}   \right) \Delta_{\text{\tiny{BCK}}}  \tau,
\end{equs}
where $ P_A $ is the projections on the letters of $A$. As an example, one has
\begin{equs}
	\mathfrak{a}(\begin{tikzpicture}[scale=0.2,baseline=-5]
		\coordinate (root) at (0,0);
		\coordinate (tri) at (0,-2);
		\coordinate (t1) at (-2,2);
		\coordinate (t2) at (2,2);
		\coordinate (t3) at (0,3);
		\draw[kernels2,tinydots] (t1) -- (root);
		\draw[kernels2] (t2) -- (root);
		\draw[kernels2] (t3) -- (root);
		\draw[symbols] (root) -- (tri);
		\node[not] (rootnode) at (root) {};t
		\node[not,label= {[label distance=-0.2em]below: \scriptsize  $ $}] (trinode) at (tri) {};
		\node[var] (rootnode) at (t1) {\tiny{$ k_{\tiny{1}} $}};
		\node[var] (rootnode) at (t3) {\tiny{$ k_{\tiny{2}} $}};
		\node[var] (trinode) at (t2) {\tiny{$ k_3 $}};
	\end{tikzpicture}  ) & = \mathcal{M}_{\tiny{\text{c}}}\left( \mathfrak{a}  \otimes P_{A}   \right) \Delta_{\text{\tiny{BCK}}}  \begin{tikzpicture}[scale=0.2,baseline=-5]
	\coordinate (root) at (0,0);
	\coordinate (tri) at (0,-2);
	\coordinate (t1) at (-2,2);
	\coordinate (t2) at (2,2);
	\coordinate (t3) at (0,3);
	\draw[kernels2,tinydots] (t1) -- (root);
	\draw[kernels2] (t2) -- (root);
	\draw[kernels2] (t3) -- (root);
	\draw[symbols] (root) -- (tri);
	\node[not] (rootnode) at (root) {};t
	\node[not,label= {[label distance=-0.2em]below: \scriptsize  $ $}] (trinode) at (tri) {};
	\node[var] (rootnode) at (t1) {\tiny{$ k_{\tiny{1}} $}};
	\node[var] (rootnode) at (t3) {\tiny{$ k_{\tiny{2}} $}};
	\node[var] (trinode) at (t2) {\tiny{$ k_3 $}};
\end{tikzpicture} 
\\
& = \mathcal{M}_{\tiny{\text{c}}} (\mathfrak{a}(\begin{tikzpicture}[scale=0.2,baseline=-5]
	\coordinate (root) at (0,0);
	\coordinate (tri) at (0,-2);
	\coordinate (t1) at (-2,2);
	\coordinate (t2) at (2,2);
	\coordinate (t3) at (0,3);
	\draw[kernels2,tinydots] (t1) -- (root);
	\draw[kernels2] (t2) -- (root);
	\draw[kernels2] (t3) -- (root);
	\draw[symbols] (root) -- (tri);
	\node[not] (rootnode) at (root) {};t
	\node[not,label= {[label distance=-0.2em]below: \scriptsize  $ $}] (trinode) at (tri) {};
	\node[var] (rootnode) at (t1) {\tiny{$ k_{\tiny{1}} $}};
	\node[var] (rootnode) at (t3) {\tiny{$ k_{\tiny{2}} $}};
	\node[var] (trinode) at (t2) {\tiny{$ k_3 $}};
\end{tikzpicture}) \otimes P_A \one )  + \mathcal{M}_{\tiny{\text{c}}} ( \mathfrak{a}(\one) \otimes P_A \begin{tikzpicture}[scale=0.2,baseline=-5]
	\coordinate (root) at (0,0);
	\coordinate (tri) at (0,-2);
	\coordinate (t1) at (-2,2);
	\coordinate (t2) at (2,2);
	\coordinate (t3) at (0,3);
	\draw[kernels2,tinydots] (t1) -- (root);
	\draw[kernels2] (t2) -- (root);
	\draw[kernels2] (t3) -- (root);
	\draw[symbols] (root) -- (tri);
	\node[not] (rootnode) at (root) {};t
	\node[not,label= {[label distance=-0.2em]below: \scriptsize  $ $}] (trinode) at (tri) {};
	\node[var] (rootnode) at (t1) {\tiny{$ k_{\tiny{1}} $}};
	\node[var] (rootnode) at (t3) {\tiny{$ k_{\tiny{2}} $}};
	\node[var] (trinode) at (t2) {\tiny{$ k_3 $}};
\end{tikzpicture} )
=\begin{tikzpicture}[scale=0.2,baseline=-5]
	\coordinate (root) at (0,0);
	\coordinate (tri) at (0,-2);
	\coordinate (t1) at (-2,2);
	\coordinate (t2) at (2,2);
	\coordinate (t3) at (0,3);
	\draw[kernels2,tinydots] (t1) -- (root);
	\draw[kernels2] (t2) -- (root);
	\draw[kernels2] (t3) -- (root);
	\draw[symbols] (root) -- (tri);
	\node[not] (rootnode) at (root) {};t
	\node[not,label= {[label distance=-0.2em]below: \scriptsize  $ $}] (trinode) at (tri) {};
	\node[var] (rootnode) at (t1) {\tiny{$ k_{\tiny{1}} $}};
	\node[var] (rootnode) at (t3) {\tiny{$ k_{\tiny{2}} $}};
	\node[var] (trinode) at (t2) {\tiny{$ k_3 $}};
\end{tikzpicture}
	\\
		\mathfrak{a}(  \begin{tikzpicture}[scale=0.2,baseline=-5]
			\coordinate (root) at (0,0);
			\coordinate (tri) at (0,-2);
			\coordinate (t1) at (-2,2);
			\coordinate (t2) at (2,2);
			\coordinate (t3) at (0,2);
			\coordinate (t4) at (0,4);
			\coordinate (t41) at (-2,6);
			\coordinate (t42) at (2,6);
			\coordinate (t43) at (0,8);
			\draw[kernels2,tinydots] (t1) -- (root);
			\draw[kernels2] (t2) -- (root);
			\draw[kernels2] (t3) -- (root);
			\draw[symbols] (root) -- (tri);
			\draw[symbols] (t3) -- (t4);
			\draw[kernels2,tinydots] (t4) -- (t41);
			\draw[kernels2] (t4) -- (t42);
			\draw[kernels2] (t4) -- (t43);
			\node[not] (rootnode) at (root) {};
			\node[not] (rootnode) at (t4) {};
			\node[not] (rootnode) at (t3) {};
			\node[not,label= {[label distance=-0.2em]below: \scriptsize  $  $}] (trinode) at (tri) {};
			\node[var] (rootnode) at (t1) {\tiny{$ k_{\tiny{4}} $}};
			\node[var] (rootnode) at (t41) {\tiny{$ k_{\tiny{1}} $}};
			\node[var] (rootnode) at (t42) {\tiny{$ k_{\tiny{3}} $}};
			\node[var] (rootnode) at (t43) {\tiny{$ k_{\tiny{2}} $}};
			\node[var] (trinode) at (t2) {\tiny{$ k_5 $}};
		\end{tikzpicture}) & = \mathcal{M}_c\left( \mathfrak{a}( \begin{tikzpicture}[scale=0.2,baseline=-5]
		\coordinate (root) at (0,0);
		\coordinate (tri) at (0,-2);
		\coordinate (t1) at (-2,2);
		\coordinate (t2) at (2,2);
		\coordinate (t3) at (0,3);
		\draw[kernels2,tinydots] (t1) -- (root);
		\draw[kernels2] (t2) -- (root);
		\draw[kernels2] (t3) -- (root);
		\draw[symbols] (root) -- (tri);
		\node[not] (rootnode) at (root) {};t
		\node[not,label= {[label distance=-0.2em]below: \scriptsize  $ $}] (trinode) at (tri) {};
		\node[var] (rootnode) at (t1) {\tiny{$ k_{\tiny{1}} $}};
		\node[var] (rootnode) at (t3) {\tiny{$ k_{\tiny{2}} $}};
		\node[var] (trinode) at (t2) {\tiny{$ k_3 $}};
	\end{tikzpicture}) \otimes P_A \begin{tikzpicture}[scale=0.2,baseline=-5]
	\coordinate (root) at (0,0);
	\coordinate (tri) at (0,-2);
	\coordinate (t1) at (-2,2);
	\coordinate (t2) at (2,2);
	\coordinate (t3) at (0,3);
	\draw[kernels2,tinydots] (t1) -- (root);
	\draw[kernels2] (t2) -- (root);
	\draw[kernels2] (t3) -- (root);
	\draw[symbols] (root) -- (tri);
	\node[not] (rootnode) at (root) {};t
	\node[not,label= {[label distance=-0.2em]below: \scriptsize  $  $}] (trinode) at (tri) {};
	\node[var] (rootnode) at (t1) {\tiny{$ k_{\tiny{4}} $}};
	\node[var] (rootnode) at (t3) {\tiny{$ \ell_1 $}};
	\node[var] (trinode) at (t2) {\tiny{$ k_5 $}};
\end{tikzpicture}    \right) = \begin{tikzpicture}[scale=0.2,baseline=-5]
\coordinate (root) at (0,0);
\coordinate (tri) at (0,-2);
\coordinate (t1) at (-2,2);
\coordinate (t2) at (2,2);
\coordinate (t3) at (0,3);
\draw[kernels2,tinydots] (t1) -- (root);
\draw[kernels2] (t2) -- (root);
\draw[kernels2] (t3) -- (root);
\draw[symbols] (root) -- (tri);
\node[not] (rootnode) at (root) {};t
\node[not,label= {[label distance=-0.2em]below: \scriptsize  $ $}] (trinode) at (tri) {};
\node[var] (rootnode) at (t1) {\tiny{$ k_{\tiny{1}} $}};
\node[var] (rootnode) at (t3) {\tiny{$ k_{\tiny{2}} $}};
\node[var] (trinode) at (t2) {\tiny{$ k_3 $}};
\end{tikzpicture} \; \; \begin{tikzpicture}[scale=0.2,baseline=-5]
\coordinate (root) at (0,0);
\coordinate (tri) at (0,-2);
\coordinate (t1) at (-2,2);
\coordinate (t2) at (2,2);
\coordinate (t3) at (0,3);
\draw[kernels2,tinydots] (t1) -- (root);
\draw[kernels2] (t2) -- (root);
\draw[kernels2] (t3) -- (root);
\draw[symbols] (root) -- (tri);
\node[not] (rootnode) at (root) {};t
\node[not,label= {[label distance=-0.2em]below: \scriptsize  $  $}] (trinode) at (tri) {};
\node[var] (rootnode) at (t1) {\tiny{$ k_{\tiny{4}} $}};
\node[var] (rootnode) at (t3) {\tiny{$ \ell_1 $}};
\node[var] (trinode) at (t2) {\tiny{$ k_5 $}};
\end{tikzpicture} 
\\ \mathfrak{a}(  \begin{tikzpicture}[scale=0.2,baseline=-5]
	\coordinate (root) at (0,0);
	\coordinate (tri) at (0,-2);
	\coordinate (t1) at (-2,2);
	\coordinate (t2) at (2,2);
	\coordinate (t3) at (0,3);
	\coordinate (t4) at (4,4);
	\coordinate (t41) at (2,6);
	\coordinate (t42) at (6,6);
	\coordinate (t43) at (4,8);
		\coordinate (t4l) at (-4,4);
	\coordinate (t41l) at (-2,6);
	\coordinate (t42l) at (-6,6);
	\coordinate (t43l) at (-4,8);
	\draw[kernels2,tinydots] (t1) -- (root);
	\draw[kernels2] (t2) -- (root);
	\draw[kernels2] (t3) -- (root);
	\draw[symbols] (root) -- (tri);
	\draw[symbols] (t2) -- (t4);
	\draw[kernels2,tinydots] (t4) -- (t41);
	\draw[kernels2] (t4) -- (t42);
	\draw[kernels2] (t4) -- (t43);\draw[symbols,tinydots] (t1) -- (t4l);
	\draw[kernels2,tinydots] (t4l) -- (t41l);
	\draw[kernels2] (t4l) -- (t42l);
	\draw[kernels2,tinydots] (t4l) -- (t43l);
	
	\node[not] (rootnode) at (root) {};
	\node[not] (rootnode) at (t4) {};
		\node[var] (rootnode) at (t3) {\tiny{$ k_{\tiny{4}} $}};
	\node[not,label= {[label distance=-0.2em]below: \scriptsize  $  $}] (trinode) at (tri) {};
	\node[not] (rootnode) at (t1) {};
	\node[var] (rootnode) at (t41) {\tiny{$ k_{\tiny{5}} $}};
	\node[var] (rootnode) at (t42) {\tiny{$ k_{\tiny{7}} $}};
	\node[var] (rootnode) at (t43) {\tiny{$ k_{\tiny{6}} $}};
	\node[var] (rootnode) at (t41l) {\tiny{$ k_{\tiny{3}} $}};
	\node[var] (rootnode) at (t42l) {\tiny{$ k_{\tiny{1}} $}};
	\node[var] (rootnode) at (t43l) {\tiny{$ k_{\tiny{2}} $}};
	\node[not] (trinode) at (t2) {};
\end{tikzpicture}) & =\mathcal{M}_c\left( \mathfrak{a}( \begin{tikzpicture}[scale=0.2,baseline=-5]
\coordinate (root) at (0,0);
\coordinate (tri) at (0,-2);
\coordinate (t1) at (-2,2);
\coordinate (t2) at (2,2);
\coordinate (t3) at (0,3);
\draw[kernels2] (t1) -- (root);
\draw[kernels2,tinydots] (t2) -- (root);
\draw[kernels2,tinydots] (t3) -- (root);
\draw[symbols,tinydots] (root) -- (tri);
\node[not] (rootnode) at (root) {};t
\node[not,label= {[label distance=-0.2em]below: \scriptsize  $  $}] (trinode) at (tri) {};
\node[var] (rootnode) at (t1) {\tiny{$ k_{\tiny{1}} $}};
\node[var] (rootnode) at (t3) {\tiny{$ k_2 $}};
\node[var] (trinode) at (t2) {\tiny{$ k_3 $}};
\end{tikzpicture}  \; \cdot \; \begin{tikzpicture}[scale=0.2,baseline=-5]
\coordinate (root) at (0,0);
\coordinate (tri) at (0,-2);
\coordinate (t1) at (-2,2);
\coordinate (t2) at (2,2);
\coordinate (t3) at (0,3);
\draw[kernels2,tinydots] (t1) -- (root);
\draw[kernels2] (t2) -- (root);
\draw[kernels2] (t3) -- (root);
\draw[symbols] (root) -- (tri);
\node[not] (rootnode) at (root) {};t
\node[not,label= {[label distance=-0.2em]below: \scriptsize  $  $}] (trinode) at (tri) {};
\node[var] (rootnode) at (t1) {\tiny{$ k_{\tiny{5}} $}};
\node[var] (rootnode) at (t3) {\tiny{$ k_6 $}};
\node[var] (trinode) at (t2) {\tiny{$ k_7 $}};
\end{tikzpicture}) \otimes P_A \begin{tikzpicture}[scale=0.2,baseline=-5]
\coordinate (root) at (0,0);
\coordinate (tri) at (0,-2);
\coordinate (t1) at (-2,2);
\coordinate (t2) at (2,2);
\coordinate (t3) at (0,3);
\draw[kernels2,tinydots] (t1) -- (root);
\draw[kernels2] (t2) -- (root);
\draw[kernels2] (t3) -- (root);
\draw[symbols] (root) -- (tri);
\node[not] (rootnode) at (root) {};t
\node[not,label= {[label distance=-0.2em]below: \scriptsize  $ $}] (trinode) at (tri) {};
\node[var] (rootnode) at (t1) {\tiny{$ \ell_{\tiny{1}} $}};
\node[var] (rootnode) at (t3) {\tiny{$ k_{\tiny{4}} $}};
\node[var] (trinode) at (t2) {\tiny{$ \ell_5 $}};
\end{tikzpicture}    \right)
\\ &  =\left( \mathfrak{a}( \begin{tikzpicture}[scale=0.2,baseline=-5]
	\coordinate (root) at (0,0);
	\coordinate (tri) at (0,-2);
	\coordinate (t1) at (-2,2);
	\coordinate (t2) at (2,2);
	\coordinate (t3) at (0,3);
	\draw[kernels2] (t1) -- (root);
	\draw[kernels2,tinydots] (t2) -- (root);
	\draw[kernels2,tinydots] (t3) -- (root);
	\draw[symbols,tinydots] (root) -- (tri);
	\node[not] (rootnode) at (root) {};t
	\node[not,label= {[label distance=-0.2em]below: \scriptsize  $  $}] (trinode) at (tri) {};
	\node[var] (rootnode) at (t1) {\tiny{$ k_{\tiny{1}} $}};
	\node[var] (rootnode) at (t3) {\tiny{$ k_2 $}};
	\node[var] (trinode) at (t2) {\tiny{$ k_3 $}};
\end{tikzpicture} ) \shuffle \mathfrak{a} ( \begin{tikzpicture}[scale=0.2,baseline=-5]
	\coordinate (root) at (0,0);
	\coordinate (tri) at (0,-2);
	\coordinate (t1) at (-2,2);
	\coordinate (t2) at (2,2);
	\coordinate (t3) at (0,3);
	\draw[kernels2,tinydots] (t1) -- (root);
	\draw[kernels2] (t2) -- (root);
	\draw[kernels2] (t3) -- (root);
	\draw[symbols] (root) -- (tri);
	\node[not] (rootnode) at (root) {};t
	\node[not,label= {[label distance=-0.2em]below: \scriptsize  $  $}] (trinode) at (tri) {};
	\node[var] (rootnode) at (t1) {\tiny{$ k_{\tiny{5}} $}};
	\node[var] (rootnode) at (t3) {\tiny{$ k_6 $}};
	\node[var] (trinode) at (t2) {\tiny{$ k_7 $}};
\end{tikzpicture}) \right)  \begin{tikzpicture}[scale=0.2,baseline=-5]
	\coordinate (root) at (0,0);
	\coordinate (tri) at (0,-2);
	\coordinate (t1) at (-2,2);
	\coordinate (t2) at (2,2);
	\coordinate (t3) at (0,3);
	\draw[kernels2,tinydots] (t1) -- (root);
	\draw[kernels2] (t2) -- (root);
	\draw[kernels2] (t3) -- (root);
	\draw[symbols] (root) -- (tri);
	\node[not] (rootnode) at (root) {};t
	\node[not,label= {[label distance=-0.2em]below: \scriptsize  $ $}] (trinode) at (tri) {};
	\node[var] (rootnode) at (t1) {\tiny{$ \ell_{\tiny{1}} $}};
	\node[var] (rootnode) at (t3) {\tiny{$ k_{\tiny{4}} $}};
	\node[var] (trinode) at (t2) {\tiny{$ \ell_5 $}};
\end{tikzpicture}
\\  &  = \begin{tikzpicture}[scale=0.2,baseline=-5]
	\coordinate (root) at (0,0);
	\coordinate (tri) at (0,-2);
	\coordinate (t1) at (-2,2);
	\coordinate (t2) at (2,2);
	\coordinate (t3) at (0,3);
	\draw[kernels2] (t1) -- (root);
	\draw[kernels2,tinydots] (t2) -- (root);
	\draw[kernels2,tinydots] (t3) -- (root);
	\draw[symbols,tinydots] (root) -- (tri);
	\node[not] (rootnode) at (root) {};t
	\node[not,label= {[label distance=-0.2em]below: \scriptsize  $  $}] (trinode) at (tri) {};
	\node[var] (rootnode) at (t1) {\tiny{$ k_{\tiny{1}} $}};
	\node[var] (rootnode) at (t3) {\tiny{$ k_2 $}};
	\node[var] (trinode) at (t2) {\tiny{$ k_3 $}};
\end{tikzpicture} \; \; \begin{tikzpicture}[scale=0.2,baseline=-5]
	\coordinate (root) at (0,0);
	\coordinate (tri) at (0,-2);
	\coordinate (t1) at (-2,2);
	\coordinate (t2) at (2,2);
	\coordinate (t3) at (0,3);
	\draw[kernels2,tinydots] (t1) -- (root);
	\draw[kernels2] (t2) -- (root);
	\draw[kernels2] (t3) -- (root);
	\draw[symbols] (root) -- (tri);
	\node[not] (rootnode) at (root) {};t
	\node[not,label= {[label distance=-0.2em]below: \scriptsize  $  $}] (trinode) at (tri) {};
	\node[var] (rootnode) at (t1) {\tiny{$ k_{\tiny{5}} $}};
	\node[var] (rootnode) at (t3) {\tiny{$ k_6 $}};
	\node[var] (trinode) at (t2) {\tiny{$ k_7 $}};
\end{tikzpicture} \; \; \begin{tikzpicture}[scale=0.2,baseline=-5]
	\coordinate (root) at (0,0);
	\coordinate (tri) at (0,-2);
	\coordinate (t1) at (-2,2);
	\coordinate (t2) at (2,2);
	\coordinate (t3) at (0,3);
	\draw[kernels2,tinydots] (t1) -- (root);
	\draw[kernels2] (t2) -- (root);
	\draw[kernels2] (t3) -- (root);
	\draw[symbols] (root) -- (tri);
	\node[not] (rootnode) at (root) {};t
	\node[not,label= {[label distance=-0.2em]below: \scriptsize  $ $}] (trinode) at (tri) {};
	\node[var] (rootnode) at (t1) {\tiny{$ \ell_{\tiny{1}} $}};
	\node[var] (rootnode) at (t3) {\tiny{$ k_{\tiny{4}} $}};
	\node[var] (trinode) at (t2) {\tiny{$ \ell_5 $}};
\end{tikzpicture} +  \begin{tikzpicture}[scale=0.2,baseline=-5]
\coordinate (root) at (0,0);
\coordinate (tri) at (0,-2);
\coordinate (t1) at (-2,2);
\coordinate (t2) at (2,2);
\coordinate (t3) at (0,3);
\draw[kernels2,tinydots] (t1) -- (root);
\draw[kernels2] (t2) -- (root);
\draw[kernels2] (t3) -- (root);
\draw[symbols] (root) -- (tri);
\node[not] (rootnode) at (root) {};t
\node[not,label= {[label distance=-0.2em]below: \scriptsize  $  $}] (trinode) at (tri) {};
\node[var] (rootnode) at (t1) {\tiny{$ k_{\tiny{5}} $}};
\node[var] (rootnode) at (t3) {\tiny{$ k_6 $}};
\node[var] (trinode) at (t2) {\tiny{$ k_7 $}};
\end{tikzpicture} \; \; \begin{tikzpicture}[scale=0.2,baseline=-5]
\coordinate (root) at (0,0);
\coordinate (tri) at (0,-2);
\coordinate (t1) at (-2,2);
\coordinate (t2) at (2,2);
\coordinate (t3) at (0,3);
\draw[kernels2] (t1) -- (root);
\draw[kernels2,tinydots] (t2) -- (root);
\draw[kernels2,tinydots] (t3) -- (root);
\draw[symbols,tinydots] (root) -- (tri);
\node[not] (rootnode) at (root) {};t
\node[not,label= {[label distance=-0.2em]below: \scriptsize  $  $}] (trinode) at (tri) {};
\node[var] (rootnode) at (t1) {\tiny{$ k_{\tiny{1}} $}};
\node[var] (rootnode) at (t3) {\tiny{$ k_2 $}};
\node[var] (trinode) at (t2) {\tiny{$ k_3 $}};
\end{tikzpicture} \; \; \begin{tikzpicture}[scale=0.2,baseline=-5]
\coordinate (root) at (0,0);
\coordinate (tri) at (0,-2);
\coordinate (t1) at (-2,2);
\coordinate (t2) at (2,2);
\coordinate (t3) at (0,3);
\draw[kernels2,tinydots] (t1) -- (root);
\draw[kernels2] (t2) -- (root);
\draw[kernels2] (t3) -- (root);
\draw[symbols] (root) -- (tri);
\node[not] (rootnode) at (root) {};t
\node[not,label= {[label distance=-0.2em]below: \scriptsize  $ $}] (trinode) at (tri) {};
\node[var] (rootnode) at (t1) {\tiny{$ \ell_{\tiny{1}} $}};
\node[var] (rootnode) at (t3) {\tiny{$ k_{\tiny{4}} $}};
\node[var] (trinode) at (t2) {\tiny{$ \ell_5 $}};
\end{tikzpicture}.
\end{equs}
In the sequel, we will consider a map $ \Psi : T(A) \rightarrow \mathcal{C}^{\infty}(\mathbb{R},\mathbb{C}) $  depending on a paremeter $N$, defined 
for every  word $ w = T_k \cdots T_1 $ by
\begin{equs}\label{Phi}
	\Psi\left( w \right)(t) = \one_{ \cap_{m=1}^k A(T_m...T_1)^c}	 \frac{  e^{i \sum_{j=1}^k \mathscr{F}(T_k)t}	}{\prod_{m=1}^{k} \sum_{j=1}^{m} \mathscr{F}(T_j)}
\end{equs}
where
\begin{equation*} 
	A(w) := \big\{ |  \sum_{j=1}^k \mathscr{F}(T_j) | \le N \big\}. 
\end{equation*}
This choice is coming from \cite{GKO13} but other choices are possible depending on which type of estimates, one wants to get. For example in \cite{T17}, $N$ is replaced by $ (2k + 1)K)^{\theta} $ making the bound depending on the number of letters. 
The map $ \Psi$ is not a character in the sense that there exist $w_1, w_2$ such that
\begin{equs}
	\Psi\left( w_1 \shuffle w_2 \right)(t)  \neq 	\Psi\left( w_1  \right)(t)\Psi\left( w_2  \right)(t).
\end{equs}
This is due to the choice of the $A(w)$.
But it is still connected to the iterated integrals described in the previous section. Indeed, we consider
\begin{equs}
	X_{0t}\left( w \right) = \int_{0 < t_1 < \cdots <t_k < t}  dX_{t_1}^{T_1} \cdots dX_{t_k}^{T_k}, \quad 
	X_{t}^{T} = \frac{e^{i \mathscr{F}(T)t}	}{i\mathscr{F}(T)}.
\end{equs}
Then, one has
\begin{equs}
		X_{0t}\left( w \right) = \int_{0 < t_1 < \cdots <t_k < t} e^{i \sum_{j=1}^k t_j \mathscr{F}(T_j)}  dt_1 \cdots dt_k.
\end{equs}
One can notice that the term $ \frac{e^{i \sum_{j=1}^k \mathscr{F}(T_k)t}	}{\prod_{m=1}^{k} \sum_{j=1}^{m} \mathscr{F}(T_j)} $ appears in the previous integral as the leading term: One performs the various integrations in time and removes always the term which is evaluated at zero. In fact, if we remove the indicator function from $ \Psi(T) $, one obtains the following map:
\begin{equs} \label{def_char}
\tilde{\Psi}\left( w \right)(t) = \tilde{\Psi}\left( T_{k-1} \cdots T_1  \right)(t) \frac{e^{i \mathscr{F}(T_k)t}	}{\sum_{j=1}^{k} \mathscr{F}(T_j)}
\end{equs} 
with the non-recursive definition:
\begin{equs}
		\tilde{\Psi}\left( w \right)(t) = 	 \frac{e^{i \sum_{j=1}^k \mathscr{F}(T_k)t}	}{\prod_{m=1}^{k} \sum_{j=1}^{m} \mathscr{F}(T_j)}
\end{equs}
and one has
\begin{proposition} \label{character_prop} The map $	\tilde{\Psi}\left( \cdot \right)(t)$ is a character for the shuffle Hopf algebra in the sense that for every words $ w $ and $ \bar{w} $
	\begin{equs}
		\tilde{\Psi}\left( w \shuffle \bar{w} \right)(t) = \tilde{\Psi}\left( w  \right)(t) \tilde{\Psi}\left(  \bar{w} \right)(t).
	\end{equs}
	\end{proposition}
\begin{proof}
	We proceed by recurrence on $\ell(w) + \ell(\bar{w})$. For $\ell(w) + \ell(\bar{w}) = 0 $
 or $ \ell(w) + \ell(\bar{w}) = 1 $, this is straightforward as one of the word $w_i$ is the empty word $\varepsilon$. When, $\ell(w) + \ell(\bar{w}) =2$, the non-trivial case is when one has $w =T$ and $ w = \bar{T} $, then
\begin{equs}
		\tilde{\Psi}\left( T \shuffle \bar{T} \right)(t) & = \tilde{\Psi}\left( T  \bar{T} \right)(t) + \tilde{\Psi}\left( \bar{T} T \right)(t)
		\\ & =  \frac{e^{i  (\mathscr{F}(T) +\mathscr{F}(\bar{T})) t}	}{\mathscr{F}(\bar{T})(\mathscr{F}(T) +\mathscr{F}(\bar{T}))} + \frac{e^{i  (\mathscr{F}(T) +\mathscr{F}(\bar{T})) t}	}{\mathscr{F}(T)(\mathscr{F}(T) +\mathscr{F}(\bar{T}))} 
		\\ &  = \frac{e^{i  (\mathscr{F}(T) +\mathscr{F}(\bar{T})) t}	}{\mathscr{F}(T) +\mathscr{F}(\bar{T})} = 	\tilde{\Psi}\left( T  \right)(t) 	\tilde{\Psi}\left(  \bar{T} \right)(t).
\end{equs}
We suppose the property true for $\ell(w) + \ell(\bar{w}) \leq r-1$.
One first observes that for two words $w = T_m ...  T_1$ and $\bar{w} = \bar{T}_n ...  \bar{T}_1$, one has
\begin{equs}
	w \shuffle \bar{w} =  T_m (T_{m-1} ...  T_{1} \shuffle \bar{w}) + \bar{T}_n( w\shuffle \bar{T}_{n-1} ...  \bar{T}_{1}) .
\end{equs}
Using this identity, one has
\begin{equs}
&	\tilde{\Psi}\left( w \shuffle \bar{w} \right)(t)
	\\
	 & =  \tilde{\Psi}\left( T_m (T_{m-1} ...  T_{1} \shuffle \bar{w})  \right)(t) + \tilde{\Psi}\left(\bar{T}_n( w\shuffle \bar{T}_{n-1} ...  \bar{T}_{1})   \right)(t)
	 \\ & = \frac{e^{i  \mathscr{F}(T_m)  t} \tilde{\Psi}\left( T_{m-1} ...  T_{1} \shuffle \bar{w}  \right)(t)	}{\sum_{i=1}^m \mathscr{F}(T_i) + \sum_{i=1}^n \mathscr{F}(\bar{T}_i)} + \frac{e^{i  \mathscr{F}(\bar{T}_n) t} \tilde{\Psi}\left(w\shuffle \bar{T}_{n-1} ...  \bar{T}_{1} \right)(t)	}{\sum_{i=1}^m \mathscr{F}(T_i) + \sum_{i=1}^n \mathscr{F}(\bar{T}_i)}.
\end{equs}
By applying the recurrence hypothesis, one gets
\begin{equs}
	\tilde{\Psi}\left( T_{m-1} ...  T_{1} \shuffle \bar{w}  \right)(t) & = \tilde{\Psi}\left( T_{m-1} ...  T_{1}   \right)(t) \Psi\left(  \bar{w}  \right)(t)
	\\  \tilde{\Psi}\left( w\shuffle \bar{T}_{n-1} ...  \bar{T}_{1}  \right)(t) &= \tilde{\Psi}\left( w  \right)(t) \tilde{\Psi}\left( \bar{T}_{n-1} ...  \bar{T}_{1}  \right)(t).
\end{equs}
Together with \eqref{def_char}, this leads to 
\begin{equs}
	& \tilde{\Psi}\left( w \shuffle \bar{w} \right)(t)
	\\
	& =   \frac{e^{i  \mathscr{F}(T_m)  t} \tilde{\Psi}\left( T_{m-1} ...  T_{1}   \right)(t) \tilde{\Psi}\left(  \bar{w}  \right)(t)	}{\sum_{i=1}^m \mathscr{F}(T_i) + \sum_{i=1}^n \mathscr{F}(\bar{T}_i)} + \frac{e^{i  \mathscr{F}(\bar{T}_n) t} \tilde{\Psi}\left( \bar{T}_{n-1} ...  \bar{T}_{1}   \right)(t) \tilde{\Psi}\left(  w  \right)(t)	}{\sum_{i=1}^m \mathscr{F}(T_i) + \sum_{i=1}^n \mathscr{F}(\bar{T}_i)}
	\\ & = \frac{(\sum_{i=1}^m \mathscr{F}(T_i)) \tilde{\Psi}\left( w   \right)(t) \tilde{\Psi}\left(  \bar{w}  \right)(t)	}{\sum_{i=1}^m \mathscr{F}(T_i) + \sum_{i=1}^n \mathscr{F}(\bar{T}_i)} + \frac{ (\sum_{i=1}^n \mathscr{F}(\bar{T}_i))  \tilde{\Psi}\left( w  \right)(t) \tilde{\Psi}\left(  \bar{w}  \right)(t)	}{\sum_{i=1}^m \mathscr{F}(T_i) + \sum_{i=1}^n \mathscr{F}(\bar{T}_i)}
	\\ & = 	\tilde{\Psi}\left( w  \right)(t) 	\tilde{\Psi}\left(  \bar{w} \right)(t).
\end{equs}
	\end{proof}
The previous proof is inspired from \cite[Lemme II.8]{Cr09}. The same character appears also in \cite[Prop. 6]{FM}.

\section{Normal forms}

\label{Sec::4}
We start by recalling the normal form from \cite{GKO13} providing explicit expression for the first terms of this decomposition. Then, we provide formulae for the general formula using the combinatorics introduced in the previous sections. We have decided to work out this expansion for the one dimensional NLS as in the original paper but all the formalism developed in the previous sections is robust and can be applied to other dispersive equations and  systems like the Zakharov system.
 
We recall 
the following cubic nonlinear Schr\"odinger equation (NLS)  on the one dimensional torus $\mathbb{T}$:
\begin{align}
	\begin{cases}
		i \partial_t u + \partial_x^2  u  = | u |^{2} u \\
		u |_{t = 0} = u_0,
	\end{cases}
	\quad (x, t) \in \mathbb{T} \times \mathbb{R}.
	\label{NLS1}
\end{align}
We rewrite the equation using twisting variables:
\begin{equs}
	v(t) =  e^{-it\partial^2_x} {u}(t).
\end{equs}
On the Fourier side, we have
$v_k(t) = e^{i t k^2} u_k(t)$ and
\begin{align}
	\begin{split}
		\partial_t v_k 
		& =  - i 
		\sum_{\substack{k = -k_1 +k_2 + k_3\\ k_2\ne k_1, k_3} }
		e^{ i \Phi(\bar{k})t } 
		\bar{v}_{k_1} v_{k_2} v_{k_3}
		- 2  i 
		\sum_{k_1 \in \mathbb{Z} } 	\bar{v}_{k_1} v_{k_1} v_{k}
	 \\
		& =:   \mathcal{N}^{(1)}(v)(k) +   \mathcal{R}^{(1)}(v)(k) .
	\end{split}
	\label{NLS4}
\end{align}
Here,  the  phase function $\Phi(\bar{n})$ is defined by 
\begin{align}
	\Phi(\bar{k}):& = \Phi(k, k_1, k_2, k_3) = k^2 + k_1^2 - k_2^2- k_3^2 
\end{align}
with $k=-k_1 + k_2 + k_3$. One can make a first connection with the combinatorial formalism developed in the previous sections by making the following observation 
\begin{equs}
	\mathcal{N}^{(1)}(v)(k) = -  \sum_{T_1 \in \hat{\CT}^{1,k}_{0}} i e^{i t \mathscr{F}(T_1) } \frac{\Upsilon[T_1](v)}{S(T_1)}, \quad \mathcal{R}^{(1)}(v)(k) = -  \sum_{\hat{T}_1 \in \hat{\CT}^{1,k}_{\text{\tiny{res}},0}} i \frac{\Upsilon[\hat{T}_1](v)}{S(\hat{T}_1)}.
	\end{equs}
We decompose $\mathcal{N}^{(1)}$ into 
\begin{equation} 
	\mathcal{N}^{(1)} = \mathcal{N}_{1}^{(1)} + \mathcal{N}_{2}^{(1)},
\end{equation}
where $\mathcal{N}_{1}^{(1)}$ is the restriction of $\mathcal{N}^{(1)}$
onto $A(T_1)$.
 Therefore, one has
\begin{equs}
	\mathcal{N}_{2}^{(1)}(v)(k) = -   \sum_{T_1 \in \hat{\CT}^{1,k}_{0}} \one_{A(T_1)^c} \, i e^{i t \mathscr{F}(T_1) } \frac{\Upsilon[T_1](v)}{S(T_1)}.
\end{equs}
 We want to derive a normal form equation:
 \begin{align}
 	\begin{split}
 		v(t) 
 		=   v(0) 
 		&  +    \sum_{j = 2}^\infty  \mathcal{N}_0^{(j)}(v)(t)
 		- \sum_{j = 2}^\infty \mathcal{N}_0^{(j)}(v)(0)  \\
 		& 
 		+ \int_0^t \bigg\{
 		\sum_{j = 1}^\infty \mathcal{N}_{1}^{(j)}(v)(t')   + \sum_{j = 1}^\infty \mathcal{R}^{(j)}(v)(t')\bigg\} dt',
 	\end{split}
 	\label{NFE1}
 \end{align}
where $ \mathcal{N}^{(j)}_0 $  are time-dependent $(2j-1)$-linear operators while 
$ \mathcal{N}^{(j)}_1 $ and $ \mathcal{R}^{(j)} $
are time-dependent $(2j+1)$-linear operators.
One can continue the previous computation by
 \begin{equs} \label{N12}
 	\begin{aligned}
 	\mathcal{N}_{2}^{(1)}(v) (k)  & =  
 	\sum_{ A_1(k)^c} 
 	\partial_t\bigg( \frac{e^{ i \Phi(\bar{k})t } }{\Phi(\bar{k})}\bigg)
 	\bar{v}_{k_1} v_{k_2} v_{k_3} \notag \\
 	& = 
 	\sum_{ A_1(k)^c} 
 	\partial_t \bigg[
 	\frac{e^{ i \Phi(\bar{k})t } }{\Phi(\bar{k})}
 	\bar{v}_{k_1} v_{k_2} v_{k_3}\bigg]  
 	-  \sum_{A_1(k)^c} 
 	\frac{e^{ i \Phi(\bar{k})t } }{\Phi(\bar{k})}
 	\partial_t\big( \bar{v}_{k_1} v_{k_2} v_{k_3}\big) \notag \\
 	& = 
 	\partial_t \bigg[
 	\sum_{ A_1(k)^c} 
 	\frac{e^{ i \Phi(\bar{k})t } }{\Phi(\bar{k})}
 	\bar{v}_{k_1} v_{k_2} v_{k_3}\bigg]  
 	-  \sum_{A_1(k)^c} 
 	\frac{e^{ i \Phi(\bar{k})t } }{\Phi(\bar{k})}
 	\partial_t\big( \bar{v}_{k_1} v_{k_2} v_{k_3}\big) \notag \\
 	& =: \partial_t \mathcal{N}_{0}^{(2)} (v) (k) + \widetilde{\mathcal{N}}^{(2)} (v) (k). 
 	\end{aligned}
 \end{equs}
The key point of the previous computation is the integration by parts that allows to move the time derivative on the $v_{k_i}$. Using the Leibniz rule, one has
\begin{equs}
	\widetilde{\mathcal{N}}^{(2)} (v) (k) = 	-  \sum_{A_1(k)^c} 
	\frac{e^{ i \Phi(\bar{k})t } }{\Phi(\bar{k})} \left( 
	\partial_t \bar{v}_{k_1} v_{k_2} v_{k_3} + \bar{v}_{k_1} \partial_t v_{k_2} v_{k_3} + 
	\bar{v}_{k_1} v_{k_2} \partial_t  v_{k_3}  \right)
\end{equs}
Then, we replace $\partial_t \bar{v}_{k_1}, \partial_t v_{k_2},  \partial_t  v_{k_3}$ using \eqref{NLS4} to get
\begin{equs}
	&	\widetilde{\mathcal{N}}^{(2)} (v) (k) = \mathcal{N}^{(2)}(v)(k) +   \mathcal{R}^{(2)}(v)(k) 
		\\ & -  \sum_{A_1(k)^c} 
		\frac{e^{ i \Phi(\bar{k})t } }{\Phi(\bar{k})} \left( 
		\overline{\mathcal{N}^{(1)}(v)(k_1)} v_{k_2} v_{k_3} + \bar{v}_{k_1} \mathcal{N}^{(1)}(v)(k_2) v_{k_3} + 
		\bar{v}_{k_1} v_{k_2} \mathcal{N}^{(1)}(v)(k_3)  \right)
		\\ &= 
		-  \sum_{A_1(k)^c} 
		\frac{e^{ i \Phi(\bar{k})t } }{\Phi(\bar{k})} \left( 
		\overline{\mathcal{R}^{(1)}(v)(k_1)} v_{k_2} v_{k_3} + \bar{v}_{k_1} \mathcal{R}^{(1)}(v)(k_2) v_{k_3} + 
		\bar{v}_{k_1} v_{k_2} \mathcal{R}^{(1)}(v)(k_3)  \right)
\end{equs}

With the previous computations, one is able to write up the following decomposition:
 \begin{equs}
		v(t) 
		&=   v(0) 
		  +     \mathcal{N}_0^{(2)}(v)(t)
		-  \mathcal{N}_0^{(2)}(v)(0)  
		+ \int_0^t \bigg\{
	 \mathcal{N}_{1}^{(1)}(v)(t')   + \sum_{j = 1}^2 \mathcal{R}^{(j)}(v)(t')\bigg\} dt' \\ & + \int_0^t 
	 \tilde{\mathcal{N}}^{(2)}(v)(t')    dt'.
\end{equs}
Then, one has to continue the expansion by applying the same decomposition to $  \mathcal{N}^{(2)} $ and using the same tricks. We want to provide a proof more conceptual by using combinatorics around decorated trees and the shuffle Hopf algebra. We first see the substitution of the $ \partial_t v_{k_i} $ as a simple composition of two Butcher series that involves the grafting product. This is the subject of the next proposition
\begin{proposition} \label{morphism_grafting} One has
	\begin{equs}
		\partial_t \Upsilon[T](v) = -  \sum_{T_1 \in \hat{\CT}^{1}} i \frac{e^{i \mathscr{F}(T_1) t}}{S(T_1)} \Upsilon[ T_1 \curvearrowright T ](v) -  \sum_{\hat{T}_1 \in \hat{\CT}^{1}_{\text{\tiny{res}}}} i \frac{\Upsilon[ \hat{T}_1 \curvearrowright T ]}{S(\hat{T}_1)} (v)
	\end{equs}
where $ \hat{\CT}^{1} = \hat{\CT}^{1}_0 \cup \hat{\CT}^{1}_1 $ and $ \hat{\CT}^{1}_{\text{\tiny{res}}} = \hat{\CT}^{1}_{\text{\tiny{res}},0} \cup \hat{\CT}^{1}_{\text{\tiny{res}},1} $. 
	\end{proposition}
\begin{proof}
	One can easily check that
	\begin{equs}
		 \Upsilon[T](v) = 2^{|N_{in,T}|} \prod_{l \in L_T} v_{k_{l},\mfp(e_l)}
	\end{equs}
where $ k_{l} = \mfp(e_u) $, $ L_T $ are the leaves of $T$, $ N_{in,T}$ are the inner nodes of $T$ with arity $3$ and $ |N_{in,T}| $ its cardinal. Here $ e_l $ is the edge leading to the leaf $l$ and one has
\begin{equs}
	v_{k_l,0} = 	v_{k_l}, \quad 	v_{k_l,1} = 	\overline{v}_{k_l}.
\end{equs}
Therefore, one has
\begin{equs}
	\partial_t \Upsilon[T](v) & = \partial_t 2^{|N_{in,T}|} \prod_{l \in L_T} v_{k_l,\mfp(e_l)}
	\\ & = \sum_{l \in L_T}  2^{|N_{in,T}|}
	\partial_t v_{k_l,\mfp(e_l)} \prod_{\bar{l} \in L_T \setminus \lbrace l\rbrace} v_{k_{\bar{l}},\mfp(e_{\bar{l}})}.
\end{equs}
Now, we use the fact that
\begin{equs}
		\partial_t v_{k_l,\mfp(e_l)} = - i \sum_{T_1 \in \hat{\CT}^{1,k_l}_{\mfp(e_l)}} e^{i  \mathscr{F}(T_1) t} \frac{\Upsilon[T_1](v)}{S(T_1)} - i \sum_{\hat{T}_1 \in \hat{\CT}^{1,k_l}_{\text{\tiny{res}},\mfp(e_l)}}  \frac{\Upsilon[\hat{T}_1](v)}{S(\hat{T}_1)}.
\end{equs}
Then
\begin{equs}
	\partial_t \Upsilon[T](v) & = - i \sum_{l \in L_T}  \sum_{T_1 \in \hat{\CT}^{1,k_l}_{\mfp(e_l)}} e^{i \mathscr{F}(T_1) t} \frac{\Upsilon[T_1](v)}{S(T_1)} 2^{|N_{in,T}|}
	 \prod_{\bar{l} \in L_T \setminus \lbrace l \rbrace} v_{k_{\bar{l}},\mfp(e_{\bar{l}})}
	 \\ & - i \sum_{l \in L_T}  \sum_{\hat{T}_1 \in \hat{\CT}^{1,k_l}_{\text{\tiny{res}},\mfp(e_l)}}  \frac{\Upsilon[\hat{T}_1](v)}{S(\hat{T}_1)}   2^{|N_{in,T}|} 
	 \prod_{\bar{l} \in L_T \setminus \lbrace l \rbrace} v_{k_{\bar{l}},\mfp(e_{\bar{l}})}.
\end{equs}
One observes that for $ T_1 \in \hat{\CT}^{1,k_l}_{\mfp(e_l)}$
\begin{equs}
	 \Upsilon[T_1](v) 2^{|N_{in,T}|}
	\prod_{\bar{l} \in L_T \setminus \lbrace l \rbrace} v_{k_{\bar{l}},\mfp(e_{\bar{l}})} =
	\Upsilon[ T_1 \curvearrowright_{l} T ] (v)
\end{equs}
where $  \curvearrowright_{l}  $ is the grafting only on the leaf $l$. This product satisfies
\begin{equs}
T_1	\curvearrowright T \, = \sum_{l \in L_T} T_1  \curvearrowright_{l} T.
\end{equs}
 Then, one has
\begin{equs}
\sum_{l \in L_T}  \sum_{T_1 \in \hat{\CT}^{1,k_l}_{\mfp(e_l)}} \Upsilon[ T_1 \curvearrowright_{l} T ] (v) =  \sum_{T_1 \in \hat{\CT}^{1}} \Upsilon[ T_1 \curvearrowright T ] (v).
\end{equs}
This is due to the fact that on the right hand side, one run the sums over a bigger space $ \hat{\CT}^{1} $ which is constrained when one performs the grafting as a specific leaf (most of the terms of the sum are zero). In the end, one obtains the final identity:
\begin{equs}
	\partial_t \Upsilon[T](v) =  - \sum_{T_1 \in \hat{\CT}^{1}} i \frac{e^{i \mathscr{F}(T_1) t}}{S(T_1)} \Upsilon[ T_1 \curvearrowright T ](v) -  \sum_{\hat{T}_1 \in \hat{\CT}^{1}_{\text{\tiny{res}}}} i \frac{\Upsilon[ \hat{T}_1 \curvearrowright T ]}{S(\hat{T}_1)} (v).
\end{equs}
	\end{proof}
\begin{theorem} \label{theorem_dev_N}
The main components of the normal form \eqref{N12} are given in Fourier space by
\begin{equs}
	\mathcal{N}_{2}^{(n)}(v)(k) & = \partial_t
	\mathcal{N}_0^{(n)}(v)(k) + \mathcal{N}^{(n+1)}(v)(k)   
 + \mathcal{R}^{(n+1)}(v)(k).
\end{equs}
where
\begin{equs}
	\mathcal{N}_0^{(n)}(v)(k) &= 	\sum_{T \in \hat{\CT}^{n,k}_{0}} \frac{\Upsilon( T)(v)}{S(T)} \Psi( \mathfrak{a}(T) ),
	\\ \mathcal{R}^{(n)}(v)(k) & = \sum_{\hat{T} \in \hat{\CT}^{n,k}_{\text{\tiny{res}},0}} \frac{\Upsilon( \hat{T})(v)}{S(\hat{T})} \mathscr{F}(\hat{T}) \hat{\Psi}( \mathfrak{a}(\hat{T}) ) (t),
	\\
	\mathcal{N}^{(n)}(v)(k) &= \sum_{T \in \hat{\CT}^{n,k}_{0}} \frac{\Upsilon( T)(v)}{S(T)} \mathscr{F}(T) \hat{\Psi}( \mathfrak{a}(T) ) (t). 
\end{equs}
with
\begin{equs} \label{def_hat_Psi}
		\hat{\Psi}( T_m .... T_1 ) (t) :=  e^{i \mathscr{F}(T_1) t} \frac{ \Psi( T ) (t)}{\sum_{i=1}^m\mathscr{F}(T_i)  } .
\end{equs}
Moreover, one has
\begin{equs}
		\mathcal{N}^{(n)}(v)(k) = 	\mathcal{N}^{(n)}_1(v)(k) + 	\mathcal{N}^{(n)}_2(v)(k)
\end{equs}
where
\begin{equs}
	\mathcal{N}^{(n)}_2(v)(k) = 	 \sum_{T \in \hat{\CT}^{n,k}_{0}} \frac{\Upsilon( T)(v)}{S(T)} i \mathscr{F}(T) \Psi( \mathfrak{a}(T) ) (t).
\end{equs}
\end{theorem}
\begin{proof}
We proceed by recurrence on $n$. We suppose the property true for $ n  \in \mathbb{N}^*$. One has
\begin{equs}
	\mathcal{N}_{2}^{(n)}(v)(k)  &=    \sum_{T \in \hat{\CT}^{n,k}_{0}}  \frac{\Upsilon( T)(v)}{S(T)} i \mathscr{F}(T) \Psi( \mathfrak{a}(T) ) (t)
\\
	& =    \sum_{T \in \hat{\CT} ^{n,k}_{0}} \frac{\Upsilon( T)(v)}{S(T)} \partial_t \Psi( \mathfrak{a}(T) ) (t).
\end{equs}
This is due to the fact that $ \Psi( \mathfrak{a}(T) ) (t) $ depends only on $t$ via the term $ e^{i \mathscr{F}(T) t} $ see \eqref{Phi}. Moreover, $ \mathfrak{a}(T) $ does not change the decorations on the edges of $T$ that are needed for computing $\mathscr{F}(T)$. Using the Leibniz rule, one obtains
\begin{equs}
		\mathcal{N}_{2}^{(n)}(v)(k)	 & =  \partial_t
	\sum_{T \in \hat{\CT}^{n,k}_{0}} \frac{\Upsilon( T)(v)}{S(T)} \Psi( \mathfrak{a}(T) ) (t) - 
	\sum_{T \in \hat{\CT}^{n,k}_{0}} \partial_t \frac{\Upsilon( T)(v)}{S(T)} \Psi( \mathfrak{a}(T) ) (t).
\end{equs}
One has from Proposition~\ref{morphism_grafting}
\begin{equs}
	\partial_t \Upsilon[T](v) = -  \sum_{T_1 \in \hat{\CT}^{1}} i \frac{e^{i \mathscr{F}(T_1) t}}{S(T_1)} \Upsilon[ T_1 \curvearrowright T ](v) -  \sum_{\hat{T}_1 \in \hat{\CT}^{1}_{\text{\tiny{res}}}} i \frac{\Upsilon[ \hat{T}_1 \curvearrowright T ]}{S(\hat{T}_1)} (v).
\end{equs}
Then,
\begin{equs}
&	\sum_{T \in \hat{\CT}^{n,k}_{0}} \partial_t \frac{\Upsilon( T)(v)}{S(T)} \Psi( \mathfrak{a}(T) ) (t)
\\ & = -	i\sum_{T \in \hat{\CT}^{n,k}_{0}}   \sum_{T_1 \in \hat{\CT}^{1}} e^{i \mathscr{F}(T_1) t}\frac{\Upsilon( T_1 \curvearrowright T)(v)}{ S(T_1)S(T)} \Psi( \mathfrak{a}(T) ) (t)
\\ & - i\sum_{T \in \hat{\CT}^{n,k}_{0}}   \sum_{\hat{T}_1 \in \hat{\CT}^{1}_{\text{\tiny{res}}}} \frac{\Upsilon( \hat{T}_1 \curvearrowright T)(v)}{ S(\hat{T}_1)S(T)}\Psi( \mathfrak{a}(T) ) (t).
	\end{equs}
	We compute the first term on the right hand side of the previous equality. The second term can be treated the same. From \eqref{def_hat_Psi}, one has	\begin{equs} \label{rewrite_T_1}
		\hat{\Psi}( T_1\mathfrak{a}(T) ) (t) :=   e^{i \mathscr{F}(T_1) t} \frac{ \Psi( \mathfrak{a}(T) ) (t)}{\mathscr{F}(T_1) + \mathscr{F}(T) } .
	\end{equs}
Then,
\begin{equs} \label{grafting_coeff}
	\frac{T_1}{S(T_1)} \curvearrowright \frac{T}{S(T)} = \sum_{\bar{T} \in \hat{\CT}^{n+1,k}_{0}} \frac{m(T_1,T,\bar{T})}{S(\bar{T})} \bar{T}
\end{equs}
for some coefficients $ m(T_1,T,\bar{T}) $.
 One has using the innner product on decorated trees:
\begin{equs}
	\big\langle 	\frac{T_1}{S(T_1)} \curvearrowright \frac{T}{S(T)} , \bar{T}  \big\rangle = \big\langle \frac{T_1}{S(T_1)} \otimes \frac{T}{S(T)} , \curvearrowright^{*} \bar{T}  \big\rangle = 	m(T_1,T,\bar{T}).  
\end{equs} 
A similar argument has been used for the composition of Regularity Structures B-series in \cite{B23}.
This yields by first using \eqref{rewrite_T_1} and \eqref{grafting_coeff}
	\begin{equs}
		& - i	\sum_{T \in \hat{\CT}^{n,k}_{0}}   \sum_{T_1 \in \hat{\CT}^{1}} e^{i \mathscr{F}(T_1) t}\frac{\Upsilon( T_1 \curvearrowright T)(v)}{ S(T_1)S(T)} \Psi( \mathfrak{a}(T) ) (t) \\
	& =	- i \sum_{T \in \hat{\CT}^{n,k}_{0}}   \sum_{T_1 \in \hat{\CT}^{1}} \frac{\Upsilon( T_1 \curvearrowright T)(v)}{S(T_1)S(T)} \left(  \mathscr{F}(T_1) + \mathscr{F}(T) \right) \hat{\Psi}( T_1\mathfrak{a}(T) ) (t)
	\\ & = - i\sum_{\bar{T} \in \hat{\CT}^{n+1,k}_{0}} \sum_{T \in \hat{\CT}^{n,k}_{0}}   \sum_{T_1 \in \hat{\CT}^{1}} \frac{\Upsilon( \bar{T})(v)}{S(\bar{T})} m(T_1,T,\bar{T}) \  \mathscr{F}(\bar{T})   \hat{\Psi}( T_1\mathfrak{a}(T) ) (t) 
\end{equs}
where we have used the fact that 
\begin{equs}
	 \mathscr{F}(T_1) + \mathscr{F}(T)  = \mathscr{F}(\bar{T}) 
\end{equs}
when $ m(T_1,T,\bar{T}) \neq 0 $. We continue the computation by using
\begin{equs}
\sum_{T \in \hat{\CT}^{n,k}_{0}}   \sum_{T_1 \in \hat{\CT}^{1}} m(T_1,T,\bar{T})  T_1\mathfrak{a}(T) =  \mathcal{M}_c\left(   P_{\bullet} \otimes \mathfrak{a}   \right) \curvearrowright^{*} \bar{T}.
\end{equs} 
which gives
\begin{equs} 
		& - i	\sum_{T \in \hat{\CT}^{n,k}_{0}}   \sum_{T_1 \in \hat{\CT}^{1}} e^{i \mathscr{F}(T_1) t}\frac{\Upsilon( T_1 \curvearrowright T)(v)}{ S(T_1)S(T)} \Psi( \mathfrak{a}(T) ) (t) \\
	& = -i	\sum_{\bar{T} \in \hat{\CT}^{n+1,k}_{0}}    \frac{\Upsilon( \bar{T})(v)}{S(\bar{T})} \mathscr{F}(\bar{T})  \hat{\Psi}(  \mathcal{M}_{\tiny{\text{c}}}\left(   P_{\bullet} \otimes \mathfrak{a}   \right) \curvearrowright^{*} \bar{T}) (t)
		\\ 
	& = -i	\sum_{\bar{T} \in \hat{\CT}^{n+1,k}_{0}}    \frac{\Upsilon( \bar{T})(v)}{S(\bar{T})} \mathscr{F}(\bar{T}) \hat{\Psi}( \mathfrak{a}(\bar{T})) (t)
\end{equs}
where we have used Proposition \ref{alternative_arb}
\begin{equs}
	\mathcal{M}_{\tiny{\text{c}}}\left(   P_{\bullet} \otimes \mathfrak{a}   \right) \curvearrowright^{*}  = \mathfrak{a}.
\end{equs}
A similar computation allows us to get
\begin{equs}
-i	\sum_{T \in \hat{\CT}^{n,k}_{0}}   \sum_{\hat{T}_1 \in \hat{\CT}^{1}_{\text{\tiny{res}}}} \frac{\Upsilon( \hat{T}_1 \curvearrowright T)(v)}{ S(\hat{T}_1)S(T)}\Psi( \mathfrak{a}(T) ) (t) = - i	\sum_{\bar{T} \in \hat{\CT}^{n+1,k}_{\text{\tiny{res}},0}}    \frac{\Upsilon( \bar{T})(v)}{S(\bar{T})} \mathscr{F}(\bar{T}) \hat{\Psi}( \mathfrak{a}(\bar{T})) (t).
\end{equs}
In the end, one has the following identity
\begin{equs}
	\mathcal{N}_{2}^{(n)}(v)(k) & = \partial_t
\sum_{T \in \hat{\CT}^{n,k}_{0}} \frac{\Upsilon( T)(v)}{S(T)} \Psi( \mathfrak{a}(T) ) (t) +    \sum_{T \in \hat{\CT}^{n+1,k}_{0}} \frac{\Upsilon( T)(v)}{S(T)} i \mathscr{F}(T) \hat{\Psi}( \mathfrak{a}(T) ) (t)
\\ &  + \sum_{\hat{T} \in \hat{\CT}^{n+1,k}_{\text{\tiny{res}},0}} \frac{\Upsilon( \hat{T})(v)}{S(\hat{T})} i \mathscr{F}(\hat{T}) \hat{\Psi}( \mathfrak{a}(\hat{T}) ) (t).
\end{equs}
Now, we can set 
\begin{equs}
	\mathcal{N}_0^{(n)}(v)(k) &= 	\sum_{T \in \hat{\CT}^{n,k}_{0}} \frac{\Upsilon( T)(v)}{S(T)} \Psi( \mathfrak{a}(T) )(t), 
	\\ \mathcal{R}^{(n+1)}(v)(k) & = \sum_{\hat{T} \in \hat{\CT}^{n+1,k}_{\text{\tiny{res}},0}} \frac{\Upsilon( \hat{T})(v)}{S(\hat{T})} i \mathscr{F}(\hat{T}) \hat{\Psi}( \mathfrak{a}(\hat{T}) ) (t),
	\\
	 \mathcal{N}^{(n+1)}(v)(k) &= \sum_{T \in \hat{\CT}^{n+1,k}_{0}} \frac{\Upsilon( T)(v)}{S(T)} i \mathscr{F}(T) \hat{\Psi}( \mathfrak{a}(T) ) (t). 
\end{equs}
We finish by decomposing $ \mathcal{N}^{(n+1)}(v)(k)$ into
\begin{equs}
	\mathcal{N}^{(n+1)}(v)(k) = 	\mathcal{N}^{(n+1)}_1(v)(k) + \mathcal{N}^{(n+1)}_2(v)(k)
\end{equs}
where 
\begin{equs}
	\mathcal{N}^{(n+1)}_2(v)(k) &= \sum_{T \in \hat{\CT}^{n+1,k}_{0}} \frac{\Upsilon( T)(v)}{S(T)} i \mathscr{F}(T) \Psi( \mathfrak{a}(T) ) (t)
\end{equs}
which concludes the proof.
\end{proof}

\end{document}